\documentclass{amsart}
\usepackage[utf8]{inputenc}
\usepackage[T1]{fontenc}
\usepackage{lmodern}
\usepackage{amsmath}
\usepackage{amssymb}
\usepackage{mathtools}
\usepackage{latexsym}
\usepackage[lite]{amsrefs}
\usepackage{nicefrac}
\usepackage{microtype}
\usepackage{color}
\usepackage{xcolor}
\usepackage[normalem]{ulem}
\usepackage{tikz-cd}
\usepackage{MnSymbol}
\usepackage{enumitem} % advanced enumerate, which handles label and ref
\setlist[enumerate,1]{label=\textup{(\arabic*)}}% ensure enumerates in theorems are upright
\DeclareMathAlphabet{\mathpzc}{OT1}{pzc}{m}{it}
\usepackage[all]{xy}
\newdir{ >}{{}*!/-5pt/@{>}}

\usepackage[pdftitle={Analytic Hochschild-Kostant-Rosenberg Theorem},
pdfauthor={Devarshi Mukherjee},
pdfauthor={Jack Kelly},
pdfauthor={Kobi Kremnizer}
pdfsubject={Mathematics}
]{hyperref}

\usepackage{scalerel,stackengine}
\stackMath

\newcommand\reallywidehat[1]{%
\savestack{\tmpbox}{\stretchto{%
  \scaleto{%
    \scalerel*[\widthof{\ensuremath{#1}}]{\kern-.6pt\bigwedge\kern-.6pt}%
    {\rule[-\textheight/2]{1ex}{\textheight}}%WIDTH-LIMITED BIG WEDGE
  }{\textheight}% 
}{0.5ex}}%
\stackon[1pt]{#1}{\tmpbox}%
}
\BibSpec{book}{%
  +{}  {\PrintPrimary}                {transition}
  +{,} { \textit}                     {title}
  +{.} { }                            {part}
  +{:} { \textit}                     {subtitle}
  +{,} { \PrintEdition}               {edition}
  +{}  { \PrintEditorsB}              {editor}
  +{,} { \PrintTranslatorsC}          {translator}
  +{,} { \PrintContributions}         {contribution}
  +{,} { }                            {series}
  +{,} { \voltext}                    {volume}
  +{,} { }                            {publisher}
  +{,} { }                            {organization}
  +{,} { }                            {address}
  +{,} { \PrintDateB}                 {date}
  +{,} { }                            {status}
  +{}  { \parenthesize}               {language}
  +{}  { \PrintTranslation}           {translation}
  +{;} { \PrintReprint}               {reprint}
  +{.} { }                            {note}
  +{.} {}                             {transition}
  +{} { \PrintDOI}                   {doi}
  +{} { available at \url}            {eprint}
  +{}  {\SentenceSpace \PrintReviews} {review}
}

\renewcommand*{\PrintDOI}[1]{\href{http://dx.doi.org/\detokenize{#1}}{doi: \detokenize{#1}}}
\renewcommand*{\MR}[1]{ \href{http://www.ams.org/mathscinet-getitem?mr=#1}{MR \textbf{#1}}}
\newcommand{\adj}[4]{#1\negmedspace: #2\rightleftarrows #3:\negmedspace #4}

\newcommand{\comment}[1]{}  %to comment out chunks of text

\theoremstyle{plain}
\newtheorem{theorem}{Theorem}[section]
\newtheorem{lemma}[theorem]{Lemma}
\newtheorem{corollary}[theorem]{Corollary}
\newtheorem{proposition}[theorem]{Proposition}
\theoremstyle{remark}
\newtheorem{remark}[theorem]{Remark}
\theoremstyle{definition}
\newtheorem{definition}[theorem]{Definition}
\newtheorem{example}[theorem]{Example}

\numberwithin{theorem}{section}

\newcommand\C{\mathbb C}
\newcommand\N{\mathbb N}
\newcommand\Q{\mathbb Q}
\newcommand\R{\mathbb R}
\newcommand\Z{\mathbb Z}

\newcommand{\coma}{\widehat}% pi-adic completion
\newcommand*{\tens}{\mathsf{T}}% tensor algebra
\newcommand{\defeq}{\mathrel{:=}} % per definition
\newcommand{\heart}{\ensuremath\heartsuit} % heart
\newcommand{\stkout}[1]{\ifmmode\text{\sout{\ensuremath{#1}}}\else\sout{#1}\fi}

\newcommand*{\onto}{\twoheadrightarrow}

\DeclarePairedDelimiter{\abs}{\lvert}{\rvert}% absolute value
\DeclarePairedDelimiter{\norm}{\lVert}{\rVert}% norm
\DeclarePairedDelimiterX{\setgiven}[2]{\{}{\}}{#1\,{:}\,\mathopen{}#2}% set given by

\newcommand\haotimes{\mathbin{\coma{\otimes}}}% adically complete tp

\newcommand\fC{\mathsf{C}} % category

\DeclareMathOperator{\coker}{coker}

\DeclareMathOperator*{\colim}{colim}
\DeclareMathOperator{\Mod}{Mod}

\DeclareMathOperator{\Fun}{Fun}

\newcommand{\Gr}{\mathbf{Gr}}
\newcommand{\fil}{\mathrm{fil}}

\newcommand{\op}{\mathrm{op}}% opposite algebra
\newcommand{\ev}{\mathrm{ev}}% evaluation homomorphism
\newcommand{\resf}{\mathbb F}% residue field of \dvr
\DeclareMathOperator{\Hom}{Hom}% space of linear maps
\DeclareMathOperator{\HH}{HH}% Hochschild homology

\begin{document}
\title{Analytic Hochschild-Kostant-Rosenberg Theorem}

\author{Jack Kelly}
\email{jack.kelly@tcd.ie}
\address{The Hamilton Mathematics Institute, School of Mathematics, Trinity College Dublin, Dublin 2, Ireland
}

\author{Kobi Kremnizer}
\email{kobi.kremnitzer@oriel.ox.ac.uk}
\address{Mathematical Institute\\
 University of Oxford\\ 
 Andrew Wiles Building\\ 
 Radcliffe Observatory
Quarter\\ Woodstock Road, Oxford, OX2 6GG\\ England
}

\author{Devarshi Mukherjee}
\email{devarshi.mukherjee@mathematik.uni-goettingen.de}

\address{Mathematisches Institut\\
  Georg-August Universit\"at Göttingen\\
  Bunsenstra\ss{}e 3--5\\
  37073 Göttingen\\
  Germany}
  
  \dedicatory{}
\thanks{ The first named author acknowledges the support of the Simons Foundation through its `Targeted Grants to Institutes' program.}

\begin{abstract}
Let \(R\) be a Banach ring. We prove that the category of chain complexes of complete bornological \(R\)-modules (and several related categories) is a derived algebraic context in the sense of Raksit. We then use the framework of derived algebra to prove a version of the Hochschild-Kostant-Rosenberg Theorem, which relates the circle action on the Hochschild algebra to the de Rham-differential-enriched-de Rham algebra of a simplicial, commutative, complete bornological algebra. This has a geometric interpretation in the language of derived analytic geometry, namely, the derived loop stack of a derived analytic stack is equivalent to the shifted tangent stack. Using this geometric interpretation we extend our results to derived schemes.
\end{abstract}

\maketitle
\tableofcontents

\section{Introduction}

The classical Hochschild-Kostant-Rosenberg Theorem relates the Hochschild homology of a commutative algebra over a base ring with algebraic differential forms on the algebra. More precisely, if \(A\) is a smooth, commutative algebra over a commutative, unital ring \(R\), then there is an isomorphism of \(A\)-modules \(\Omega_{A/R}^* \cong \HH_*(A/R)\), between the \(A\)-module of K\"ahler differentials, and the Hochschild cohomology groups of \(A\) with coefficients in \(R\), for all \(* \geq 0\). There is, however, important information that is lost in this mere isomorphism of modules. The Hochschild complex \(\mathsf{HH}(A/R)\) is a simplicial commutative \(R\)-algebra, equipped with an \(S^1\)-action, while the de Rham complex \(DR(A/R)\) with its de Rham differential is a commutative differential graded \(R\)-algebra. Furthermore, the \(S^1\)-action enriched Hochschild algebra and the de Rham complex with its usual differential, are the initial objects with these properties; the former in the \(\infty\)-category \(S^1\text{-}\mathsf{sCAlg}_R\) of simplicial \(S^1\)-equivariant commutative \(R\)-algebras, and the latter in the \(\infty\)-category \(\epsilon\text{-}\mathsf{cdga}_R\) of mixed complexes. When \(R\) has characteristic zero, the agreement between the enriched Hochschild and de Rham algebras follows from a symmetric, monoidal equivalence of these two \(\infty\)-categories, due to To\"en-Vezzosi (see \cite{toen2011algebres}). In the integral setting, this \(\infty\)-categorical equivalence is no longer true and the analogous identification of Hochschild and de Rham theory is far more complicated. This is the subject of recent work by Raksit \cite{raksit2020hochschild}, who uses derived algebra to construct a completely formal proof over \(\Z\). A similar result that uses techniques from derived algebraic geometry in the mixed characteristic case is due to Moulinos-Rubalo-To\"en ~\cite{toen-robalo}. 
 
In this article, we propose a version of the Hochschild-Kostant-Rosenberg Theorem that identifies Hochschild and de Rham theory, when the base object is a \textit{Banach ring}. Our approach uses the framework of homotopical algebra in the context of exact categories \cite{kelly2018exact}, and the series of articles \cites{bambozzi2016dagger,BaBK,bambozzi2019analytic,ben2020fr} which build the foundations of derived analytic geometry. These articles show that for a Banach ring \(R\), the categories \(\mathsf{CBorn}_R\) and \(\mathsf{Ind}(\mathsf{Ban}_R)\) of complete bornological \(R\)-modules and inductive systems of Banach \(R\)-modules, are formally very similar to the category \(\mathsf{Mod}_\Z\) of abelian groups. They are closed, symmetric monoidal (with the completed projective tensor product) quasi-abelian categories. These properties enable us to do homotopy theory in functional analytic settings. In particular, we show that the derived \(\infty\)-category of \(\mathsf{CBorn}_R\) is a \textit{derived algebraic context} in the sense of Raksit, facilitating a purely formal HKR identification in the generality of derived commutative complete bornological \(R\)-algebras. 

The second theme that we investigate is the geometric meaning of our result in the language of derived analytic geometry. Let us first recall the analogous interpretation in derived algebraic geometry. Suppose \(A\) is a finite type \(R\)-algebra, and \(X = \mathsf{Spec}(A)\). Then the simplicial commutative \(R\)-algebra \(S^1 \otimes A\) is a model for the simplicial ring underlying the Hochschild complex \(\mathsf{HH}(A/R)\). This coincides with the \textit{derived loop space} \(\mathcal{L}(X)\), defined as the derived mapping space of functions \(S^1 \to X\). On the de Rham side, there is an identification between the shifted tangent space \(\mathsf{T}(X)[1]\) of \(X\) and the (derived) space of maps \(R \oplus R[-1] \to X\). Therefore, the HKR Theorem implies an equivalence  \[\mathcal{L}X = \underline{\mathsf{Map}}(S^1, X) \cong \underline{\mathsf{Map}}(R \oplus R[-1], X) = \tens(X)[1]\] between the derived loop space and the shifted tangent space. We prove the same equivalence holds when we replace a scheme by a sufficiently nice derived \textit{analytic stack}. Using a different formalism for derived analytic geometry (namely the formalism of derived analytic spaces of Porta and Yue Yu  \cites{porta2016higher,porta2015derived1, porta2015derived, porta2017derived,porta2018derived, porta2017representability,porta2018derivedhom}), Antonio, Petit, and Porta (\cite{Antoniothesis} Theorem 6.5.0.1) have also proven an analytic HKR theorem for derived analytic spaces. Our result should recover their theorem for derived analytic spaces which are $n$-localic for some $n$, and a minor modification of our result will recover their HKR theorem for all derived analytic spaces. We shall in fact prove a geometric HKR theorem for geometry relative to any derived algebraic context, and thereby also recover the algebraic HKR theorem of Ben-Zvi-Nadler (\cite{ben-zvinadler} Proposition 1.1).

This paper is organised as follows.

In Section \ref{sec:bornologies}, we recall preliminary and background material on \textit{bornological} structures arising from \textit{Banach rings}. A Banach ring is a complete, normed ring which includes fields such as \(\C\), \(\R\), and \(\Q_p\) with their Euclidean and \(p\)-adic topologies; the ring \(\Z\) with the Euclidean topology, the \(p\)-adic integers \(\Z_p\), and so on. For any such Banach ring \(R\), we can talk about the category \(\mathsf{Ban}_R\) of Banach \(R\)-modules. This is a closed symmetric monoidal category with the completed projective tensor product, which has finite limits and colimits. Taking the free cocompletion by filtered colimits, we obtain the category \(\mathsf{Ind}(\mathsf{Ban}_R)\) of inductive systems of Banach \(R\)-modules. This is a cocomplete and complete category. Furthermore, taking the full subcategory of inductive systems with monomorphic structure maps yields our second category of interest - the (concrete) category \(\mathsf{CBorn}_R\) of complete bornological \(R\)-modules. This category has a concrete description in terms of formally bounded subsets of an \(R\)-module whenever \(R\) is a non-trivially valued Banach field. In this case, it is a closed symmetric monoidal category with all limits and colimits.

Section \ref{sec:DACs} recalls some background material on higher category theory from \cites{raksit2020hochschild,L} that is relevant for our purposes. Since the HKR result we prove is in the generality of Banach rings (that are not necessarily fields), the Hochschild and de Rham algebras are most conveniently viewed as filtered and graded objects in the derived \(\infty\)-category of complexes over \(\mathsf{CBorn}_R\) (or \(\mathsf{Ind}(\mathsf{Ban}_R)\)), respectively. The canonical filtration obtained by applying the Postnikov filtration functor on the Hochschild complex is called the \textit{HKR-filtration}. To incorporate the \(S^1\)-action on the Hochschild complex, we use the notion of \textit{filtered circle} \(\mathbb{T}_{\mathrm{fil}}\), introduced in \cite{raksit2020hochschild}. This is the Postnikov filtration functor applied to the group algebra \(R[S^1]\). In order that the interaction between the filterations on \(\mathbb{T}\) and the Hochschild complex goes through conveniently, we work in the larger \(\infty\)-category, which includes non-connective objects. These are called derived commutative \(R\)-algebra objects in a derived commutative algebra context. Our main result in Section \ref{sec:exact} is that the derived \(\infty\)-category of the quasi-abelian category \(\mathsf{Ind}(\mathsf{Ban}_R)\) is a derived algebraic context. 

In Section \ref{sec:HKR}, we recall from \cite{raksit2020hochschild}, the definitions of the \(S^1\)-equivariant Hochschild complex and the Hodge-completed derived de Rham cohomology for derived commutative algebras. Specialising to the derived algebraic context \(\mathsf{Ind}(\mathsf{Ban}_R)\), we call the latter the \textit{bornological derived de Rham cohomology}, which will play the same role in derived analytic geometry, as the derived de Rham cohomology in derived algebraic geometry. The main result in \cite{raksit2020hochschild} applied to our setting then implies the following version of the HKR Theorem:

\begin{theorem}
Let \(A\) be a derived commutative algebra object in the derived algebraic context \(\mathsf{Ind}(\mathsf{Ban}_R)\) or \(\mathsf{CBorn}_R\). Then for any derived commutative \(A\)-algebra, the associated graded of the Hochschild algebra \(\mathsf{HH}(B/A)\) is equivalent to the bornological derived de Rham cohomology \(\mathsf{L}_{B/A}\). 
\end{theorem}

Sections \ref{sec:categorical_geometry}, \ref{sec:monoidal_geometry} and \ref{sec:obstruction_theories} provide the homotopical algebraic background to prove the geometric HKR Theorem in characteristic zero. It uses the To\"en-Vezzosi approach to derived algebraic geometry, and is also the approach used in \cite{BKK} to lay the foundations of derived analytic geometry. Roughly, starting with an \(\infty\)-category \(\mathbf{M}\) with a Grothendieck topology, one can define the categories of \textit{prestacks}, \textit{stacks} and \textit{geometric stacks} as suitable functor categories, taking values in \(\infty\)-groupoids. This abstract formalism already suffices to define the geometrisation of the \(S^1\)-equivariant Hochschild algebra - the \textit{derived loop stack} \(\mathcal{L}(\mathcal{X})\) of a stack \(\mathcal{X}\) over \(\mathbf{M}\). However, in order that the shifted tangent stack \(\mathsf{T}(X)[-1]\) is tractable, we need more structure on the underlying \(\infty\)-category \(\mathbf{M}\). The ideal situation in this context is when \(\mathbf{M}\) is a \textit{relative derived algebraic context}, which roughly is a derived algebraic context with a geometrically interesting subcategory of affines (see \ref{subsubsec:RDAG}).  These axioms are satisfied by our categories of interest, namely, \(\mathbf{Ch}(\mathsf{CBorn}_R)\) and \(\mathbf{Ch}(\mathsf{Ind}(\mathsf{Ban}_R))\), with suitable subcategories of \textit{dagger Stein algebras} (see \ref{subsec:analytic_geometry}). Finally, given such a relative derived algebraic context, we study obstruction theories as a method of producing geometrically well-behaved contangent complexes and shifted tangent spaces associated to relative schemes. This leads to an appropriate geometrisation of our `affine' HKR result in Section \ref{sec:Geometric_HKR}:

\begin{theorem}
Let \(\mathcal{X}\) be a relative scheme in a \(\Q\)-enriched, relative derived algebraic geometry context. Then there is an equivalence 
\[\mathcal{L}(\mathcal{X}) \cong \mathsf{T}(\mathcal{X})[-1]\]
that is natural in \(\mathcal{X}\).
\end{theorem}

Finally in Section \ref{sec:condensed} we explain how one can prove a HKR theorem in the context of Condensed Mathematics, in the sense of \cite{clausenscholze1}, \cite{clausenscholze2}.

\section{Preliminaries from bornological algebra}\label{sec:bornologies}

In this section, we recall some terminology from the series of articles (\cites{bambozzi2016dagger,BaBK,bambozzi2019analytic,ben2020fr}), which develops various forms of analytic geometry relative to closed, symmetric monoidal categories. The idea of relative algebraic geometry goes back to Deligne (\cite{deligne2007categories}), and can be summarised as follows: let \((\fC, \otimes, 1)\) be a closed, symmetric monoidal category. One can then abstractly define \textit{affine schemes} as the opposite category \(\mathsf{Aff}(\fC) \defeq \mathsf{CAlg}(\fC)^\op\) of commutative monoids in \(\fC\). The category of \textit{pre-sheaves} is defined as the functor category \(\mathsf{PSh}(\fC) \defeq \mathsf{Fun}(\mathsf{Aff}(\fC)^\op, \mathsf{Set})\). By equipping the category of affine schemes with a suitable Grothendieck topology, we can then single out the full subcategory \(\mathsf{Sh}(\fC)\) of \textit{sheaves}. These are presheaves satisfying descent for the chosen Grothendieck topology on \(\mathsf{Aff}(\fC)\). By replacing the category of sets with the category of simplicial sets, equipped with its usual Kan-Quillen model structure, we can define the category \[\mathsf{PStk}(\fC) \defeq \mathsf{Fun}(\mathsf{Aff}(\fC)^\op, \mathsf{sSet})\] of \textit{pre-stacks} in \(\fC\). This is itself a model category in an obvious way: a morphism \(F \to G\) in \(\mathsf{PStk}(\fC)\) is a weak equivalence (resp. fibration) if for all \(X \in \mathsf{Aff}(\fC)\), \(F(X) \to G(X)\) is a weak equivalence (resp. fibration) in \(\mathsf{sSet}\). The category \(\mathsf{Stk}(\fC)\) of \textit{stacks} is the full subcategory of \(\mathsf{PStk}(\fC)\) consisting of prestacks, satisfying descent for hypercovers relative to the Grothendieck topology on \(\mathsf{Aff}(\fC)\).   We will return to this general setup in more detail in Section \ref{sec:monoidal_geometry}.

In what follows, we specialise the abstract framework above to the context of analytic geometry. This uses the language of \textit{bornologies}, rather than more frequently studied world of \textit{topological} algebras. The advantages of using bornologies over topologies have previously been pointed out on several occasions in non-commutative geometry (\cites{Meyer:HLHA, CCMT, PhD_Muk}), and of course, analytic geometry. We recall here the relevant notions from bornological geometry.

\begin{definition}\label{def:Banach_ring}
A \textit{Banach ring} is a commutative, unital ring \(R\) with a function \(\abs{\cdot} \colon R \to \R_{\geq 0}\) satisfying the following properties:

\begin{itemize}
\item \(\abs{a} = 0\) if and only if \(a = 0\);
\item \(\abs{a+b} \leq \abs{a} + \abs{b}\) for all \(a\), \(b \in R\);
\item there exists a \(C>0\) such that \(\abs{ab} \leq C\abs{a}\abs{b}\) for all \(a\), \(b \in R\);
\item \(R\) is complete with respect to the metric induced by the function \(\abs{\cdot}\).
\end{itemize}

A Banach ring $R$ is said to be non-Archimedean if $|a+b|\le\mathrm{max}\{|a|,|b|\}$.

\end{definition}

There are several choices of symmetric monoidal categories that we can plug into the theory of relative algebraic geometry sketched above. For instance, starting with a Banach ring \(R\), we can define the following categories:

\begin{itemize}
\item \(\mathsf{NMod}_R^{1/2}\) - semi-normed \(R\)-modules;
\item \(\mathsf{NMod}_R\) - normed \(R\)-modules;
\item \(\mathsf{Ban_R}\)- Banach \(R\)-modules;
\item \(\mathsf{Ind}(\mathsf{NMod}_R^{1/2})\) - inductive systems of semi-normed \(R\)-modules;
\item \(\mathsf{Ind}(\mathsf{NMod}_R)\) - inductive systems of normed \(R\)-modules;
\item \(\mathsf{Ind}(\mathsf{Ban}_R)\) - inductive systems of Banach \(R\)-modules.
\end{itemize}

We now define the objects of the categories defined above more concretely.

\begin{definition}\label{def:Ban_module}
Let \((R,\abs{\cdot}_R)\) be a Banach ring. A \textit{semi-normed \(R\)-module} is an \(R\)-module \(M\), together with a function \(\abs{\cdot}_M \colon M \to \R_{\geq 0}\) satisfying 

\begin{itemize}
\item \(\abs{0}_M = 0\);
\item \(\abs{m+n}_M \leq \abs{m}_M + \abs{n}_M\);
\item \(\abs{a\cdot m}_M \leq C \abs{a}_R \abs{m}_M\) for some \(C>0\);
\end{itemize}

A \textit{normed \(R\)-module} is a semi-normed \(R\)-module where the norm is non-degenerate, that is, \(\abs{m}_M = 0\) if and only if \(m = 0\). We call a normed \(R\)-module a \textit{Banach \(R\)-module} if it is complete with respect to the metric induced by the norm. 
\end{definition}
If $R$ is non-Archimedean, then a semi-normed (/normed/ Banach) module is said to be \textit{non-Archimedean} if $\abs{m+n}_M\le\mathrm{max}\{\abs{m}_M ,\abs{n}_M\}$. When $R$ is non-Archimedean it will be assumed that all semi-normed $R$-modules are also non-Archimedean.

\begin{remark}\label{rem:absolute_homogeneity}
We prefer not to use the notion of absolute homogeneity (i.e., \(\norm{a \cdot m}_M = \norm{a}_R \norm{m}_M\)) as used in, for instance, \cite{CCMT}*{Example 2.4}. If one requires absolute homogeneity, then one must be mindful of torsion issues. For instance, if \(R = \Z_p\), then \(\resf_p\) with the trivial norm (\(\norm{m} = 1\) for all \(m \neq 0\)) is a normed \(R\)-module in the sense of Definition \ref{def:Ban_module}, but is \textit{not} a normed \(R\)-module if we additionally require absolute homogeneity. 

But a more fundamental reason comes from our future work in global analytic geometry, which is analytic geometry over Banach rings. In the category of Banach rings, the ring \(\Z\) with its Euclidean topology is the initial object. And, since any such global theory should contain ordinary algebraic geometry, we require that \(\Z\) with the trivial norm is a normed module over \(\Z\) with the Euclidean norm. This is of course only possible if we relax the condition of absolute homogeneity.  
\end{remark}

\begin{remark}[An important convention]
From now on we shall fix a universe $\mathbb{U}$ containing the underlying set of $R$, and assume that the underlying sets of all semi-normed/ normed/ Banach $R$-modules are $\mathbb{U}$-small. So as not to overload notation, $\mathsf{Ban}_{R}$ will denote the category of $\mathbb{U}$-small Banach $R$-modules. This will allow us to take the ind-completion without concern for set-theoretic problems. 
\end{remark}
%\textcolor{red}{The condition \(\norm{am} \leq C \norm{a}\norm{m}\) in the definition above requires some attention. If we allow absolute homogeneity, i.e., \(\norm{am} = \norm{a}\norm{m}\) (which is the standard notion in some contexts), then one must be mindful of torsion issues. For instance, if \(R = \Z_p\), then \(M = \resf_p\) is a normed \(R\)-module (with the trivial norm) in the above definition, but absolute homogeneity would prevent torsion-modules from being normed.} \textcolor{blue}{ I think absolute homogeneity is a bad way to go even for torsion-free things - over $\mathbb{Z}_{triv}$ would not be a normed  $\mathbb{Z}_{an}$-module, so the definition we have here is definitely the right one.}.

The morphisms in the categories \(\mathsf{NMod}_R^{1/2}\), \(\mathsf{NMod}_R\), and \(\mathsf{Ban}_R\) are bounded linear maps: that is, \(R\)-linear maps \(f \colon M \to N\) satisfying \(\norm{f(m)}_N \leq C \norm{m}_M\) for some \(C>0\), and for all \(m \in M\). These categories are all closed symmetric monoidal when equipped with the \textit{projective tensor product}, which we now describe. For $R$ modules $M$ and $N$, denote by $M\otimes_{\pi,R}N$ the semi-normed $R$-module whose underlying $R$-module is the usual tensor product $M\otimes_{R}N$, and we equip this with the semi-norm
$$||v||_{\pi}\defeq\mathrm{inf}\{\sum ||m_{i}||_{M} ||n_{i}||_{N}: v=\sum m_{i}\otimes n_{i}\}$$
in the Archimedean case, and 
$$||v||_{\pi}\defeq\mathrm{inf}\{\mathrm{max}\{||m_{i}||_{M} ||n_{i}||_{N}: v=\sum m_{i}\otimes n_{i}\}$$
in the non-Archimedean case. The tensor product $\overline{\otimes}_{\pi,R}$ in $\mathsf{NMod}_R$ is the separation of the semi-normed $R$-module $M\otimes_{\pi,R}N$, and the tensor product $\haotimes_{\pi,R}$ in $\mathsf{Ban}_{R}$ is the separated completion. Since the projective tensor product is the only tensor product that features in what follows in this article, we drop the \(\pi\) from the subscript, and denote the completed tensor product by \(\haotimes_R\). For all the three tensor products, the internal hom $\underline{\mathsf{Hom}}(M,N)$ is the collection of bounded \(R\)-linear maps from $M$ to $N$, with semi-norm
$$||T||\defeq\mathrm{inf}\{C\ge 0:||Tm||_{N}\le ||m||_{M}\;\forall m\in M\}.$$

From the viewpoint of our applications, we are mainly interested in complete objects. This makes the category \(\mathsf{Ban}_R\) more relevant than \(\mathsf{NMod}_R\) and \(\mathsf{NMod}_R^{1/2}\). With the completed projective tensor product \(\haotimes_R\), \(\mathsf{Ban}_R\) is a symmetric, monoidal category that clearly has all finite limits and colimits. But the category \(\mathsf{Ban}_R\) is not closed under arbitrary products and coproducts. This is corrected by working in the category \(\mathsf{Ind}(\mathsf{Ban}_R)\) of all formal inductive systems \(\setgiven{X \colon I \to \mathsf{Ban}_R}{I \text{ directed set}}\). 

\subsection{Interlude: Indisation of categories}
Let $\mathsf{C}$ be a \(\mathbb{U}\)-small category. The \textit{indisation } of $\mathsf{C}$, denoted $\mathsf{Ind}(\mathsf{C})$ is the free cocompletion of $\mathsf{C}$ by filtered colimits. As in \cite{kashiwara-schapira}*{Chapter 6}, it can be explicitly realised as the full subcategory of the category of presheaves $\mathsf{Fun}(\mathsf{C}^{op},\mathsf{Set})$ spanned by those functors which are \textit{filtered} colimits of representable functors. We will write an object of $\mathsf{Ind}(\mathsf{C})$ as $``\colim_{\mathcal{I}}"X_{i}$ where $\mathcal{I}$ is a filtered category, and $X:\mathcal{I}\rightarrow\mathpzc{C}$ is a diagram, and we identify $\mathsf{C}$ as a full subcategory of $\mathsf{Ind}(\mathsf{C})$ via the Yoneda embedding.  If $``\colim_{\mathcal{J}}"Y_{j}$ is another object, then we have
$$\mathsf{Hom}_{\mathsf{Ind}(\mathsf{C})}(``\colim_{\mathcal{I}}"X_{i},``\colim_{\mathcal{J}}"Y_{j})\cong\lim_{\mathcal{I}^{op}}\colim_{\mathcal{J}}\mathsf{Hom}_{\mathsf{C}}(X_{i},Y_{j}).$$ 

\begin{definition}
If $\mathsf{C}$ is a category which has all filtered colimits then we call an object $c\in\mathsf{C}$ \textit{tiny} if $\mathrm{Hom}(c,-):\mathsf{C}\rightarrow\mathsf{sSet}$ commutes with filtered colimits.
\end{definition}

The following  fact will be crucial for us. 
\begin{proposition}\cite{kashiwara-schapira}*{Proposition 6.3.4}
An object of $\mathsf{Ind}(\mathsf{C})$ is tiny if and only if it is isomorphic to an object of $\mathsf{C}$. 
\end{proposition}
In fact the proof of this proposition is essentially tautological once one has a concrete description of the colimit of a diagram of filtered objects in $\mathsf{Ind}(\mathsf{C})$. Let $\mathcal{I}$ be a filtered category, and $D:\mathcal{I}\rightarrow\mathsf{Ind}(\fC)$ a functor. For each $i\in\mathcal{I}$ write $D(i)=``\colim_{j_{I}\in\mathcal{J}_{i}}"D_{j_{i}}$. Consider the category $\mathrm{Tot}(D)$ whose set of objects is $\coprod_{i\in\mathcal{I}}\mathrm{Ob}(\mathcal{J})_{i}$. A map $j_{i}\rightarrow j'_{i'}$ is a pair $(\alpha,f)$ where $\alpha:i\rightarrow i'$ is a map in $\mathcal{I}$, and $f:D_{j_{i}}\rightarrow D_{j'_{i'}}$ is a map in $\mathsf{C}$ such that the diagram below commutes
\begin{displaymath}
\xymatrix{
D_{j_{i}}\ar[d]\ar[r]^{f} & D_{j'_{i'}}\ar[d]\\
D(i)\ar[r]^{D(\alpha)} & D(i').
}
\end{displaymath}

\begin{proposition}
$``\colim_{\mathrm{Tot}(D)}"D_{j_{i}}$ is the colimit of $D$. 
\end{proposition}

Categories of inductive systems, such as \(\mathsf{Ind}(\mathsf{Ban}_R)\), are usually too large for explicit computations. However, in many categories \(\fC\) of interest to us, it is possible to replace it with a concrete category which is derived equivalent to $\mathsf{Ind}(\fC)$. We shall say that a diagram $D:\mathcal{I}\rightarrow\mathrm{C}$ is \textit{monomorphic} if for each map $\alpha:i\rightarrow j$ in $\mathcal{I}$, $D(\alpha):D(i)\rightarrow D(j)$ is a monomorphism in $\mathrm{C}$. A \textit{monomorphic colimit} is a colimit of a monomorphic diagram. An object $X$ of $\mathsf{Ind}(\fC)$ is said to be \textit{essentially monomorphic} if it has a presentation as a colimit $X\cong``\colim_{\mathcal{I}}"X_{i}$ of a monomorphic diagram in $D:\mathcal{I}\rightarrow\fC$. We have the following general result:

\begin{proposition}[Proposition 3.32, \cite{bambozzi2016dagger}]
Let $\fC$ be a category equipped with a faithful functor $U:\fC\rightarrow \mathsf{Sets}$ which preserves and reflects monomorphisms. Then the functor $U_{\mathsf{Ind}(\fC)}: \mathsf{Ind}^{m}(\fC)\rightarrow \mathsf{Sets}$, $``\colim_{\mathcal{I}}"X_{i}\mapsto \colim_{\mathcal{I}}U(X_{i})$ is faithful, that is, $\mathsf{Ind}^{m}(\fC)$ is a concrete category.
\end{proposition}

The explicit description of colimits in $\mathsf{Ind}(\mathrm{C})$ above immediately gives the following. 

\begin{corollary}
The category $\mathsf{Ind}^{m}(\fC)$ is closed under monomorphic filtered colimits in $\mathsf{Ind}(\fC)$. In particular if $c$ is an object of $\fC$ then the functor
$$\mathsf{Hom}_{\mathsf{Ind}^{m}(\fC)}(c,-)$$
commutes with monomorphic filtered colimits.
\end{corollary}

\subsection{Bornological modules of convex type}

The following lemma summarises why \(\mathsf{Ind}(\mathsf{Ban}_R)\) is a desirable setting for relative geometry, in a sense that was motivated at the start of this section:

\begin{lemma}\label{lem:ind_ban_bicomplete}
The category \(\mathsf{Ind}(\mathsf{Ban}_R)\) of formal inductive systems is a bicomplete, closed, symmetric monoidal category. 
\end{lemma}

\begin{proof}
That \(\mathsf{Ind}(\mathsf{Ban}_R)\) is a closed, symmetric monoidal category follows from the more general result that if \(\fC\) is a closed, symmetric monoidal category, then so is the category \(\mathsf{Ind}(\mathsf{C})\) (see \cite{Meyer:HLHA}*{Proposition 1.136}, for instance). The claim about bicompleteness follows from adapting \cite{Meyer:HLHA}*{Proposition 1.135} to the more general context of Banach \(R\)-modules, and this goes through without any obstruction. \qedhere
\end{proof}

\begin{remark}
The category \(\mathsf{Ind}(\mathsf{Ban}_R)\) embeds fully faithfully into the larger categories \(\mathsf{Ind}(\mathsf{NMod}_R)\) and \(\mathsf{Ind}(\mathsf{NMod}_R^{1/2})\), of inductive systems of normed and semi-normed \(R\)-modules. Since the categories \(\mathsf{NMod}_R\) and \(\mathsf{NMod}_R^{1/2}\) are closed and have finite limits and colimits, we can again bootstrap the arguments in \cite{Meyer:HLHA}*{Proposition 1.135, Proposition 1.136} to conclude that \(\mathsf{Ind}(\mathsf{NMod}_R)\) and \(\mathsf{Ind}(\mathsf{NMod}_R^{1/2})\) are closed, symmetric monoidal bicomplete categories. 
\end{remark}

When \(R\) is a non-trivially valued Banach field, one can give an intrinsic description of the subcategory \(\mathsf{Ind}^m(\mathsf{Ban}_R)\subseteq \mathsf{Ind}(\mathsf{Ban}_R)\) of essentially monomorphic objects in terms of \textit{bornologies}, which we now discuss.

\begin{definition}\label{def:bornology}
A \textit{bornology} on a set \(X\) is a collection \(\mathfrak{B}\) of its subsets, satisfying the following:

\begin{itemize}
\item for every \(x \in X\), \(\{x\} \in \mathfrak{B}\);
\item if \(S\), \(T \in \mathfrak{B}\), then \(S \cup T \in \mathfrak{B}\);
\item if \(S \in \mathfrak{B}\), and \(T \subseteq S\), then \(T \in \mathfrak{B}\). 
\end{itemize}

\end{definition}

We call members of a bornology on a set its \textit{bounded subsets}. Let \((X,\mathfrak{B}_X)\) and \((Y,\mathfrak{B}_Y)\) be bornological sets. A function \(f \colon X \to Y\) is called \textit{bounded} if it maps bounded subsets of \(X\) to bounded subsets of \(Y\). Bornological sets and bounded maps form a category, $\mathsf{bSet}$.  

If \((X,\mathfrak{B}_X)\) and \((Y,\mathfrak{B}_Y)\) are bornological sets, then their product is the bornological set whose underlying set is $X\times Y$ such that a set $B\subset X\times Y$ is bounded if there are bounded subsets $B_{1}\in\mathfrak{B}_{X}, B_{2}\in\mathfrak{B}_{Y}$ such that $B\subset B_{1}\times B_{2}$.

If $R$ is a Banach ring and $M$ a semi-normed $R$-module, then $M$ can be equipped with a natural bornology called the \textit{von Neumann bornology}, whereby a subset $B\subset M$ is bounded if the restriction of the semi-norm on $M$ to $B$ is a bounded function. In this way, $R$ can itself be equipped with a bornology. 

\begin{definition}\label{def:bornological_module}
Let \(R\) be a Banach ring. A \textit{bornological \(R\)-module} \(M\) is an \(R\)-module with a bornology such that the scalar multiplication and addition maps
$$R\times M\rightarrow M$$
$$M\times M\rightarrow M$$
are bounded functions.
\end{definition}

We denote by $\mathsf{bMod}_{R}$ the subcategory of $\mathsf{bSet}$ consisting of bornological $R$-modules and bounded $R$-linear maps. The categories $\mathsf{Born}^{\frac{1}{2}}_{R}\defeq\mathsf{Ind}^{m}(\mathsf{NMod}_{R}^{\frac{1}{2}})$, $\mathsf{Born}_{R}\defeq\mathsf{Ind}^{m}(\mathsf{NMod}_{R})$, and $\mathsf{CBorn}_{R}\defeq\mathsf{Ind}^{m}(\mathsf{Ban}_{R})$ will be called the categories of bornological $R$-modules of convex type, separated bornological $R$-modules of convex type, and complete $R$-modules of convex type respectively. To relate inductive systems with bornologies, we first note that equipping a semi-normed $R$-module with the von Neumann bornology defines a canonical functor
$$b:\mathsf{NMod}_R^{1/2}\rightarrow\mathsf{bMod}_{R}.$$

\begin{proposition}
Let \(R\) be a Banach ring. There is a faithful functor $\mathsf{Ind}^m(\mathsf{NMod}_R)\rightarrow\mathsf{bMod}_{R}$. If $R$ is a Banach field then this functor is fully faithful.
\end{proposition}
\begin{proof}
Let $``\colim_{\mathcal{I}}"X_{i}$ be an essentially monomorphic inductive system in \(\mathsf{Ind}^m(\mathsf{NMod}^{\frac{1}{2}}_R)\), where \(I\) is a directed set. Then its inductive limit \(\varinjlim_i M_i = \bigcup_{i \in I} M_i\) with the \textit{inductive limit bornology} - that is, a subset is bounded if and only if it is contained in the image of \(M_i\) for some \(i \in I\) - is a bornological \(R\)-module. This construction is functorial, and we still denote this functor by $b$. To prove that this functor is faithful it suffices to prove that for any semi-normed $R$-module $N$ the map
$$\colim_{\mathcal{I}}\mathsf{Hom}_{\mathsf{NMod}^{\frac{1}{2}}_R}(N,M_{i})\rightarrow\mathsf{Hom}_{\mathsf{bMod}_{R}}(b(N),b(M))$$
is injective. Since each map $\mathsf{Hom}_{\mathsf{NMod}^{\frac{1}{2}}_R}(N,M_{i})\rightarrow\mathsf{Hom}_{\mathsf{bMod}_{R}}(b(N),b(M))$ is injective, so is the map from the colimit.

For the final claim, in the case that $R=k$ is a non-trivially valued Banach field, we need to show that the map $\colim_{\mathcal{I}}\mathsf{Hom}_{\mathsf{NMod}^{\frac{1}{2}}_R}(N,M_{i})\rightarrow\mathsf{Hom}_{\mathsf{bMod}_{R}}(b(N),b(M))$ is an epimorphism. Let $f:b(N)\rightarrow b(M)$ be a bounded map. The ball $B_{N}(0,1)$ of radius $1$ in $N$ is bounded. Thus $f(B_{N}(0,1))$ is contained in a bounded set in $b(M)$, i.e. it there is some $i$ and some $r_{i}>0$ such that $f(B_{N}(0,1))\subset B_{M_{i}}(0,r_{i})$. Since $f$ is $R$-linear, $f$ maps the span of $B_{N}(0,1)$, namely $N$, to the span of $B_{M_{i}}(0,r_{i})$, namely $M_{i}$. Thus $f$ factors through a bounded map $b(N)\rightarrow b(M_{i})$.
\end{proof}

If $R$ is a non-trivially valued Banach field, then an object $M\in\mathsf{bMod}_{R}$ is in the essential image of $b:\mathsf{Ind}^m(\mathsf{NMod}_R)\rightarrow\mathsf{bMod}_{R}$ precisely when the absolutely convex hull of any bounded set is bounded by \cite{bambozzi2014generalization}*{Chapter 1}.

\begin{remark}
Over a non-Archimedean ring $R$ it is sensible to require that bornological modules over $R$ satisfy the assumption that the $R$-span of any bounded subset is bounded. In this case, there is an equivalence of categories between bornological \(R\)-modules and essentially monomorphic inductive systems of \(R\)-modules. However, the semi-norm on the \(R\)-modules is additional structure. Bornological analysis is extensively developed in  \cites{CCMT,Cortinas-Meyer-Mukherjee:NAHA, Meyer-Mukherjee:HA}, where the authors in particular discuss the relationships between bornologies and inductive systems over the \(p\)-adic integers \(\Z_p\), or more generally, complete discrete valuation rings. The main reason for the difference between the two approaches to bornologies stems from Remark \ref{rem:absolute_homogeneity} - by allowing for absolute homogeneity, the finite field \(\mathbb{F}_p\) (viewed as a complete bornological \(\Z_p\)-module) cannot be expressed as an object of \(\mathsf{Ind}(\mathsf{Ban}_{\Z_p})\).
\end{remark}

The categories \(\mathsf{Born}_R\) and \(\mathsf{CBorn}_R\) still share the same appealing properties as \(\mathsf{Ind}(\mathsf{NMod}_R)\) and \(\mathsf{Ind}(\mathsf{Ban}_R)\), namely:

\begin{lemma}\label{lem:bornologies_bicomplete}
The categories \(\mathsf{Born}_R\) and \(\mathsf{CBorn}_R\) are closed, symmetric monoidal categories with the projective and completed projective tensor products, respectively. They are both bicomplete.
\end{lemma}

\begin{proof}
See \cite{ben2020fr}*{Lemma 3.59}, for instance. 
\end{proof}

\begin{remark}
It might appear more natural to work in the category \(\mathsf{Fr}_R\) of Fr\`echet \(R\)-modules. Indeed, when \(R = \C\), this category contains \(\mathcal{C}^\infty(M)\) smooth functions on a manifold, and holomorphic functions \(\mathcal{O}(X)\) on a complex analytic space. But in sharp contrast with \(\mathsf{CBorn}_R\) or \(\mathsf{Ind}(\mathsf{Ban}_R)\), the category \(\mathsf{Fr}_R\) is neither closed nor bicomplete. A fundamental issue with categories such as \(\mathsf{Fr}_R\) or the category of complete, locally convex topological vector spaces is that there is no canonical topology on the set of continuous linear maps between such spaces. This is one of the foundational issues that have led to the development of bornological analysis, and more recently, \textit{condensed mathematics} due to Clausen and Scholze. 
\end{remark}

\section{Derived Algebraic Contexts}\label{sec:DACs}

\subsection{$(\infty,1)$-Category Preliminatires}
The results of this paper are most naturally formulated using the language of $(\infty,1)$-categories. Often these results can be stated in an abstract way, without referring to a specific model. However to be precise, we will fix the model category of simplicial categories with the Bergner model structure (\cite{bergner2007model}) as our model. For some proofs and computations we will need to work with presentations of $(\infty,1)$-categories by model categories, which will be assumed to be simplicial and combinatorial. 
\begin{itemize}
\item
We will use the mathrm font $\mathrm{A},\mathrm{B},\mathrm{C},\ldots$ to denote $1$-categories.
\item
We will use the mathpzc font $\mathpzc{A},\mathpzc{B},\mathpzc{C},\ldots$ to denote model categories
\item
We will use boldface $\mathbf{A},\textbf{B},\mathbf{C},\ldots$ to denote $(\infty,1)$-categories. 
\item
Our model category structures generally arise from exact categories which we will denote using the mathcal font, $\mathcal{C},\mathcal{D},\mathcal{E},\ldots$.
\item
For a relative category $\mathpzc{C}$, we will denote the model category it presents by $\mathrm{L}^{H}(\mathpzc{C})$. 
\end{itemize}
The notation $\mathrm{L}^{H}$ here is from \cite{dwyer1980calculating} and stands for `hammock localisation'. For simplicial model categories there is an equivalent $(\infty,1)$-category which is easy to describe. If $\mathpzc{C}$ is a simplicial model category, let $\mathpzc{C}^{cf}$ denote the full subcategory consisting of fibrant-cofibrant objects. This is a simplicial category by restricting the simplicial structure on $\mathpzc{C}$, and is equivalent to $\mathrm{L}^{H}(\mathpzc{C})$.

\subsection{Stable $(\infty,1)$-Categories and $t$-Structures}
Recall (\cite{luriestable}) that an $(\infty,1)$-category $\mathbf{C}$ is said to be \textit{stable} if the suspension-loop adjunction 
$$\adj{\Sigma}{\mathbf{C}}{\mathbf{C}}{\Omega}$$
is an equivalence. If $\mathbf{C}$ is a stable $(\infty,1)$-category, then its homotopy category $\mathrm{Ho}(\mathbf{C})$ has a canonical triangulated category structure. 

A \textit{\(t\)-structure} on a triangulated category (\cite{luriestable} Definition 6.1) \((\fC, \Sigma)\) is a pair \((\fC_{\leq 0}, \fC_{\geq 0})\) of full subcategories satisfying the following:

\begin{itemize}
\item for \(X \in \fC_{\geq 0}\) and \(Y \in \fC_{\leq 0}\), \(\Hom_\fC(X,Y[-1]) = 0\);
\item there are inclusions \(\Sigma \fC_{\geq 0} \subseteq \fC_{\geq 0}\) and \(\Sigma^{-1} \fC_{\leq 0} \subseteq \fC_{\leq 0}\);
\item for any \(X \in \fC\), there is a distinguished triangle \(X' \to X \to X''\), where \(X' \in \fC_{\geq 0}\) and \(X'' \in \Sigma^{-1}\fC_{\leq 0}\).
\end{itemize}

Let \(\fC_{\geq n}\) denote the subcategory \(\Sigma^n(\fC_{\geq 0})\), and \(\fC_{\leq n}\) the subcategory \(\Sigma^n(\fC_{\leq 0})\). By definition, a \textit{t-structure on a stable \(\infty\)-category} \(\mathbf{C}\) is a \(t\)-structure on the homotopy category $\mathrm{Ho}(\fC)$ of \(\mathbf{C}\). Likewise, the subcategories \(\mathbf{C}_{\geq n}\) and \(\mathbf{C}_{\leq n}\) are defined as the full subcategories on the objects of $\mathrm{Ho}(\fC)_{\geq n}$ and $\mathrm{Ho}(\fC)_{\leq n}$, respectively.

\begin{definition}
If $(\mathbf{C}_{\ge0},\mathbf{C}_{\le0})$ is a $t$-structure on a stable $(\infty,1)$-category $\mathbf{C}$, then the \textit{heart of the $t$-structure}, denoted $\mathbf{C}^{\heart}$, is $\mathbf{C}_{\ge0}\cap\mathbf{C}_{\le0}$. 
\end{definition}

The heart of a $t$-structure is an abelian category by \cite{L}*{Remark 1.2.1.12}.

\begin{definition}[\cite{luriestable}*{Section 7}]
A $t$-structure $(\mathcal{T}_{\ge n},\mathcal{T}_{\le n})$ on a triangulated category $\mathcal{T}$ is said to be \textit{right complete} if the subcategory $\cap_{n\in\mathbb{Z}}\mathcal{T}_{\le n}$ consist only of zero objects of $\mathcal{T}$. A $t$-structure on a stable $(\infty,1)$-category $\textbf{C}$ is said to be \textit{right complete} if the $t$-structure on $\mathsf{Ho}(\textbf{C})$ is right complete.
\end{definition}

\begin{definition}[\cite{luriestable}*{Section 7}]
Let $(\mathbf{C}_{\ge0},\mathbf{C}_{\le0})$ be a $t$-structure on a stable $(\infty,1)$-category $\mathbf{C}$. The \textit{right completion} of $\mathbf{C}$, denoted $\hat{\mathbf{C}}$ is the subcategory of $\textbf{Fun}(N({\mathbb{Z}}^{op}),\mathbf{C})$ consisting of functors $F$ such that
\begin{enumerate}
\item
For each $n\in\mathbb{Z},F(n)\in\mathbf{C}_{\ge n}$.
\item
For each $m\ge n$, the map $F(m)\rightarrow F(n)$ induces an equivalence $F(m)\rightarrow \tau_{\ge m}F(n)$. 
\end{enumerate}
\end{definition}

Lurie (\cite{luriestable}*{Section 7}) shows that a $t$-structure $(\mathbf{C}_{\ge0},\mathbf{C}_{\le0})$ on a stable $(\infty,1)$-category $\mathbf{C}$ is right-complete if and only if the map $\mathbf{C}\rightarrow\hat{\mathbf{C}}$ is an equivalence.

\begin{example}[\cite{raksit2020hochschild}, Construction 3.3.6]
A fundamental example of a \(t\)-structure is the \textit{Postnikov \(t\)-structure} on the derived \(\infty\)-category \(\mathbf{Ch}(\mathsf{Mod}_\Z)\), defined as 
\begin{multline*}
\mathbf{Ch}(\mathsf{Mod}_\Z)_{\geq 0}^P = \setgiven{X \in \mathbf{Ch}(\mathsf{Mod}_\Z)}{H_n(X) \cong 0, n<0}, \\
 \mathbf{Ch}(\mathsf{Mod}_\Z)_{\leq 0}^P = \setgiven{X \in \mathbf{Ch}(\mathsf{Mod}_\Z)}{H_n(X) \cong 0, n>0}.
\end{multline*}
The heart of this $t$-structure is equivalent to (the nerve of) $\mathsf{Mod}_\Z$.
\end{example}
Note that $\Z$ here can be replaced by any unital commutative ring $R$. We will later generalise this example to complexes of exact categories with enough projectives.

\subsection{$t$-Structures on Model Categories}
Let $\mathpzc{C}$ be a pointed, combinatorial, simplicial model category. Consider the Quillen adjunction
$$\adj{\Sigma}{\mathpzc{C}}{\mathpzc{C}}{\Omega}$$
where $\Sigma(c)\defeq  0\coprod_{c}0$ is the suspension functor, and $\Omega(c)\defeq 0\prod_{c}0$ is the loop functor. We call $\mathpzc{C}$ \textit{stable} if the suspension-loop adjunction is a Quillen equivalence. Note that  imits (resp. colimits) in $\mathrm{L}^{H}(\mathpzc{C})$ are presented by homotopy limits (resp. colimits) in $\mathpzc{C}$. It follows that $\mathrm{L}^{H}(\mathpzc{C})$ is a stable $(\infty,1)$-category if and only if $\mathpzc{C}$ is a stable model category.

\begin{definition}
Let $\mathpzc{C}$ be a combinatorial, stable, simplicial model category. A $t$-\textbf{structure} on $\mathpzc{C}$ is a pair of full simplicial subcategories stable by equivalences $(\mathpzc{C}_{\ge0},\mathpzc{C}_{\le0})$ such that $(\mathrm{Ho}(\mathpzc{C}_{\ge0}),\mathrm{Ho}(\mathpzc{C}_{\le0}))$ defines a $t$-structure on the triangulated category $\mathrm{Ho}(\mathpzc{C})$. The $t$-structure is said to be \textit{right complete} if the corresponding $t$-structure on $\mathrm{Ho}(\mathpzc{C})$ is right complete.
\end{definition}

 Using the equivalence $\mathrm{Ho}(\mathrm{L}^{H}(\mathpzc{C}))\cong\mathrm{Ho}(\mathpzc{C})$ and the fact that for $\mathpzc{D}\subset\mathpzc{C}$ a full simplicial subcategory stable by equivalences, $\mathrm{L}^{H}(\mathpzc{D})$ is a full subcategory of $\mathrm{L}^{H}(\mathpzc{C})$, the following is tautological. 

\begin{proposition}\label{prop:tmodelinf}
Let $\mathpzc{C}$ be a combinatorial, stable, simplicial model category and $(\mathpzc{C}_{\ge0},\mathpzc{C}_{\le0})$ a $t$-structure on $\mathpzc{C}$. Then $(\mathrm{L}^{H}(\mathpzc{C}_{\ge0}),\mathrm{L}^{H}(\mathpzc{C}_{\le0}))$ is a $t$-structure on the stable $(\infty,1)$-category $\mathrm{L}^{H}(\mathpzc{C})$, which is right complete provided the $t$-structure on $\mathpzc{C}$ is.  
\end{proposition}

\subsection{$t$-Structures on Monoidal $(\infty,1)$-Categories}
In this section we let $\mathbf{C}$ be a presentable, closed monoidal $(\infty,1)$-category.

\begin{definition}[\cite{raksit2020hochschild}, Definition 3.3.1]
A $t$-structure $(\mathbf{C}_{\ge0},\mathbf{C}_{\le0})$  on $\mathbf{C}$ is said to be \textit{compatible} if
\begin{enumerate}
\item
$\mathbf{C}_{\le0}$ is closed under  filtered colimits in $\mathbf{C}$
\item
The unit object $k$ of $\mathbf{C}$ lies in $\mathbf{C}_{\ge0}$.
\item
If $X,Y\in\mathbf{C}_{\ge0}$ then $X\otimes Y\in\mathbf{C}_{\ge0}$. 
\end{enumerate}
\end{definition}

Let $\mathpzc{C}$ be a simplicial combinatorial monoidal category satisfying the monoid axiom. It follows from \cite{lurieDAGII}*{Section 1.6} that $\mathrm{L}^{H}(\mathpzc{C})$ has a natural structure as a monoidal $(\infty,1)$-category. 

%\begin{proposition}
%Let $\mathpzc{C}$ be a combinatorial, stable, simplicial monoidal model category and $(\mathpzc{C}_{\ge0},\mathpzc{C}_{\le0})$ a compatible $t$-structure on $\mathpzc{C}$. Then $(\mathrm{L}^{H}(\mathpzc{C}_{\ge0}),\mathrm{L}^{H}(\mathpzc{C}_{\le0}))$ is a compatible $t$-structure on the stable $(\infty,1)$-category $\mathrm{L}^{H}(\mathpzc{C})$.  
%\end{proposition}

\begin{definition}
Let $(\mathpzc{C}, \otimes)$ be a combinatorial, stable, simplicial monoidal model category. A $t$-structure $(\mathpzc{C}_{\ge0},\mathpzc{C}_{\le0})$  on $\mathpzc{C}$ is said to be \textit{compatible} if
\begin{enumerate}
\item
$\mathpzc{C}_{\le0}$ is closed under homotopy filtered colimits in $\mathpzc{C}$
\item
The unit object $k$ of $\mathpzc{C}$ lies in $\mathpzc{C}_{\ge0}$.
\item
If $X,Y\in\mathpzc{C}_{\ge0}$ then $X\otimes Y\in\mathpzc{C}_{\ge0}$. 
\end{enumerate}
\end{definition}

\begin{proposition}\label{prop:tcompatmodel}
Let $\mathpzc{C}$ be a combinatorial, stable, simplicial monoidal model category and $(\mathpzc{C}_{\ge0},\mathpzc{C}_{\le0})$ a compatible $t$-structure on $\mathpzc{C}$. Then $(\mathrm{L}^{H}(\mathpzc{C}_{\ge0}),\mathrm{L}^{H}(\mathpzc{C}_{\le0}))$ is a compatible $t$-structure on the stable $(\infty,1)$-category $\mathrm{L}^{H}(\mathpzc{C})$.  
\end{proposition}

\begin{proof}
First note that filtered colimits in $\mathrm{L}^{H}(\mathpzc{C})$ are presented by homotopy filtered colimit in $\mathpzc{C}$. Thus $\mathrm{L}^{H}(\mathpzc{C}_{\le0})$ is closed under filtered colimits. 

If $X$ and $Y$ are objects of $\mathrm{L}^{H}(\mathpzc{C})$, then their tensor product is computed as $X\otimes Y$. The unit of $\mathrm{L}^{H}(\mathpzc{C})$ is the fibrant-cofibrant replacement of the unit of $\mathbf{C}$. It follows immediately that the monoidal structure on $\mathrm{L}^{H}(\mathpzc{C})$ is compatible with the $t$-structure. 
\end{proof}

\subsection{Derived algebraic contexts}

We are now ready to define derived algebraic contexts, and explain some of their important properties. 

\begin{definition}[\cite{raksit2020hochschild}, Definition 4.2.1]\label{def:DAC}
A \textit{derived algebraic context} is a stable, presentable symmetric monoidal \(\infty\)-category \(\mathbf{C}\) with a compatible \(t\)-structure \((\mathbf{C}_{\geq 0}, \mathbf{C}_{\leq 0})\) and a small full subcategory \(\mathbf{C}^0 \subseteq \mathbf{C}^\heartsuit\) satisfying the following properties:

\begin{itemize}
\item The \(t\)-structure is right complete;
\item The subcategory \(\mathbf{C}^0\) is a symmetric monoidal subcategory of \(\mathbf{C}\) and is closed under \(\mathbf{C}^\heartsuit\)-symmetric powers: for \(X \in \mathbf{C}^0\) and \(n \geq 0\), we have that \(\mathsf{Sym}_{\mathbf{C^\heartsuit}}^n(X) \in \mathbf{C}^0\);
\item The subcategory \(\mathbf{C}^0\) is closed under the formation of finite coproducts in \(\mathbf{C}\) and its objects form a set of tiny projective generators in \(\mathbf{C}_{\geq 0}\).  
\end{itemize} 

\end{definition}

As explained in \cite{raksit2020hochschild}, $\mathbf{C}_{\ge0}$ is the $(\infty,1)$-categorical free sifted cocompletion of $\mathbf{C}^{0}$, denoted $\mathcal{P}_{\Sigma}(\mathbf{C}^{0})$, and the right-completeness assumption then implies that $\mathbf{C}$ is the stabilisation of $\mathbf{C}_{\ge0}$. Moreover the functor
$$\mathrm{Stab}(\mathbf{C}_{\ge0})\rightarrow\mathbf{C}$$
is $t$-exact, where the left-hand side is equipped with the $t$-structure of \cite{luriestable}*{Proposition 16.4}. Thus a derived algebraic context is uniquely determined by the category $\mathbf{C}^{0}$.

Raksit uses this setup to construct categories of both non-connective and connective \textit{derived commutative rings}. Let us sketch the construction here.

Consider the functor
$$\mathsf{Sym}^{\heart}:\mathbf{C}^{0}\rightarrow\mathbf{C}_{\ge0}$$
which is the composition of the functor $\mathbf{C}^{\heart}\rightarrow\mathbf{C}_{\ge0}$ with the nerve of the $1$-categorical symmetric algebra functor on $\mathbf{C}^{\heart}$, restricted to $\mathbf{C}^{0}$. By sifted cocompletion this extends to a unique sifted colimit functor
$$\mathsf{L}\mathsf{Sym}_{\mathbf{C}_{\geq 0}}:\mathbf{C}_{\ge0}\rightarrow\mathbf{C}_{\ge0}.$$
By \cite{raksit2020hochschild}*{Section 4.2}, this functor is monadic, and in fact extends to a monad on all of $\mathbf{C}$, denoted $\mathsf{L}\mathsf{Sym}$. This monad is called  the \textit{derived symmetric algebra monad on \(\mathbf{C}\)}.  A module over this monad is called a \textit{derived commutative algebra object of \(\mathbf{C}\)}. Denote the \(\infty\)-category of derived commutative algebra objects of \(\mathbf{C}\) by \(\mathsf{DAlg}(\mathbf{C})\). Given a derived commutative algebra object \(A\), a derived commutative algebra object \(B\) of \(\mathbf{C}\) equipped with a morphism \(A \to B\) in \(\mathsf{DAlg}(\mathbf{C})\), is called a \textit{derived commutative \(A\)-algebra}. The \(\infty\)-category of derived commutative \(A\)-algebras is denoted by \(\mathsf{DAlg}_A(\mathbf{C})\). Finally, we denote by
$$\mathsf{DAlg}(\mathbf{C})^{cn}\defeq\mathsf{DAlg}(\mathbf{C})\times_{\mathbf{C}}\mathbf{C}_{\ge0}$$
the category of \textit{connective derived commutative algebras}.

On the other hand we have the usual symmetric algebra monad $\mathsf{Sym}_{\mathbf{C}}$ on $\mathbf{C}$, which is the monad associated to the commutative operad. The category of modules for this monad will be denoted
$$\mathsf{CAlg}(\mathbf{C}).$$ Again by \cite{raksit2020hochschild}, there is a map of monads
$$\mathsf{Sym}_{\mathbf{C}}\rightarrow\mathsf{L}\mathsf{Sym}_{\mathbf{C}}.$$ This is not in general an equivalence, but by \cite{raksit2020hochschild}*{Proposition 4.2.27}, the induced map of categories $$\Theta:\mathsf{DAlg}(\mathbf{C})\rightarrow\mathsf{CAlg}(\mathbf{C})$$
commutes with small limits and colimits.

\begin{example}
The \(\infty\)-category \(\mathbf{C}_{\mathrm{ab}} = \mathbf{Ch}(\mathsf{Mod}_\Z)\) with its usual algebraic tensor product and Postnikov \(t\)-structure is a derived algebraic context. Here \(\mathbf{C}_{\mathrm{ab}}^0\) is the full subcategory spanned by \(\setgiven{\Z^n}{n \in \N}\). Derived commutative algebra objects in \(\mathbf{Ch}(\mathsf{Mod}_\Z)\) are called \textit{derived commutative rings}. By \cite{raksit2020hochschild}*{Remark 4.3.2}, \(\mathbf{C}_{\mathrm{ab}}\) is the initial derived algebraic context: that is, if \(\mathbf{C}\) is any other derived algebraic context, there is a unique morphism of derived algebraic contexts \(\mathbf{C}_{\mathrm{ab}} \to \mathbf{C}\). 
Furthermore, the \textit{connective} part  
is equivalent to the \(\infty\)-category of simplicial commutative \(\Z\)-algebras.  It is in this sense that the \(\infty\)-category of derived commutative rings can be viewed as a non-connective analogue of the \(\infty\)-category of simplicial commutative rings. 
\end{example}

This works for any unital commutative ring $R$. We shall write $\mathbf{C}_{R}$ for the derived algebraic context $\mathbf{Ch}(\mathsf{Mod}_{R})$ equipped with the Postnikov $t$-structure, and with $\mathbf{C}^{0}$ being the set of free $R$-modules of finite rank.. The following is essentially tautological.

\begin{proposition}
Let $\mathbf{C}$ be a derived algebraic context. Then the unique map $i:\mathbf{C}_{\mathrm{ab}}\rightarrow\mathbf{C}$ factors through the base change
$$R\otimes_{\mathbb{Z}}(-):\mathbf{C}_{\mathrm{ab}}\rightarrow\mathbf{C}_{R}$$
if and only if $\mathsf{Ho}(\mathbf{C})$ is enriched over $R$. 
\end{proposition}

\begin{proof}
Suppose the unique map $\mathbf{C}_{\mathrm{ab}}\rightarrow\mathbf{C}$ factors through the base change
$R\otimes_{\mathbb{Z}}(-):\mathbf{C}_{\mathrm{ab}}\rightarrow\mathbf{C}_{R}$. Let $k$ be the monoidal unit of $\mathbf{C}$. Then for any objects $X,Y$ of $\mathbf{C}$, $\mathrm{Hom}_{\mathrm{Ho}(\mathbf{C})}(X,Y)$ is a module over $\mathrm{Hom}_{\mathrm{Ho}(\mathbf{C})}(k,k)$. But there is a map $R\cong\mathrm{Hom}_{\mathrm{Ho}(\mathbf{C}_{R})}(R,R)\rightarrow\mathrm{Hom}_{\mathrm{Ho}(\mathbf{C})}(k,k)$, so $\mathrm{Hom}_{\mathrm{Ho}(\mathbf{C})}(X,Y)$ is an $R$-module.

Conversely suppose $\mathrm{Ho}(\mathbf{C})$ is enriched over $R$. Let $\mathbf{C}^{0}_{R}$ denote the full subcategory of $\mathbf{C}_{R}^{\heart}$ consisting of free $R$-modules. The category $\mathbf{C}_{R}$ is equivalent to the stabilisation of $\mathcal{P}_{\Sigma}(\mathbf{C}^{0}_{R})$, the free cocompletion of $\mathbf{C}^{0}_{R}$ by sifted colimits. Thus it suffices to show that the map $i:\mathbf{C}^{0}_{\mathrm{ab}}\rightarrow\mathbf{C}$ extends to a map $i_{R}:\mathbf{C}^{0}_{R}\rightarrow\mathbf{C}$. Let $R^{\mathcal{I}}$ be an object of $\mathbf{C}^{0}_{R}$. We define $i_{R}(R^{m})\defeq i(\mathbb{Z}^{m})$. We need to construct a map $\mathrm{Map}_{\mathbf{C}^{0}_{R}}(R^{m},R^{b})\rightarrow\mathrm{Map}_{\mathbf{C}}(i(\mathbb{Z}^{m}),i(\mathbb{Z}^{n}))$. This is equivalent to a map $\mathrm{Hom}_{\Mod_{R}}(R^{m},R^{n})\rightarrow\mathrm{Hom}_{\mathrm{Ho}(\mathbf{C})}(i(\mathbb{Z}^{m}),i(\mathbb{Z}^{n}))$, i.e. a map $R^{mn}\rightarrow\mathrm{Hom}_{\mathrm{Ho}(\mathbf{C})}(i(\mathbb{Z}^{m}),i(\mathbb{Z}^{n}))$. Now there is a map $\mathbb{Z}^{mn}\rightarrow\mathrm{Hom}_{\mathrm{Ho}(\mathbf{C})}(i(\mathbb{Z}^{m}),i(\mathbb{Z}^{n}))$, and since $\mathrm{Ho}(\mathbf{C})$ is enriched over $R$, this uniquely extends to a map $R^{mn}\rightarrow\mathrm{Hom}_{\mathrm{Ho}(\mathbf{C})}(i(\mathbb{Z}^{m}),i(\mathbb{Z}^{n}))$, as required.
\end{proof}

Later in this paper we will be particularly interested in the case that $R=\mathbb{Q}$.

\begin{definition}
A derived algebraic context $\mathbf{C}$ is said to be \textit{rational} if $\mathsf{Ho}(\mathbf{C})$ is enriched over $\mathbb{Q}$. 
\end{definition}

The rational setting is much simpler.  One has for example that the functor $\Theta:\mathsf{DAlg}(\textbf{C})\rightarrow\mathsf{CAlg}(\textbf{C})$ is an equivalence.

%Note that over $\mathbb{Q}$ the map of operads $\mathrm{E}_{\infty}(-)\rightarrow\mathsf{CAlg}(-)$ is an equivalence. Thus derived commutative algebras in $\mathbf{C}$ and $\mathrm{E}_{\infty}$-algebras in $\mathbf{C}$ (or commutative monoids internal to $\mathbf{C}$) are equivalent.

%\begin{proposition}
%Let $(\mathpzc{C},\mathpzc{C}_{\ge0},\mathpzc{C}_{\le0},\mathpzc{C}^{0})$ be a model derived algebraic context.
%\end{proposition}
%\begin{cor}
%A derived algebraic context $\mathbf{C}$ is rational if and only if the unique map $\mathbf{Ch}({}_{\mathbb{Z}}\mathsf{Mod})\rightarrow\mathbf{C}$ factors through the base change
%$$\mathbb{Q}\otimes_{\mathbb{Z}}(-):\mathbf{C}_{\mathrm{ab}}\rightarrow\mathbf{C}_{\mathbb{Q}}$$
%\end{cor}

\subsection{Derived algebraic contexts from model categories}

For computations and proofs, it is convenient to have derived algebraic contexts presented by some model category. In the next section we will in fact see that there is a canonical choice of model category presenting a given derived algebraic context.

%\textcolor{blue}{
%\begin{defnition}
%A monoidal model category is said to have symmetric cofibrant objects if for any cofibrant object $C$, and any $n$, $\mathsf{Sym}^{n}(C)$ is cofibrant.
%\end{definition}
%}

\begin{definition}\label{defn:modelder}
A \textit{model derived algebraic context} is a stable, combinatorial, simplicial monoidal model category $\mathpzc{C}$ satisfying the monoid axiom, equipped with a compatible $t$-structure $(\mathpzc{C}_{\ge0},\mathpzc{C}_{\le0})$ and a small full subcategory $\mathpzc{C}^{0}\subset\mathpzc{C}^{\heart}$ such that
\begin{itemize}
\item
The $t$-structure is right complete;
\item
The subcategory $\mathpzc{C}^{0}$ consists of fibrant-cofibrant objects, and is a symmetric monoidal subcategory of $\mathpzc{C}$ which is closed under forming \(\mathpzc{C}^{\heartsuit}\)-symmetric powers and finite coproducts in \(\mathpzc{C}\);
\item The objects of \(\mathpzc{C}^0\) form a set of tiny projective generators in \(\mathpzc{C}_{\geq 0}\), i.e. 
\begin{itemize}
\item
they are compact objects in $\mathpzc{C}_{\ge0}$.
\item
for each $P\in\mathpzc{C}^{0}$ the functor $\mathbb{R}\underline{\mathsf{Hom}}(P,-):\mathpzc{C}_{\ge0}\rightarrow\mathpzc{sSet}$ commutes with homotopy geometric realisations.  
\item
every object of $\mathpzc{C}_{\ge0}$ can be written as a homotopy sifted colimit of objects of $\mathpzc{C}^{0}$.
\end{itemize}
\end{itemize}
%A model derived algebraic context is said to be a \textit{commutative model derived algebraic context} if in addition it satisfies the commutative monoid axiom.
\end{definition}

%\textcolor{red}{Silly question: In the presentable case, is the underlying \(\infty\)-category of a model DAC in our sense a DAC context in the sense of Raksit?}\textcolor{brown}{ Ye it should be, the proofs will be spread over a few different papers (some of them are actually a lot more recent than you might expect), but I'll compile the references. I had meant to put in a statement of this actually!}

\begin{proposition}
Let $(\mathpzc{C},\mathpzc{C}_{\ge0},\mathpzc{C}_{\le0},\mathpzc{C}^{0})$  be a model derived algebraic context. Then $(\mathrm{L}^{H}(\mathpzc{C}),\mathrm{L}^{H}(\mathpzc{C}_{\ge0}),\mathrm{L}^{H}(\mathpzc{C}_{\le0}),\mathrm{L}^{H}(\mathpzc{C}^{0}))$ is a derived algebraic context.
\end{proposition}
\begin{proof}
This follows immediately from combining Definition \ref{defn:modelder} with Propositions \ref{prop:tmodelinf} and \ref{prop:tcompatmodel}.
% A $t$-structure on $\mathrm{L}^{H}(\mathpzc{C})$ is the same thing as a $t$-structure on $\mathsf{Ho}(\mathrm{L}^{H}(\mathpzc{C}))\cong\mathsf{Ho}(\mathpzc{C})$, so a $t$-structure on $\mathpzc{C}$ presents one on $\mathrm{L}^{H}(\mathpzc{C})$. Moreover there is an equivalence $\mathrm{L}^{H}(\mathpzc{C}^{\heart})\cong N(\mathrm{Ho}(\mathpzc{C})^{\heart})$. 
%
%The category $\mathbf{C}^{0}$ presents a well-defined full subcategory $\mathrm{L}^{H}(\mathbf{C}^{0})\subset \mathrm{L}^{H}(\mathbf{C})$. Since objects of $\mathrm{L}^{H}(\mathbf{C}^{0})$ are fibrant-cofibrant
%\cite{https://arxiv.org/pdf/1206.3645.pdf} Proposition 4.15 that $\mathbf{C}$ presents a stable $(\infty,1)$-category. 
\end{proof}

\subsection{Graded and filtered objects}

In order to work with the de Rham and the Hochschild complex integrally and in positive characteristic, we will need \textit{graded} and \textit{filtered} derived categories. In this section, we recall (following \cite{raksit2020hochschild}) how one can equip the \(\infty\)-category \(\mathbf{Gr}(\mathbf{C})\) of graded objects of \(\mathbf{C}\) with the structure of a (model) derived algebraic context, when \(\mathbf{C}\) is a (model) derived algebraic context.

Given an \(\infty\)-category \(\mathbf{C}\), \(\mathbf{Gr}(\mathbf{C})\) is defined as the \(\infty\)-category of \(\infty\)-functors \(N(\Z) \to \mathbf{C}\), where \(\Z\) are the integers, viewed as a category with only identity morphisms. If \(\mathbf{C}\) is stable, presentable and symmetric monoidal, then so is \(\mathbf{Gr}(\mathbf{C})\), where the symmetric monoidal structure given by the Day convolution product \[(X \circledast Y)^n \defeq \underset{i+j = n}\coprod X^i \otimes Y^j,\] where \(X = (X^n)_{n \in \Z}\) and \(Y = (Y^n)_{n \in \Z}\) are graded objects of \(\mathbf{C}\).

Now suppose \(\mathbf{C}\) is additionally a derived algebraic context. Denoting by \((\mathbf{C}_{\geq 0}, \mathbf{C}_{\leq 0})\) the \(t\)-structure of \(\mathbf{C}\), one can define a canonical (\textit{neutral}) \(t\)-structure on \(\Gr(\mathbf{C})\), with \(\Gr(\mathbf{C})_{\geq 0}\) (respectively, \(\Gr(\mathbf{C})_{\leq 0}\)) defined as graded objects valued in \(\mathbf{C}_{\geq 0}\) (respectively, in \(\mathbf{C}_{\leq 0}\)). Additionally, there are two other \(t\)-structures that play a role in the theory to follow: the \textit{positive \(t\)-structure} defined as the \(t\)-structure \((\Gr(\mathbf{C})_{\geq 0^+}, \Gr(\mathbf{C})_{\leq 0^+})\), where \(\Gr(\mathbf{C})_{\geq 0^+}\) (respectively, \(\Gr(\mathbf{C})_{\leq 0^+}\))  consists of graded objects \(X = (X^n)_{n \in \Z}\) with \(X^n \in \mathbf{C}_{\geq n}\) (respectively, with \(X^n \in \mathbf{C}_{\leq n}\)). One can analogously define a \textit{negative \(t\)-structure}, \((\Gr(\mathbf{C})_{\geq 0^-}, \Gr(\mathbf{C})_{\leq 0^-})\) by replacing \(n\) with \(-n\). These three \(t\)-structures on \(\Gr(\mathbf{C})\), constructed in \cite{raksit2020hochschild}*{Construction 3.3.2}, all have the same heart, namely, \(\Gr(\mathbf{C}^\heartsuit) = \mathsf{Fun}(N(\Z), \mathbf{C}^\heartsuit)\). Finally, the Day symmetric monoidal structure on \(\Gr(\mathbf{C})\) is compatible with each of these \(t\)-structures, and restricts to a symmetric monoidal structure on their heart \( \Gr(\mathbf{C}^\heartsuit)\). We denote by \(\Gr(\fC^\heartsuit)^{\otimes \kappa}\) the induced symmetric monoidal structure on \(\Gr(\fC^\heartsuit)\), with the Koszul sign rule implemented in the data of the symmetry isomorphisms. 

Now let \(\mathbf{C}\) be a derived algebraic context. Then the \(\infty\)-category \(\mathbf{Gr}(\mathbf{C})\) of graded objects of \(\mathbf{C}\) can be regarded as a derived algebraic context: the Day convolution product defined above yields the symmetric monoidal structure, the \(t\)-structure is given by the neutral \(t\)-structure, and \(\mathbf{Gr}(\mathbf{C})^0\) is defined as the full subcategory spanned by finite coproducts of objects \(X \in \mathbf{C}^0\), viewed as objects of \(\mathbf{Gr}(\mathbf{C})\), via the canonical inclusion \(\mathbf{C} \to \mathbf{Gr}(\mathbf{C})\). The derived commutative algebra objects of \(\mathbf{Gr}(\mathbf{C})\) are called \textit{graded derived commutative algebra objects of \(\mathbf{C}\)}, and are denoted \(\mathsf{DAlg}(\mathbf{Gr}(\mathbf{C}))\). Similarly, if \(A\) is a derived commutative algebra object of \(\mathbf{C}\), we can analogously define the \(\infty\)-category \(\mathbf{Gr}(\mathsf{DAlg}_A(\mathbf{C}))\) of graded derived commutative \(A\)-algebra objects of \(\mathbf{C}\). 

If $\mathbf{C}\cong \mathrm{L}^{H}(\mathpzc{C})$ is presented be a model derived algebraic context, then so is $\mathbf{Gr}(\mathbf{C})$. There is a model structure on $\mathpzc{Gr}(\mathpzc{C})\defeq\coprod_{n\in\mathbb{Z}}\mathpzc{Gr}(\mathpzc{C})$ in which fibrations/ cofibrations/ equivalences are defined component-wise. Then $\mathbf{Gr}(\mathbf{C})\cong \mathrm{L}^{H}(\mathpzc{Gr}(\mathpzc{C}))$.

We now turn to filtrations. Let \(\Z\) be regarded as a category with its standard pre-order. That is, its objects are integers, and there is a unique morphism \(n \to m\) if and only if \(n\leq m\). If \(\mathbf{C}\) is a presentable, stable symmetric monoidal \(\infty\)-category, we call the \(\infty\)-category \[\mathbf{Fil}(\mathbf{C}) \defeq \mathsf{Fun}(N(\Z), \mathbf{C})\] the \textit{filtered \(\infty\)-category over \(\mathbf{C}\)}. As in the graded case, the filtered \(\infty\)-category is a stable, presentable symmetric monoidal \(\infty\)-category, where the symmetric monoidal structure is given by the Day convolution product 

\[ (X \circledast Y)_n \defeq \mathsf{colim}_{i+j \geq n} X_i \otimes Y_j,\] where \(X = (X_i)\) and \(Y = (Y_j)\) are filtered objects in \(\mathbf{C}\). If \(\mathbf{C}\) is additionally a derived algebraic context, then in a similar manner as with the graded case, one can equip \(\mathbf{Fil}(\mathbf{C})\) with the structure of a derived algebraic context. 

Again, as in the graded case, we can express filtered objects in an \(\infty\)-category in terms of a concrete presentation. This is done as follows in \cite{gwilliam2018enhancing}. Let $\mathpzc{C}$ be a combinatorial left proper model category. Consider the category $\mathpzc{Seq}(\mathpzc{C})\defeq\mathpzc{Fun}(\mathbb{Z},\mathpzc{C})$, where $\mathbb{Z}$ is regarded as a poset with the usual increasing ordering. With the projective model structure for functors, this is a combinatorial left proper model category. If $\mathpzc{C}$ is a symmetric monoidal model category satisfying the monoid axiom, then with Day convolution, so is $\mathpzc{Seq}(\mathpzc{C})$. We have
$$\mathrm{L}^{H}(\mathpzc{Seq}(\mathpzc{C}))\cong\mathbf{Fil}(\mathrm{L}^{H}(\mathpzc{C}))$$

We will need the following definition, as this is the right home for the filtered Hochschild and Hodge-complete derived de Rham cohomology. 

\begin{definition}
Let \(\mathbf{C}\) be a symmetric monoidal \(\infty\)-category. A filtered object \(X = (X_i) \in \mathbf{Fil}(\mathbf{C})\) is called \textit{complete} if \(\lim_{i \to \infty} X_i \cong 0\). Denote the full subcategory of complete, filtered objects by \(\mathbf{Fil}^{\wedge}(\mathbf{C})\).   
\end{definition}

\begin{remark}
Again by \cite{gwilliam2018enhancing}, if $\mathcal{C}$ is a combinatorial, left proper, stable model category then there is a left Bousfield localisation $\mathpzc{Fil}(\mathpzc{C})$ of $\mathpzc{Seq}(\mathpzc{C})$ which presents $\mathbf{Fil}^{\wedge}(\mathrm{L}^{H}(\mathpzc{C}))$. If $\mathpzc{C}$ is monoidal and satisfies the monoid axiom, then once again there is a monoidal model structure on $\mathpzc{Fil}(\mathpzc{C})$ satisfying the monoid axiom (given by completing the Day convolution tensor product).
. \end{remark}

There is an intimate relationship between filtered and graded objects, by way of the \textit{associated graded functor} \[\mathsf{gr} \colon \mathbf{Fil}(\mathbf{C}) \to \mathbf{Gr}(\mathbf{C}),\] assigning to \((X_n)_{n \in \Z}\), the  cofibre \(\mathsf{cofib}(X_{n+1} \to X_n)\) at each degree \(n\). This restricts to an equivalence of symmetric monoidal \(\infty\)-categories (see \cite{raksit2020hochschild}*{Theorem 3.2.14}) 

 \[\mathbf{Fil}^{\wedge}(\mathbf{C}) \cong \mathsf{Mod}_{\mathbb{D}_-}(\mathbf{Gr}(\mathbf{C}),\] where \(\mathbb{D}_-\) denotes the graded object \(k \oplus k[-1]\), \(k\) the symmetric monoidal unit of \(\mathbf{C}\). We will actually need the split square-zero algebra \(k \oplus k[+1]\) to define a higher algebraic analogue of a positive degree shift as in a cochain complex, and this will be discussed in Section 5. 

There is another useful functor $\mathrm{split}:\mathbf{Gr}(\mathbf{C})\rightarrow\mathbf{Fil}(\mathbf{C})$ given by $\mathrm{split}(X)^{i}\defeq \coprod_{j\ge i}X^{j}$. This functor has a left adjoint $\mathrm{und}:\mathbf{Fil}(\mathbf{C})\rightarrow\mathbf{Gr}(\mathbf{C})$. Details of all of these constructions can be found in \cite{raksit2020hochschild}*{3.1.1}.

\section{Derived algebraic contexts from exact categories}\label{sec:exact}

In this section, we discuss the framework of exact categories, which we use in order to do homotopy theory and higher algebra. The objective of this rather technical part of this article is to create a roadmap that allows one to pass from the world of homological algebra in functional analytic categories, to the world of higher algebra. This is quite straightforward in the purely algebraic situation, where one can pass from chain complexes over a  Grothendieck abelian category \(\mathrm{Ch}(\mathcal{A})\) to a stable, presentable \(\infty\)-category \(\mathbf{Ch}(\mathcal{A})\) via the coherent nerve construction (see \cite{L}*{Proposition 1.3.5.15}). But categories that arise in functional analysis - including \(\mathsf{Ind}(\mathsf{Ban}_R)\) and \(\mathsf{CBorn}_R\) - are almost never abelian. Consequently, we need a more general framework of homological algebra, and this is provided by \textit{quasi-abelian} or \textit{exact} categories. In the end, however, it will turn out that \textit{all} derived algebraic contexts are equivalent to derived algebraic contexts arising from certain \textit{abelian} categories. The formalism of exact categories will still be useful for explicit computations. For example $\mathsf{CBorn}_{R}$ is a concrete category, but the corresponding abelian category is its left heart, which is not concrete.

\subsection{Exact categories}

Let \(\fC\) be an additive category with kernels and cokernels. A diagram of the form \begin{equation*}\label{conflation}
K \overset{i}\to E \overset{p}\to Q
\end{equation*} is called an \textit{extension} if \(i = \ker(p)\) and \(p=\coker(i)\).  An \textit{exact category} is an additive category \(\fC\) with a distinguished class of extensions \(\mathcal{E}\) called \textit{conflations}, satisfying certain properties. An arrow in \(\fC\) is called an \textit{inflation} (respectively, \textit{deflation}) if it is the arrow \(i\) (respectively, \(p\)) in a conflation. The conflations must satisfy the following axioms:

\begin{itemize}
\item[\(\bullet\)] the identity map on the zero object is a deflation;
\item[\(\bullet\)] if \(A \overset{f}\onto B\) and \(B \overset{g}\onto C\) are deflations, so is \(A \overset{g\circ f}\onto C\);
\item[\(\bullet\)] the pullback of a deflation along an arbitrary map exists and is again a deflation;
\item[\(\bullet\)] the pushout of an inflation along an arbitrary map exists and is again an inflation.
\end{itemize}

A \textit{quasi-abelian category} is a finitely complete and cocomplete additive category $\mathcal{E}$ such that the class of \textit{all} kernel-cokernel pairs defines an exact structure on $\mathcal{E}$. Note that an exact category is a category equipped with some structure, but being quasi-abelian is a \textit{property} of a category. From now on when we say `quasi-abelian category', we will be referring to a quasi-abelian category equipped with the exact structure consisting of all kernel-cokernel pairs. These are studied in great detail in \cite{qacs}.

\begin{example}
It is trivial to see that any abelian category, such as left (or right) modules \(\mathsf{Mod}_R\) over a ring \(R\), is a quasi-abelian category.
\end{example}

\subsection{Elementary Exact Categories and Projective Model Structures}

As motivated in the introduction to this section, we would like to do homotopical and higher algebra in settings where the underlying module category is not abelian. In the setting of exact categories, we define the \textit{derived category} of an exact category as the localisation of the homotopy category of complexes \(\mathsf{Ch}(\mathcal{E})\) at the quasi-isomorphisms. For an arbitrary exact category, in order that we can define a simplicial, combinatorial model structure on \(\mathsf{Ch}(\mathcal{E})\), whose homotopy category is the derived category, we need some extra structure.

\begin{definition}
A subcategory $\mathcal{P}\subset\mathcal{E}$ is said to \textit{generate} $\mathcal{E}$ if for every object $X\in\mathcal{E}$, there is an object $P\in\mathcal{P}$, and a deflation $P\onto X$. 
\end{definition}

\begin{definition}\label{def:enough_projectives}
An object \(P\) in an exact category \(\mathcal{E}\) is called \textit{projective} if the functor \(\Hom(P,-) \colon \mathcal{E} \to \mathsf{Mod}_\Z\) sends a deflation to a surjection. An exact category has \textit{enough projectives} if the full subcategory of projective objects generates $\mathcal{E}$. 
\end{definition}

Let \(\mathcal{E}\) be an exact category. In what follows, let $\mathcal{S}$ be a class of maps in \(\mathcal{E}\) and let $\mathcal{I}$ be a filtered category. Denote by \(\mathsf{Fun}_{\mathcal{S}}(\mathcal{I}, \mathcal{E})\) the category of functors \(\mathcal{I} \to \mathcal{E}\) such that for \(i \leq j\), \(F(i) \to F(j) \in \mathcal{S}\). In this category, we say a diagram \[0 \to F \to G \to H \to 0\] is \textit{exact} if for each \(i \in \mathcal{I}\), the induced diagram \[0 \to F(i) \to G(i) \to H(i) \to 0\] is a conflation in \(\mathcal{E}\). Note that we do not require this to yield an exact category structure on \(\mathsf{Fun}_{\mathcal{S}}(\mathcal{I},\mathcal{E})\). We only need this notion of exactness for the following:

\begin{definition}
An exact category $\mathcal{E}$ is said to be \textit{weakly $(\mathcal{I},\mathcal{S})$-elementary} if the functor
$$\colim:\mathsf{Fun}_{\mathcal{S}}(\mathcal{I},\mathcal{E})\rightarrow\mathcal{E}$$
exists and is exact.
\end{definition}

We will typically be interested in the case that $\mathcal{I}$ is some ordinal. In fact this ordinal will usually be $\aleph_{0}$, the poset of natural numbers.

In what follows $\mathcal{S}$ will typically be one of the following classes
\begin{itemize}
\item
the class of all morphisms in $\mathcal{E}$;
\item
the class $\mathbf{AdMon}$ of inflations in $\mathcal{E}$;
\item
the class $\mathbf{SplitMon}$ of split monomorphisms in $\mathcal{E}$;
\end{itemize}
Note that $\mathcal{E}$ is weakly $(\aleph_{0},\mathbf{SplitMon})$-elementary precisely if countable coproducts are exact. 

\begin{definition}
\begin{enumerate}
\item
Let $\mathcal{S}$ be a class of morphisms in an exact category $\mathcal{E}$. 
\begin{enumerate}
\item
An object $X$ of $\mathcal{E}$ is said to be $\mathcal{S}$-\textit{tiny} if the functor $\mathrm{Hom}(X,-):\mathcal{A}\rightarrow\mathrm{Ab}$ commutes with colimits of diagrams in $\mathsf{Fun}_{\mathcal{S}}(\mathcal{I}, \mathcal{E})$ for any small filtered category $\mathcal{I}$;
\item
An exact category $\mathcal{E}$ is said to be $\mathcal{S}$-\textit{elementary} if it has a small generating subcategory $\mathcal{P}$ consisting of $\mathcal{S}$-tiny projectives;
\end{enumerate}
\item
An exact category $\mathcal{E}$ is said to be \textit{quasi-elementary} if it has a small generating subcategory $\mathcal{P}$ consisting of projectives $P$ such that for each $P\in\mathcal{P}$, $\mathrm{Hom}(P,-):\mathcal{E}\rightarrow\mathrm{Ab}$ commutes with coproducts.
\end{enumerate}
\end{definition}

If $\mathcal{E}$ is $\mathcal{S}$-elementary for $\mathcal{S}$ being the class of all morphisms then we just say that $\mathcal{E}$ is \textit{elementary}.  Note that being $\mathbf{SplitMon}$-elementary is generally stronger than being quasi-elementary.

\begin{theorem}[ \cite{kelly2016homotopy} Theorem 4.3.58, Theorem 4.3.64]
\begin{enumerate}
\item
Let $\mathcal{E}$ be an exact category with enough projectives and which has kernels. Then the \textit{projective model structure} exists on $\mathrm{Ch}_{\ge0}(\mathcal{E})$. The weak equivalences are the quasi-isomorphisms, and the cofibrations are the maps which are degreewise inflations with projective cokernel.
\item
Let $\mathcal{E}$ be a weakly $(\aleph_{0},\mathbf{SplitMon})$-elementary exact category which has kernels and enough projectives. Then the \textit{projective model structure} exists on $\mathrm{Ch}(\mathcal{E})$. The weak equivalences are the quasi-isomorphisms, and the fibrations are the maps which are degreewise conflations.
\end{enumerate}
If $\mathcal{E}$ is $\mathbf{AdMon}$-elementary and is locally presentable as a category then both of the model structures are combinatorial. 
\end{theorem}

These model structures are moreover both simplicial (in fact $\mathrm{sAb}$-enriched). Indeed $\mathrm{Ch}(\mathcal{E})$ is naturally enriched over $\mathrm{Ch(Ab)}$ by
$$\mathrm{Hom}_{n}(X_{\bullet},Y_{\bullet})\defeq\prod_{i\in\mathbb{Z}}\mathrm{Hom}(X_{i},Y_{i+n})$$
with the differential being defined on $\mathrm{Hom}(X_{i},Y_{i+n})$ by $df= d^{Y}_{i+n}\circ f - (-1)^{n}f\circ d_{i}^{X}$. Then one defines
$$\mathrm{Hom}_{\Delta^{op}}(X_{\bullet},Y_{\bullet})\defeq N^{\mathrm{Ab}}(\mathrm{Hom}_{\bullet}(X_{\bullet},Y_{\bullet}))$$
where $N^{\mathrm{Ab}}$ is the normalised Moore complex functor from the Dold-Kan correspondence. It is shown in \cite{kelly2016homotopy}*{Corollary 4.4.91} that this defines a simplicial model structure on $\mathrm{Ch}(\mathcal{E})$ as long as $\mathcal{E}$ is (countably) complete and cocomplete. The restriction to $\mathrm{Ch}_{\ge0}(\mathcal{E})$ also defines a simplicial model structure. 

\begin{remark}[Important Remark]\label{rem:DoldKan}
For any finitely complete additive category there is a Dold-Kan equivalence
$$\adj{N^{\mathcal{E}}}{\mathrm{Ch}_{\ge0}(\mathcal{E})}{\mathrm{s}\mathcal{E}}{\Gamma^{\mathcal{E}}}$$
When $\mathcal{E}$ is an exact category with enough projectives then \cite{christensen2002quillen} Theorem 6.3 implies that there is a model category structure on $\mathrm{s}\mathcal{E}$ and moreover that the Dold-Kan equivalence is in fact a Quillen equivalence of model structures. This is also explained in detail in \cite{kelly2016homotopy}*{Section 4.4}. Furthermore, $\mathrm{s}\mathcal{E}$ is enriched over $\mathrm{sAb}$ in an obvious way. If we denote the enrichment functor $\mathrm{Hom}_{\Delta^{op}}^{\mathrm{s}\mathcal{E}}(-,-)$ then
$$\mathrm{Hom}_{\Delta^{op}}(X_{\bullet},Y_{\bullet})\cong\mathrm{Hom}_{\Delta^{op}}^{\mathrm{s}\mathcal{E}}(\Gamma^{\mathcal{E}}(X_{\bullet}),\Gamma^{\mathcal{E}}(Y_{\bullet})).$$
\end{remark}

\subsubsection{Indisation of exact categories}

One pool of examples of elementary exact categories comes from indisation of small exact categories. Let $\mathcal{E}$ be a small exact category with kernels and with enough projectives, and consider the formal completion of $\mathcal{E}$ by small filtered colimits $\mathsf{Ind}(\mathcal{E})$.

\begin{proposition}\label{prop:ind(E)_exact}
$\mathsf{Ind}(\mathcal{E})$ is an elementary exact category with kernels and $\mathsf{Ind}^{m}(\mathcal{E})$ is an $\mathbf{AdMon}$-elementary exact category with kernels. If $\mathcal{E}$ is quasi-abelian, then so are $\mathsf{Ind}(\mathcal{E})$ and $\mathsf{Ind}^{m}(\mathcal{E})$.
\end{proposition}

\begin{proof}
As explained in \cite{kelly2021note}, it follows from \cite{braunling2016tate}*{Section 3} that $\mathsf{Ind}^{m}(\mathcal{E})$ and $\mathsf{Ind}(\mathcal{E})$ are extension closed subcategories of the abelian category $\mathpzc{Lex}(\mathcal{E}^{op},\mathrm{Ab})$ of left-exact functors from $\mathcal{E}^{op}$ to the category of abelian groups. In particular they are exact categories. Moreover, a sequence
\begin{displaymath}
\xymatrix{
0\ar[r] & X\ar[r]^{f} & Y\ar[r]^{g} & Z\ar[r]&0
}
\end{displaymath}
is exact precisely if there is a filtered category $\mathcal{I}$, functors $\alpha,\beta,\gamma:\mathcal{I}\rightarrow\mathpzc{C}$ (which are monomorphic systems in the case of $\mathsf{Ind}^{m}(\mathcal{E})$), and natural transformations $\tilde{f}:\alpha\rightarrow\beta,\tilde{g}:\beta\rightarrow\gamma$, such that
\begin{displaymath}
\xymatrix{
0\ar[r] & \alpha\ar[r]^{\tilde{f}} & \beta\ar[r]^{\tilde{g}} & \gamma\ar[r] &0
}
\end{displaymath}
is an exact sequence of functors, $f\cong``\colim"\tilde{f}$, and $g\cong``\colim"\tilde{g}$. It follows immediately that projective objects in $\mathcal{E}$ are tiny projectives in both $\mathsf{Ind}(\mathcal{E})$ and $\mathsf{Ind}^{m}(\mathcal{E})$. Let $``\colim_{\mathcal{I}}"E_{i}$ be an object of $\mathsf{Ind}(\mathcal{E})$. For each $i$ let $P_{i}\rightarrow E_{i}$ be an admissible epimorphism in $\mathcal{E}$ with $P_{i}$ projective. Then $\bigoplus_{i\in\mathcal{I}}P_{i}\rightarrow\bigoplus_{i\in\mathcal{I}}E_{i}\rightarrow``\colim_{\mathcal{I}}"E_{i}$ is an admissible epimorphism in $\mathsf{Ind}(\mathcal{E})$ with $\bigoplus_{i\in\mathcal{I}}P_{i}$ projective. Note that this direct sum is in fact an object of $\mathsf{Ind}^{m}(\mathcal{E})$, so this process also works there. 

For the claim that $\mathsf{Ind}(\mathcal{E})$ is quasi-abelian when $\mathcal{E}$ is, just compare the description of short exact sequences in the induced exact structure with the description of kernels and cokernels in the ind-category from \cite{kashiwara-schapira}*{Section 6.4}. 
\end{proof}

\subsection{Monoidal Exact Categories}

We now study the interaction between monoidal structure on the exact categories we are interested in with the set of projective objects from the previous subsection. While the monoidal structure allows us to do (higher) algebra, the existence of tiny projectives provides for a convenient presentation for the derived \(\infty\)-category of complexes. For brevity, we call an exact category $\mathcal{E}$ equipped with a closed symmetric monoidal structure, with monoidal functor $\otimes$, monoidal unit $k$, and internal hom $\underline{\mathsf{Hom}}$ a \textit{monoidal exact category}. 

\begin{definition}
An object $F$ of a monoidal exact category $\mathcal{E}$ is said to be \textit{flat} if the functor
$$F\otimes(-):\mathcal{E}\rightarrow\mathcal{E}$$
sends conflations to conflations. A flat object is said to be \textit{strongly flat} if in addition $F\otimes(-)$ commutes with kernels.
\end{definition}

Note that in a closed monoidal abelian category an object is flat if and only if it is strongly flat.

\begin{definition}
A monoidal exact category $\mathcal{E}$ is said to be \textit{quasi-projectively monoidal} if 
\begin{enumerate}
\item
It has enough projectives.
\item
The tensor product of two projectives is projective.
\end{enumerate}
If in addition projectives are (strongly) flat then $\mathcal{E}$ is said to be \textit{(strongly) projectively monoidal}.
\end{definition}

We call an exact category $\mathcal{E}$ \textit{(quasi-/ strongly) monoidal elementary} if it is both elementary and (quasi-/strongly) projectively monoidal. 

\begin{lemma}\label{lem:indmon}
Let $\mathcal{E}$ be a small finitely complete and cocomplete projectively monoidal exact category. Suppose that $\mathsf{Ind}(\mathcal{E})$ has projective limits. Then $\mathsf{Ind}(\mathcal{E})$ is a projectively monoidal elementary exact category. Moreover, the tensor product of tiny objects is tiny. If $\mathsf{Ind}^{m}(\mathcal{E})$ is cocomplete and projectives in $\mathcal{E}$ are strongly flat, then $\mathsf{Ind}(\mathcal{E})$ is  projectively monoidal $\mathbf{AdMon}$-elementary and the tensor product of $\mathbf{AdMon}$-tiny objects is $\mathbf{AdMon}$-tiny.
\end{lemma}

\begin{proof}
Define $\otimes:\mathsf{Ind}(\mathcal{E})\times \mathsf{Ind}(\mathcal{E})\rightarrow \mathsf{Ind}(\mathcal{E})$ by
$$``\colim_{\mathcal{I}}"E_{i}\otimes``\colim"_{\mathcal{F}_{j}}\defeq``\colim_{\mathcal{I}\times\mathcal{J}}"E_{i}\otimes F_{j}$$
and 
$$\underline{\mathrm{Hom}}(``\colim_{\mathcal{I}}"E_{i},``\colim"_{\mathcal{F}_{j}})\defeq\lim_{\mathcal{I}^{op}}\colim_{\mathcal{J}}(E_{i},F_{j})$$
It is straightforward to check that this defines a closed monoidal structure on $\mathsf{Ind}(\mathcal{E})$. 

For the case of $\mathsf{Ind}^{m}(\mathcal{E})$, if it is cocomplete then it is a reflective subcategory of  $\mathsf{Ind}(\mathcal{E})$. In particular it is closed under limits. Moreover for any $E\in\mathcal{E}$, $\underline{\mathrm{Hom}}(E,-)$ sends monomorphic systems to monomorphic systems. It follows that for $M\in\mathsf{Ind}(\mathcal{E})$ and $N\in\mathsf{Ind}^{m}(\mathcal{E})$, $\underline{\mathrm{Hom}}(M,N)$ is in $\mathrm{Ind}^{m}(\mathcal{E})$. General nonsense implies that there is a unique closed monoidal structure $\otimes^{m}$ on $\mathsf{Ind}^{m}(\mathcal{E})$ such that the functor $\mathsf{Ind}(\mathcal{E})\rightarrow\mathsf{Ind}^{m}(\mathcal{E})$ is strong monoidal. Since all projectives in $\mathsf{Ind}(\mathcal{E})$ are objects of $\mathsf{Ind}^{m}(\mathcal{E})$, the tensor product of two projectives in $\mathsf{Ind}^{m}(\mathcal{E})$ is isomorphic to the tensor product in $\mathsf{Ind}(\mathcal{E})$, and is therefore projective. To see that projectives are flat, just note that if $P\in\mathcal{E}$ is projective and $F\in\mathsf{Ind}^{m}(\mathcal{E})$ then $P\otimes F\cong P\otimes^{m}F$. 
\end{proof}

\begin{remark}
Lemma \ref{lem:indmon} is established in the case of $\mathsf{Ind}(\mathcal{E})$ for $\mathcal{E}$ quasi-abelian in \cite{qacs} Section 2.1.4.
\end{remark}

As motivated already, the hypothesis of projectively monoidal elementary exact category implies the following:

\begin{theorem}[\cite{kelly2016homotopy} Theorem 4.3.68]\label{thm:monoidlmodel}
Let $\mathcal{E}$ be a quasi-monoidal elementary exact category. Then with the projective model structures, $\mathsf{Ch}_{\ge0}(\mathcal{E})$ and $\mathsf{Ch}(\mathcal{E})$ are monoidal model categories satisfying the monoid axiom. Moreover they are simplicial and combinatorial. In particular, they present locally presentable monoidal $(\infty,1)$-categories $\mathbf{Ch}_{\ge0}(\mathcal{E})$ and $\mathbf{Ch}(\mathcal{E})$ respectively.
\end{theorem}

\begin{remark}
The statement of Theorem 4.3.68 in \cite{kelly2016homotopy} requires that $\mathcal{E}$ is monoidal elementary, however the proof only requires it be quasi-monoidal elementary. Indeed flatness of projectives is only relevant for the proof of the second part of Proposition 4.2.55 in loc. cit. which is not relevant for the monoid axiom. The first part of Proposition 4.2.55 regarding the monoid axiom does not require this restriction. 
\end{remark}

\subsubsection{Symmetric Powers}

In this section we analyse symmetric powers of projectives in symmetric monoidal exact categories.

\begin{definition}
A projectively monoidal exact category \((\mathcal{E},\otimes)\) with countable coproducts is said to have \textit{symmetric projectives} if for any projective $P\in\mathcal{E}$ and any $n\in\mathbb{N}$, $\mathsf{Sym}_{\mathcal{E}}^{n}(P)$ is projective.
\end{definition}

%\textcolor{red}{Is \(S^n\) the same thing as \(\mathsf{Sym}_{\mathsf{C}}^n\) in previous notation?}\textcolor{brown}{Yep have changed this now}

\begin{remark}
If $\mathcal{E}$ is quasi-projectively monoidal and enriched over $\mathbb{Q}$ then it has symmetric projectives. Indeed in this case $\mathsf{Sym}_{\mathcal{E}}^{n}(P)$ is a summand of $P^{\otimes n}$, which is projective. 
\end{remark}

\begin{definition}
Call a projective generating set \(\mathcal{P}\) in a monoidal exact category \((\mathcal{E},\otimes)\) \textit{symmetrically closed} if the full subcategory on $\mathcal{P}$ is closed under both tensor and symmetric powers, with respect to the tensor product of \(\mathcal{E}\).
\end{definition}

%
%\begin{proposition}\label{prop:sym_closed}
%Suppose that \((\mathcal{E},\otimes)\)  has a symmetrically closed projective generating set. Then \((\mathcal{E},\otimes)\)  has symmetric projectives.
%\end{proposition}
%
%\begin{proof}
%Let $\mathcal{P}$ be a symmetrically closed projective generating set, and let $P$ be any projective. Then $P$ is a summand of a direct sum of objects of $\mathcal{P}$, $\bigoplus_{i\in\mathcal{I}}P_{i}$. Therefore the symmetric power $\mathsf{Sym}_{\mathcal{E}}^{n}(P)$ is a summand of $\bigoplus_{i\in\mathcal{I}}\mathsf{Sym}_{\mathcal{E}}^{n}(P_{i})$. In particular, $\mathsf{Sym}_{\mathcal{E}}^{n}(P)$ is projective.
%\end{proof}

For $G$ a finite group and $(\mathcal{E},\otimes,k)$ a finitely cocomplete monoidal category, denote by $k[G]$ the group algebra of $G$ in $\mathcal{E}$. It is the associative monoid $k[G]\defeq\coprod_{g\in G}k$ with multiplication determined by group multiplication in $G$. If $V$ is a $k[G]$-module, i.e. an object equipped with a map $G\rightarrow\mathrm{Aut}(V)$, then we write $G\otimes V\defeq k[G]\otimes V$. 

\begin{corollary}\label{cor:symproj}
A finitely cocomplete exact category $\mathcal{E}$ has symmetric projectives if and only if for any finite collection $P_{1},\ldots,P_{n}$ of projectives the coequalizer of the map $\Sigma_{n}\otimes P_{1}\otimes\ldots\otimes P_{n}\rightarrow P_{1}\otimes\ldots\otimes P_{n}$ is projective.
\end{corollary}

\begin{proof}
One direction is clear. Suppose that $\mathcal{E}$ has symmetric projectives, and let $P_{1},\ldots,P_{n}$  be a finite collection of projectives. Write $P\defeq\bigoplus_{i=1}^{n} P_i$. Then $P$ is projective. The coequaliser of $\Sigma_{n}\otimes P_{1}\otimes\ldots\otimes P_{n}\rightarrow P_{1}\otimes\ldots\otimes P_{n}$ is a retract of $S^{n}(P)$, and is therefore projective.
\end{proof}

\begin{lemma}\label{lem:IndBan_symclosed}
Let \(\mathcal{E}\) be a finitely cocomplete monoidal exact category, in which the tensor product commutes with finite colimits. Then the category $\mathsf{Ind}(\mathcal{E})$ has symmetric projectives if and only if $\mathcal{E}$ has symmetric projectives.
\end{lemma}

\begin{proof}
Since $\mathcal{E}$ is closed under tensor products and colimits as a subcategory of $\mathsf{Ind}(\mathcal{E})$, if the latter has symmetric projectives then the former clearly does as well.

Conversely, suppose $\mathcal{E}$ has symmetric projectives, and let $P$ be a projective in $\mathsf{Ind}(\mathcal{E})$. $P$ is a retract of an object of the form $\coprod_{i\in\mathcal{I}}P_{i}$ where the $P_{i}$ are projectives in $\mathcal{E}$. Then $S^{n}(P)$ is a retract of the coequaliser of $\coprod_{(i_{1},\ldots,i_{n})\in\mathcal{I}^{n}}\Sigma_{n}\otimes P_{i_{1}}\otimes\ldots\otimes P_{i_{n}}\rightarrow\coprod_{(i_{1},\ldots,i_{n})\in\mathcal{I}^{n}}P_{i_{1}}\otimes\ldots\otimes P_{i_{n}}$. The claim now follows from Corollary \ref{cor:symproj}.
\end{proof}

\begin{proposition}\label{prop:projclosed}
Let $\mathcal{E}$ be a quasi-projectively monoidal elementary exact category with symmetric projectives. Suppose further that the tensor product of two compact objects is compact. Then there exists a set $\mathcal{P}$ of tiny projective generators of $\mathcal{E}$ such that 
\begin{enumerate}
\item
for any $P\in\mathcal{P}$ and any $n\in\mathbb{N}$, $\mathsf{Sym}_{\mathcal{E}}^{n}(P)\in\mathcal{P}$
\item
for any $P,Q\in\mathcal{P}$, $P\otimes Q\in\mathcal{P}$
\item
$\mathcal{P}$ is closed under finite coproducts.
\end{enumerate}
\end{proposition}

\begin{proof}
Let $\tilde{\mathcal{P}}$ be any set of tiny projective generators. Let 
\begin{itemize}
\item
$T(\tilde{\mathcal{P}})$ be the  set of finite tensor products of objects in $\tilde{\mathcal{P}}$
\item
$S(\tilde{\mathcal{P}})$ be the set of symmetric powers of objects in $\tilde{\mathcal{P}}$
\item
$C(\tilde{\mathcal{P}})$ be the  set of finite coproducts of objects in $\tilde{\mathcal{P}}$
\end{itemize}
Since the tensor product of two compact objects is compact, and a finite colimit of compact objects is compact, all of these sets are sets of tiny projectives containing $\tilde{\mathcal{P}}$. In particular they are themselves sets of tiny projective generators. Then
$$\mathcal{P}\defeq\bigcup_{n\in\mathbb{N}}(T\circ S\circ C)^{n}(\tilde{\mathcal{P}})$$ is the required set of tiny projective generators. 
\end{proof}

\subsection{$t$-Structures on Exact Categories}

In this subsection, we describe the generalisation of the Postnikov $t$-structure to quasi-elementary exact categories with kernels.

\begin{definition}
Let $\mathcal{E}$ be an exact category with enough projectives, and consider the Postnikov \(t\)-structure \((\mathbf{Ch}(\mathsf{Mod}_\Z)_{\geq 0}^P, \mathbf{Ch}(\mathsf{Mod}_\Z)_{\leq 0}^P)\) on \(\mathbf{Ch}(\mathsf{Mod}_\Z)\). The \textit{projective} $t$-\textit{structure on }$\mathbf{Ch}(\mathcal{E})$ is the $t$-structure defined as follows:
\begin{itemize}
\item
$\mathbf{Ch}_{\ge 0}(\mathcal{E})$ is the full subcategory of $\mathbf{Ch}(\mathcal{E})$ consisting of objects $X$ such that for every projective $P$, $\Hom(P,X)$ is in $\mathbf{Ch}_{\ge0}(\mathsf{Mod}_{\mathbb{Z}})^P$.
\item
$\mathbf{Ch}_{\le 0}(\mathcal{E})$ is the full subcategory of $\mathbf{Ch}(\mathcal{E})$ consisting of objects $X$ such that for every projective $P$, $\Hom(P,X)$ is in $\mathbf{Ch}_{\le0}(\mathsf{Mod}_{\mathbb{Z}})^P$.
\end{itemize}
\end{definition}

We claim that this does define a $t$-structure on $\mathbf{Ch}(\mathcal{E})$ as long as $\mathcal{E}$ has kernels. For an object $X\in\mathbf{Ch}(\mathcal{E})$ let $\tau_{\ge0}^{L}X$ denote the complex whose $n$th entry is $0$ for $n<0$, is $X_{n}$ for $n>0$, and is $\ker(d_{X}^{0})$ for $n=0$. Similarly, we can define \(\tau_{\geq n}^L(X)\) for any \(n \geq 0\). Denote by $i_{0}$ the natural map $\tau_{\ge0}^{L}X\rightarrow X$. The proof of the following is an adaptation from the abelian case, and is  identical to \cite{milicic2007lectures}*{Lemma 1.1.1}:

\begin{lemma}\label{lem:truncation_inclusion_adjunction}
Let $X$ be a complex in $\mathbf{Ch}_{\ge n}(\mathcal{E})$, and $Y$ an object of $\mathbf{Ch}(\mathcal{E})$. The canonical map 
$$i \colon \Hom_{D(\mathcal{E})}(X,\tau_{\ge n}Y)\rightarrow \Hom_{D(\mathcal{E})}(X,Y)$$
is a bijection. In other words, the truncation functor \(\tau_{\geq n} \colon \mathsf{Ho}(\mathbf{Ch}(\mathcal{E})) \to \mathsf{Ho}(\mathbf{Ch}(\mathcal{E})_{\geq n})\) is  right adjoint to the inclusion \(\mathsf{Ho}(\mathbf{Ch}(\mathcal{E})_{\geq n}) \subseteq \mathsf{Ho}(\mathbf{Ch}(\mathcal{E})\). 
\end{lemma}

\begin{proof}
Consider a morphism in the derived category of \(\mathcal{E}\) between \(X\) and \(Y\). It is represented by a zig-zag
\begin{displaymath}
\xymatrix{
 & C\ar[dl]^{s}_{\sim}\ar[dr]^{f} & \\
 X & & Y
}
\end{displaymath}
with $s$ a quasi-isomorphism. Therefore $C$ is also in $\mathbf{Ch}_{\ge n}(\mathcal{E})$. Consider the diagram
\begin{displaymath}
\xymatrix{
 & C\ar[dl]^{s}_{\sim}\ar[dr]^{f} & \\
 X & \tau^{L}_{\ge n}C\ar[u]^{j}_{\sim}\ar[d]^{Id} & Y\\
  &\tau^{L}_{\ge n}C\ar[ul]^{s\circ j}\ar[ur]_{f\circ j} & 
}
\end{displaymath}
Thus from the outset we may assume that $C_{k}=0$ for $k<n$. Then the map $f_{n}$ factors through $\ker(d_{n}^{Y})$, i.e. $f$ factors through $\tau^{L}_{\ge n}Y$. It follows that the map $\Hom_{D(\mathcal{E})}(X,\tau_{\ge n}Y)\rightarrow \Hom_{D(\mathcal{E})}(X,Y)$ is surjective.

Let $\psi:X\rightarrow\tau_{\ge n}Y$ be a map such that $i\circ\psi=0$. The map $\psi$ is once again represented by a zig-zag 
\begin{displaymath}
\xymatrix{
 & C\ar[dl]^{t}_{\sim}\ar[dr]^{g} & \\
 X & & \tau_{\ge n}Y
}
\end{displaymath}
where $C_{k}=0$ for $k<n$. Since $i\circ\psi=0$, we have a diagram 
\begin{displaymath}
\xymatrix{
 & C\ar[dl]^{t}_{\sim}\ar[dr]^{i\circ g} & \\
 X & D\ar[d]\ar[u]^{j} & Y\\
  & X\ar[ul]^{Id}\ar[ur]^{0}
}
\end{displaymath}
Now $D$ is in $\mathbf{Ch}_{\ge n}(\mathcal{E})$, so again as before we may assume that $D_{k}=0$ for $k<n$. $a=i\circ g\circ j=0$ in the homotopy category of complexes $\mathcal{K}(\mathcal{E})$, i.e. $a$ is homotopic to zero, vis a homotopy $h$: $D\rightarrow Y$. It is easy to check that $h$ in fact factors through $\tau^{L}_{\ge n}Y$. Thus $g\circ j$ is homotopic to $0$, and $\psi=0$ in $\mathcal{K}(\mathcal{E})$.
\end{proof}

\begin{theorem}\label{thm:projectivetstructure}
If kernels exist in $\mathcal{E}$, then the projective $t$-structure defines a $t$-structure on $\mathbf{Ch}(\mathcal{E})$. Moreover, $\mathcal{E}$ is a full subcategory of $\mathbf{Ch}(\mathcal{E})^{\heart}$.
\end{theorem}

\begin{proof}
For $P$ projective the functor $\Hom(P,-)$ commutes with shifts, so clearly if $X$ is in $\mathbf{Ch}_{\ge0}(\mathcal{E})$ then $X[-1]\in \mathbf{Ch}_{\ge0}(\mathcal{E})$. 

By Lemma \ref{lem:truncation_inclusion_adjunction}, we have
a bijection
$$Hom_{D(\mathcal{E})}(X,\tau^{L}_{\ge 0}Y[1])\rightarrow Hom_{D(\mathcal{E})}(X,Y[1])$$
But $\tau_{\ge 0}^{L}Y[1]\cong 0$ since $Y\in\mathbf{Ch}_{\le0}(\mathcal{E})$. Thus $\Hom_{D(\mathcal{E})}(X,Y[1])=\{0\}$.

Consider the distinguished triangle
$$\tau_{\ge 0}^{L}(X)\rightarrow X\rightarrow \mathsf{cone}(i_{0})\rightarrow \tau_{\ge 0}^{L}(X)[-1]$$
Applying $\Hom(P,-)$ to this distinguised triangle gives the distinguished triangle 
$$\tau^{L}_{\ge 0}Hom(P,X)\rightarrow \Hom(P,X)\rightarrow \mathsf{cone}(\Hom(P,i_{0}))\rightarrow \tau_{\ge 0}^{L}\Hom(P,X)[-1]$$
Clearly $\tau^{L}_{\ge 0}\Hom(P,X)$ is in $\mathbf{Ch}_{\ge0}(\mathsf{Mod}_{\mathbb{Z}})^P$ and $\mathsf{cone}(\Hom(P,i_{0}))$ is in $\mathbf{Ch}_{\le0}(\mathsf{Mod}_{\mathbb{Z}})^P$. Thus $\tau^{L}_{\ge n}(X)$ is in $\mathbf{Ch}_{\ge0}(\mathcal{E})$ and $\mathsf{cone}(i_{0})$ is in $\mathbf{Ch}_{\le-1}(\mathcal{E})$

Clearly an object $X$ of $\mathcal{E}$, regarded as a complex concentrated in degree $0$, is in $\mathbf{Ch}_{\ge0}(\mathcal{E})\cap \mathbf{Ch}_{\le0}(\mathcal{E})$.
\end{proof}

\begin{proposition}
Let $\mathcal{P}$ be a generating subcategory of projectives in $\mathcal{E}$, and let $X\in\mathbf{Ch}(\mathcal{E})$. Then $X\in \mathbf{Ch}_{\ge0}(\mathcal{E})$ (resp. $\mathbf{Ch}_{\le0}(\mathcal{E})$) if and only if for any $P\in\mathcal{P}$, $\Hom(P,X)\in\mathbf{Ch}_{\ge0}(\mathsf{Mod}_{\mathbb{Z}})^P$ (resp. $\mathbf{Ch}_{\le0}(\mathsf{Mod}_{\mathbb{Z}})^P$)
\end{proposition}

\begin{proof}
Clearly if $X\in \mathbf{Ch}_{\ge0}(\mathcal{E})$ (resp. $\mathbf{Ch}_{\le0}(\mathcal{E})$) then $\Hom(P,X)\in\mathbf{Ch}_{\ge0}(\mathsf{Mod}_{\mathbb{Z}})^P$ (resp. $\mathbf{Ch}_{\le0}(\mathsf{Mod}_{\mathbb{Z}})^P$) for any $P\in\mathcal{P}$. 

Conversely, suppose $\Hom(P,X)\in \mathbf{Ch}_{\ge0}(\mathsf{Mod}_{\mathbb{Z}})^P$ (resp. $\mathbf{Ch}_{\le0}(\mathsf{Mod}_{\mathbb{Z}})$) for any $P\in\mathcal{P}$. Let $Q$ be an arbitrary projective. Then $Q$ is a summand of a coproduct of objects in $\mathcal{P}$. Since $\mathbf{Ch}_{\ge0}(\mathsf{Mod}_{\mathbb{Z}})^P$ and $\mathbf{Ch}_{\le0}(\mathsf{Mod}_{\mathbb{Z}})^P$ are closed under products and summands, the claim follows.
\end{proof}

\begin{remark}
The heart of this $t$-structure is the full subcategory consisting of objects $X_{\bullet}$ such that for any $P$, $\Hom(P,X_{\bullet})$ is in $\mathbf{Ch}(\mathsf{Mod}_{\mathbb{Z}})^{\heart}\cong \mathsf{Mod}_{\mathbb{Z}}$. It contains $\mathcal{E}$ as a full subcategory, regarded as complexes concentrated in degree $0$, but in general these categories do not coincide. The heart is an \textit{abelianisation} of $\mathcal{E}$ in the sense of \cite{kelly2016homotopy}. If $\mathcal{E}$ is a quasi-abelian category, then the heart is the \textit{left heart} in the sense of \cite{qacs}*{1.2}.
\end{remark}

\subsubsection{Properties of the Projective $t$-Structure}

Here we record some important properties of the projective $t$-structure on an elementary exact category. First we give a model category presentation of the non-negative part of the $t$-structure.

\begin{proposition}
Let $\mathcal{E}$ be a quasi-elementary exact category with kernels. Then $\mathbf{Ch}_{\ge0}(\mathcal{E})$ is equivalent to $\mathrm{L}^{H}(\mathrm{Ch}_{\ge0}(\mathcal{E}))$, where $\mathrm{Ch}_{\ge0}(\mathcal{E})$ is equipped with the projective model structure.
\end{proposition}

\begin{proof}
There is a Quillen adjunction
$$\adj{i_{\ge0}}{\mathrm{Ch}_{\ge0}(\mathcal{E})}{\mathrm{Ch}(\mathcal{E})}{\tau^{L}_{\ge0}}$$
This realises $\mathrm{L}^{H}(\mathrm{Ch}_{\ge0}(\mathcal{E}))$ as a coreflective subcategory of $\textbf{Ch}(\mathcal{E})$. It remains to check that the essential image of $i_{\ge0}$ contains $\textbf{Ch}_{\ge0}(\mathcal{E})$. However this is immediate from the proof of Theorem \ref{thm:projectivetstructure}, which shows that $\tau^{L}_{\ge0}$ is the truncation functor for the projective $t$-structure. 
\end{proof}

\begin{proposition}
Let $\mathcal{E}$ be a quasi-elementary exact category with kernels. Then any projective object of $\mathcal{E}$ is projective as an object of $\mathrm{L}^{H}({Ch}_{\ge0}(\mathcal{E}))$. Moreover any object $X_{\bullet}$ of $\mathrm{Ch}_{\ge0}(\mathcal{E})$ may be written as a homotopy sifted colimit of objects of $\mathpzc{C}^{0}$. 
\end{proposition}

\begin{proof}
Let $P$ be a projective object. First note that $\mathbb{R}\mathrm{Hom}_{\bullet}(P,-)\cong\mathrm{Hom}_{\bullet}(P,-)$ as functors since $P$ is projective. Let $Z:\Delta^{op}\rightarrow\mathrm{Ch}_{\ge0}(\mathcal{E})$ be a functor. Consider the Dold-Kan equivalence
$$\adj{N^{\mathrm{Ch}_{\ge0}(\mathcal{E})}}{\mathpzc{Fun}(\Delta^{op},\mathrm{Ch}_{\ge0}(\mathcal{E}))}{\mathrm{Ch}_{\ge0}(\mathrm{Ch}_{\ge0}(\mathcal{E}))}{\Gamma^{\mathrm{Ch}_{\ge0}(\mathcal{E})}}$$

By checking on generating cofibrations for the projective model structure on $\mathpzc{Fun}(\Delta^{op},\mathrm{Ch}_{\ge0}(\mathcal{E}))$, we have an equivalence
$$\mathrm{hocolim}_{\Delta^{op}}Z\cong\mathrm{Tot}(N^{\mathrm{Ch}_{\ge0}(\mathcal{E})}(Z))$$
Now an easy computation shows that for $P$ projective we have natural isomorphisms
$$\mathrm{Hom}_{\bullet}(P,N^{\mathrm{Ch}_{\ge0}(\mathcal{E})}(Z))\cong N^{\mathrm{Ch}_{\ge0}(\mathpzc{Ab})}(\underline{\mathrm{Hom}}_{\Delta^{op}}(P,Z))$$
Thus to show that $P$ is projective in $\mathpzc{C}_{\ge0}$ it suffices to show that $\mathrm{Hom}_{\bullet}(P,-)$ commutes with totalisations of double complexes. This is clear.

Now let $X_{\bullet}$ be a complex, and $P_{\bullet}\rightarrow X_{\bullet}$ an equivalence with $P_{\bullet}$ a complex of projectives such that each $P_{n}$ is a coproduct of objects of $\mathpzc{C}^{0}$. Then $X_{\bullet}$ is the homotopy colimit of the simplicial object $\Gamma^{\mathcal{E}}(P_{\bullet})$ and each $\Gamma_{n}^{\mathcal{E}}(P_{\bullet})$ is a coproduct, in particular a homotopy sifted colimit, of objects of $\mathpzc{C}^{0}$. This completes the proof. 
\end{proof}

Next we show that the non-positive part is closed under filtered colimits. For this result we need $\mathcal{E}$ to be elementary.

\begin{proposition}
Let $\mathcal{E}$ be an elementary exact category. Then $\mathbf{Ch}_{\le0}(\mathcal{E})$ is closed under filtered colimits. %But then \(\mathbf{Ch}_{\leq 0}(\mathsf{CBorn}_R)\) is not.
\end{proposition}

\begin{proof}
Let $\mathcal{P}$ be a generating set consisting of tiny projectives. For $P\in\mathcal{P}$, the functor $\Hom(P,-):\mathbf{Ch}(\mathcal{E})\rightarrow\mathbf{Ch}(\mathsf{Mod}_{\mathbb{Z}})$ commutes with filtered colimits. Since $\mathbf{Ch}_{\le0}(\mathsf{Mod}_{\mathbb{Z}})^P$ is closed under filtered colimits, so is $\mathbf{Ch}_{\le0}(\mathcal{E})$.
\end{proof}

\begin{proposition}\label{prop:postnikovexactrightcomplete}
Let $\mathcal{E}$ be a quasi-elementary exact category. Then the projective $t$-structure on $\mathbf{Ch}(\mathcal{E})$ is right complete. 
\end{proposition}

\begin{proof}
Let $X$ be an object of $\mathbf{Ch}(\mathcal{E})$. As shown in \cite{kelly2016homotopy} Corollary 2.3.67, $X$ has a resolution $P\rightarrow X$ where $P$ is a complex of projectives and is $DG$-projective. Moreover by the proof of Corollary 2.3.67 in \cite{kelly2016homotopy}, $P$ can be written as a colimit $\colim_{n\in\mathbb{N}}P_{n}$ where $P_{n}\rightarrow P_{n+1}$ induces an isomorphism $P_{n}\rightarrow\tau^{L}_{\ge n}P_{n+1}$. This shows that $\mathbf{Ch}(\mathcal{E})$ is right complete. 
\end{proof}

\begin{lemma}\label{lem:compatiblet}
Let $\mathcal{E}$ be a quasi-projectively monoidal elementary exact category. The projective $t$-structure on $\mathbf{Ch}(\mathcal{E})$ is compatible with the tensor product. 
\end{lemma}

\begin{proof}
We have already shown that $\mathbf{Ch}_{\le0}(\mathcal{E})$ is closed under filtered colimits. Moreover the unit $k\in\mathcal{E}\subset\mathbf{Ch}_{\ge0}(\mathcal{E})$. Now let $X,Y\in\mathbf{Ch}_{\ge0}(\mathcal{E})$. Then there are cofibrant resolutions $P\rightarrow X$, $Q\rightarrow Y$, where $P$ and $Q$ are both concentrated in degrees $\ge0$. The tensor product of $X$ and $Y$ in $\mathbf{Ch}(\mathcal{E})$ is computed by $P\otimes Q$, which is concentrated in degrees $\ge0$, i.e. is in $\mathbf{Ch}_{\ge0}(\mathcal{E})$. 
\end{proof}

\subsection{Derived Algebraic Contexts from Exact Categories}

In this section we prove our main results regarding the construction of derived algebraic contexts from exact categories.

\begin{theorem}
Let $\mathcal{E}$ be a quasi-monoidal elementary exact category with symmetric projectives such that the tensor product of two compact objects is compact. Let $\mathcal{P}^{0}$ be a set of tiny projective generators closed under finite direct sums, tensor products, and the formation of symmetric powers. Then, with $\mathrm{Ch}(\mathcal{E})$ equipped with the projective model structure and the projective t-structure, $(\mathrm{Ch}(\mathcal{E}),\mathrm{Ch}_{\ge0}(\mathcal{E}),\mathrm{Ch}_{\le0}(\mathcal{E}),\mathcal{P}^{0})$ is a model derived algebraic context.
\end{theorem}

\begin{proof}
By Theorem \ref{thm:projectivetstructure} the projective $t$-structure exists on $\mathrm{Ch}(\mathcal{E})$ and by Lemma \ref{lem:compatiblet} the monoidal structure is compatible with the  projective $t$-structure. By Corollary \ref{prop:postnikovexactrightcomplete} the projective $t$-structure is right complete. By construction $\mathcal{P}^{0}$ is a set of tiny projective generators in $\mathrm{Ch}_{\ge0}(\mathcal{E})$, consists of fibrant-cofibrant objects which are projective generators in $\mathrm{Ch}_{\ge0}(\mathcal{E})$, and is closed under taking tensor products and symmetric powers in $\mathrm{Ch}(\mathcal{E})$.
\end{proof}

Let $(\mathbf{C},\mathbf{C}_{\ge0},\mathbf{C}_{\le0},\mathbf{C}^{0})$ be a derived algebraic context and consider the heart $\mathbf{C}^{\heart}$. By assumption the objects of $\mathbf{C}^{0}$ form a set of tiny projective generators of $\mathbf{C}^{\heart}$, so it is elementary. Moreover, again by assumption, it inherits a symmetric monoidal structure from $\mathbf{C}$ such that $\mathbf{C}^{0}$ forms a symmetrically closed projective generating set. Hence $\mathbf{C}^{\heart}$ is a quasi-monoidal elementary exact category, and therefore
$$(\mathbf{Ch}(\mathbf{C}^{\heart}),\mathbf{Ch}_{\ge0}(\mathbf{C}^{\heart}),\mathbf{Ch}_{\le0}(\mathbf{C}^{\heart}),\mathcal{P}^{0})$$
is a derived algebraic context. The functor
$$\mathbf{C}^{\heart}\rightarrow\mathbf{Ch}(\mathbf{C}^{\heart})$$
induces a sifted colimit-preserving strongly monoidal functor
$$\mathbf{C}\cong\mathcal{P}_{\Sigma}(\mathbf{C}^{0})\rightarrow\mathbf{Ch}_{\ge0}(\mathbf{C}^{\heart})$$
and by stabilisation a strongly monoidal functor
$$\mathbf{C}\rightarrow\mathbf{Ch}(\mathbf{C}^{\heart}).$$ Since this functor is the identity on $\mathbf{C}^{0}$, it defines an equivalence of derived algebraic contexts. In particular every derived algebraic context is equivalent to 
$$(\mathbf{Ch}(\mathbf{C}^{\heart}),\mathbf{Ch}_{\ge0}(\mathbf{C}^{\heart}),\mathbf{Ch}_{\le0}(\mathbf{C}^{\heart}),\mathcal{P}^{0}).$$

\begin{remark}\label{rem:classificationdac}
Let $(\mathbf{C},\mathbf{C}_{\ge0},\mathbf{C}_{\le0},\mathbf{C}^{0})$ be a derived algebraic context. Then
$$(\mathbf{Ch}(\mathbf{C}^{\heart}),\mathbf{Ch}_{\ge0}(\mathbf{C}^{\heart}),\mathbf{Ch}_{\le0}(\mathbf{C}^{\heart}),\mathcal{P}^{0})$$
is a derived algebraic context which is equivalent to $(\mathbf{C},\mathbf{C}_{\ge0},\mathbf{C}_{\le0},\mathbf{C}^{0})$. 
\end{remark}

\subsubsection{Relation to HA Contexts}

A consequence of Remark \ref{rem:classificationdac} is that a large class of derived algebraic contexts are modeled by HA contexts in the sense of \cite{toen2008homotopical}*{Definition 1.1.0.11}. We will give a slightly truncated definition here (for the category $\mathsf{C}_{0}$ in \cite{toen2008homotopical}, we always take $\mathsf{C}=\mathsf{C}_{0}$).

\begin{definition}\label{def:HA context}
Let \(\mathpzc{M}\) be a combinatorial, symmetric, monoidal model category \(\mathpzc{M}\). We call \(\mathpzc{M}\) a \textit{homotopical algebra context} (or, \textit{HA}-context) if it satisfies the following properties:

\begin{itemize}
\item the  model  category \(\mathpzc{M}\) is  proper,  pointed  and  for  any  two  objects \(X\) and \(Y\) in \(\mathpzc{M}\), the  natural morphisms \[QX \coprod QY \to X \coprod Y \to RX \times RY\] are weak equivalences, where \(Q\) and \(R\) are the cofibrant and fibrant replacement functors, respectively;
\item \(\mathsf{Ho}(\mathpzc{M})\) is an additive category;
\item with the transferred model structure and monoidal structure \(\otimes_A\) for some commutative algebra object \(A \in \mathsf{CAlg}(\mathpzc{M})\), the category \(\mathpzc{Mod}_A\) is a combinatorial, proper, symmetric, monoidal model category;
\item for any cofibrant object \(M \in \mathpzc{Mod}_A\), the functor \[ - \otimes_A M \colon \mathpzc{Mod}_A \to \mathpzc{Mod}_A\] preserves weak equivalences;
\item with the transferred model structures the categories \(\mathsf{CAlg}(\mathpzc{Mod}_A)\) and \(\mathsf{CAlg}^{nu}(\mathpzc{Mod}_A)\) are combinatorial, proper model categories (where  \(\mathsf{CAlg}^{nu}(\mathpzc{Mod}_A)\) is the category of \textit{non-unital commutative monoids};
\item if \(B\) is cofibrant in \(\mathsf{CAlg}(\mathpzc{Mod}_A)\), then the functor \[B \otimes_A - \colon \mathpzc{Mod}_A \to \mathpzc{Mod}_B\] preserves weak equivalences. 
\end{itemize}
\end{definition}

The exact categories in which we are interested are HA contexts. 

\begin{theorem}\label{theorem:exact_HAC}
Let \((\mathcal{E}, \otimes, k, \underline{\mathsf{Hom}})\) be a locally presentable, closed, projectively monoidal exact category that is \(\mathbf{AdMon}\)-elementary and enriched over $\mathbb{Q}$. Then \(\mathbf{Ch}(\mathcal{E})\) and \(\mathbf{Ch}_{\geq 0}(\mathcal{E})\) are HA-contexts.  
\end{theorem}
\begin{proof}
See \cite{kelly2016homotopy}*{Theorem 6.9}.
\end{proof}

Even in the algebraic case complexes of abelian groups do not form an HA context - the transferred model structure does not exist on commutative dg-rings. One can remedy this by instead considering simplicial objects. Indeed, it will be shown in \cite{BKK} that for  \((\mathcal{E}, \otimes, k, \underline{\mathsf{Hom}})\) a locally presentable, closed, monoidal elementary exact category with symmetric projectives, $\mathsf{s}\mathcal{E}$ is an HA-context.

If $(\mathbf{C},\mathbf{C}_{\ge0},\mathbf{C}_{\le0},\mathbf{C}^{0})$ is a derived algebraic context such that all objects of $\mathbf{C}^{0}$ are flat in $\mathbf{C}^{\heart}$, then in particular $\mathbf{C}^{\heart}$ is a monoidal elementary abelian category. Therefore, $\mathbf{C}_{\ge0}$ will be presented by the HA context $\mathrm{s}\mathbf{C}^{\heart}$. However, it can happen that $\mathcal{E}$ is a monoidal elementary exact category, but $\mathbf{Ch}(\mathcal{E})^{\heart}$ is not (for example if projectives in $\mathcal{E}$ are flat but not strongly flat), and $\mathbf{Ch}_{\ge0}(\mathbf{C}^{\heart})\cong\mathbf{Ch}_{\ge0}(\mathcal{E})$ is still presented by an HA context (namely $\mathrm{s}\mathcal{E})$.

\subsection{Main Example: Bornological Modules}
In what follows, we show that the categories of relevance in analytic geometry as introduced in Section \ref{sec:bornologies}, are strong monoidal $\textbf{AdMon}$-elementary quasi-abelian categories. This includes the categories \(\mathsf{CBorn}_R\) and \(\mathsf{Ind}(\mathsf{Ban}_R)\) of complete bornological and inductive systems of Banach \(R\)-modules, where \(R\) is any Banach ring (with the latter of course being elementary). Most of these claims have been proven in the case where \(R = \C\), or a Banach field, and the proofs carry over also to the integral case. 

\begin{lemma}\label{lem:bornologies_exact}
Let \(R\) be a Banach ring. Then the categories \(\mathsf{NMod}_R^{1/2}\), \(\mathsf{NMod}_R\), \(\mathsf{Ban}_R\), \(\mathsf{Born}_R\), \(\mathsf{CBorn}_R\) are all  quasi-abelian categories. Furthermore, the \(\mathsf{Ind}\)-completions  \(\mathsf{Ind}(\mathsf{NMod}_R^{1/2})\), \(\mathsf{Ind}(\mathsf{NMod}_R)\) and \(\mathsf{Ind}(\mathsf{Ban}_R)\) are  quasi-abelian categories.
\end{lemma}

\begin{proof}
By Proposition \ref{lem:bornologies_bicomplete} and Proposition \ref{prop:ind(E)_exact} it suffices to prove that \(\mathsf{NMod}_R^{1/2}\), \(\mathsf{NMod}_R\), \(\mathsf{Ban}_R\) are quasi-abelian.

We first prove that \(\mathsf{NMod}_R^{1/2}\) is quasi-abelian. Given a bounded \(R\)-linear map \(f \colon V \to W\), its kernel is given by the \(R\)-module \(f^{-1}(0)\) with the subspace topology induced by the norm on \(V\), and the canonical inclusion into \(V\). The cokernel of \(f\) is given by the \(R\)-module \(W/\overline{f(V)}\) with the quotient norm and the map \(W \onto W/\overline{f(V)}\). Now suppose \(f\) is a strict monomorphism and \(g \colon V \to V'\) is an arbitrary bounded \(R\)-module map. Their pushout is given by the \(R\)-module \(W' \defeq \coker(V \overset{(g,-f)}\to V' \oplus W)\), equipped with the quotient semi-norm induced by \(V' \oplus W\). The canonical morphisms are given by the compositions \(V' \to V' \oplus W \to W'\) and \(W \to W \oplus V' \to W'\). We claim that the composed map \(f' \colon V' \to W'\) is a strict monomorphism - that is, an injective \(R\)-module map with closed image. To see injectivity, we note that the map \(f'\) is given explicitly by \(f'(v') = [(v',0)]\), where \([\cdot,\cdot]\) denotes the equivalence class of \((v',0)\) under the quotient map \(V' \oplus W \to W'\). So if \(f'(v') = 0\), there is a \(v \in V\) such that \(g(v) = v'\) and \(f(v)=0\). Since \(f\) is in particular injective, \(v=0\) and hence \(v' = 0\).

We now show that the map \(f'\) has closed image. Let \(v' \in V'\). The strictness and the boundedness of \(g\) implies that there are positive constants \(C\), \(C'\) and \(C''\) such that 

\begin{multline*}
\norm{v'}_{V'} \leq \norm{v' + g(v)}_{V'} + \norm{g(v)}_{V'} 
\leq \norm{v' + g(v)}_{V'} + C \norm{v} \\
 \leq \norm{v' + g(v)}_{V'} + C \norm{v}_V \leq \norm{v' + g(v)}_{V'} + C' \norm{f(v)}_W \leq C'' \norm{(v' + g(v), - f(v))}_{V' \oplus W}.
\end{multline*}

Since \(v'\) was arbitrary, we have 

\begin{multline*}
\norm{v'}_{V'} \leq C'' \underset{v \in V}\inf \norm{(v'+g(v),-f(v))}_{V' \oplus W} \leq C'' \underset{(x,y) \in \mathsf{im}(g,-f)}\inf \norm{(v'+ x,y)}_{V' \oplus W} \\
\leq C'' \norm{[(v',0)]}_{W'} \leq C'' \norm{f'(v')}_{W'}. 
\end{multline*}

What we have therefore shown is that \(\mathsf{NMod}_R^{1/2}\) is quasi-abelian. The categories \(\mathsf{NMod}_R\) and \(\mathsf{Ban}_R\) are quasi-abelian since these are full, additive subcategories of \(\mathsf{NMod}_R^{1/2}\), that are closed under kernels and cokernels by strict maps.
\end{proof}

Given a Banach ring \(R\) and \(r>0\), let \(R_r\) denote the Banach ring whose underlying ring is \(R\), and whose norm is given by \(\norm{x}_r \defeq r \norm{x}\), \(x \in R\). Similarly, if \((M,\norm{\cdot})\) is a (semi)-normed or Banach \(R\)-module, we denote by \(M_r\), the module \(M\) with semi-norm \(\norm{m}_r \defeq  r \norm{m}\). It is a norm if \(M\) is normed, and it is complete in this norm if \(M\) is complete in its norm. In order to talk about projective objects in our categories of interest, we introduce the following \textit{contracting} categories: let \(\fC\) be any of the categories \(\mathsf{NMod}_R^{1/2}\), \(\mathsf{NMod}_R\) or  \(\mathsf{Ban}_R\). We define \(\fC^{\leq 1}\) as the category with the same objects as \(\fC\), but with those morphisms \(f \colon M \to N\) in \(\fC\) that satisfy \(\norm{f(m)}_N \leq \norm{m}_M\). The restriction to contracting morphisms ensures the following:

\begin{lemma}\label{lem:contracting}
The categories \(\fC^{\leq 1}\) have infinite coproducts and infinite products.
\end{lemma}

\begin{proof}
Let \((M_i, \norm{\cdot}_i)_{i \in I}\) be an arbitrary collection of objects in \(\mathsf{NMod}_R^{1/2}\). In both the Archimedean and non-Archimedean case ,their product in \( (\mathsf{NMod}_R^{1/2})^{\leq 1}\) is given by \(\prod_{i \in I}^{\leq 1} M_i = \setgiven{(m_i) \in \prod_{i \in I} M_i}{\sup_{i \in I} \norm{m_i}_i<\infty}\) with the supremum norm \(\norm{(m_i)} =  \sup_{i \in I} \norm{m_i}_i\). Also in both cases the underlying $R$-module of the coproduct is given by the usual coproduct of $R$-modules. In the Archimedean case the norm is given by the \(l^1\)-norm \(\norm{(m_i)} =  \sum_{i \in I} \norm{m_i}_i\), and the in the non-Archimedean case it is given by \(\norm{(m_i)} =  \max_{i \in I} \norm{m_i}_i\). The constructions in the normed categories are the same, and in the Banach categories are given by completing the constructions in \(\mathsf{NMod}_R^{\le1}\). 
\end{proof}

Since we will need these contracting coproducts throughout the article, we fix some notation: for a normed set \(X\), we denote by \(l^1(X, R)\) the contracting coproduct in any of the categories in Lemma \ref{lem:contracting}. Concretely, \(l^1(X,R) = \coprod_{x \in X}^{\leq 1} R_{\norm{x}_X}\) with the \(l^1\)-norm \(\norm{(r_x)_{x\in X}}= \sum_{x \in X} \norm{r_x}_R \norm{x}_X\) in the Archimedean case, and $||(r_{x})_{x\in X}=\mathrm{sup}_{x\in X}||r_{x}||_{R}||x||_{X}$ in the non-Archimedean case.

The coproducts in the contracting categories \(\fC^{\leq 1}\) are crucial in the proof of the following, which is  \cite{ben2020fr}*{Lemma 3.35, Lemma 3.36} 

%allows us to prove the following, which has already appeared in the case \(\mathsf{Ban}_R\) and \(\mathsf{CBorn}_R\) in \cite{ben2020fr}:

%\begin{lemma}\label{lem:bornologies_enough_projectives}
%%The categories \(\mathsf{NMod}_R^{1/2}\), \(\mathsf{NMod}_R\),  \(\mathsf{Ban}_R\)
%%
%% and \(\mathsf{CBorn}_R\) have 
%The category \(\mathsf{Ban}_R\) has enough strongly flat projectives. Moreover the tensor product of two projectives is projective.
%\end{lemma}
%
%\begin{proof}
%For any semi-normed \(R\)-module \(M\), there is a strict epimorphism \(\coprod_{m \in M}^{\leq 1} R_{\norm{m}} \onto M\). 
%What follows is essentially \cite{ben2020fr}*{Lemma 3.35, Lemma 3.36} it it shown that projectives are flat in $\mathsf{Ban}_{R}$ and that the tensor product of two projects is projective. 
%\end{proof}

\begin{corollary}\label{lem:CBorn_elementary}
The category \(\mathsf{Ind}(\mathsf{Ban}_R)\) is a strongly monoidal elementary exact category, and \(\mathsf{CBorn}_R\) is a strongly monoidal $\mathbf{AdMon}$-elementary exact category.
\end{corollary}

%
%\begin{example}\label{ex:CBorn_projectively_mon}
%The categories \(\mathsf{Ban}_R\) and \(\mathsf{CBorn}_R\) are strongly projectively monoidal since the complete projective tensor product commutes with cokernels. FIX: COMBINE THIS WITH PROPOSITION LATER.
%\end{example}

%\begin{corollary}\label{cor:IndBan_projectively_mon}
%The category \(\mathsf{Ind}(\mathsf{Ban}_R)\) is a strong projectively monoidal (elementary) exact category. 
%\end{corollary}

\begin{corollary}
Let $R$ be a Banach ring. Then
$$\mathbf{C}^{IB}_{R}\defeq(\mathbf{Ch}(\mathrm{Ind(Ban_{R})}),\mathbf{Ch}_{\ge0}(\mathrm{Ind(Ban)}_{R}),\mathbf{Ch}_{\ge0}(\mathrm{Ind(Ban)}_{R}),\mathcal{P}^{0})$$
is a derived algebraic context. 
\end{corollary}

\begin{remark}
Consider now the category $\mathrm{CBorn}_{R}$. There is a Quillen equivalence 
$$\mathrm{Ch}(\mathrm{CBorn}_{R})\cong\mathrm{Ch}(\mathrm{Ind(Ban_{R})})$$ 
which restricts to a Quillen equivalence 
$$\mathrm{Ch}_{\ge0}(\mathrm{CBorn}_{R})\cong\mathrm{Ch}_{\ge0}(\mathrm{Ind(Ban_{R})})$$ 
Via this equivalence we may transport the $t$-structure on $\mathbf{Ch}(\mathrm{Ind(Ban_{R})})$ to one on $\mathbf{Ch}(\mathrm{CBorn}_{R})$, and the non-negative part of this transferred $t$-structure coincides with the non-negative part of the projective $t$-structure on $\mathbf{Ch}(\mathrm{CBorn}_{R})$. However the transferred non-positive part, which we denote by $\widetilde{\mathbf{Ch}}_{\ge0}(\mathrm{CBorn_{R}})$ will be different. This gives us a derived algebraic context
$$\mathbf{C}^{CB}_{R}\defeq (\mathbf{Ch}(\mathrm{CBorn_{R}}),\mathbf{Ch}_{\ge0}(\mathrm{CBorn_{R}}),\widetilde{\mathbf{Ch}}_{\ge0}(\mathrm{CBorn_{R}}),\mathcal{P}^{0})$$
which is equivalent to $\mathbf{C}^{IB}_{R}$. 
\end{remark}
% It is likely that one can consider symmetric spectra $\mathsf{Sp}^\Sigma(\mathcal{E})$ in this category to obtain another HA-context, such that $(\mathsf{s}\mathcal{E},\mathsf{Sp}(\mathcal{E}))$ is a stable HA-context.

%\begin{corollary}
%Let $(\mathbf{C},\mathbf{C}_{\ge0},\mathbf{C}_{\le0},\mathbf{C}^{0})$ be a derived algebraic context. 
%\begin{enumerate}
%\item
%The category $\mathrm{DAlg}^{cn}(\mathbf{C})$ is presented by the model category $\mathpzc{sComm}(\mathbf{C}^{\heart})$ of simplicial commutative algebras in $\mathbf{C}^{\heart}$
%
%\end{enumerate}
%\end{corollary}
%
%
%
%\begin{corollary}
%If $(\mathbf{C},\mathbf{C}_{\ge0},\mathbf{C}_{\le0},\mathbf{C}^{0})$ is a rational derived algebraic context then the map
%$$\mathsf{Sym}_{\mathbf{C}}\rightarrow\mathsf{LSym}_{\mathbf{C}}$$
%is an equivalence of monads. In particular the functor
%$$\Theta:\mathsf{DAlg}(\mathbf{C})\rightarrow\mathsf{CAlg}(\mathbf{C})$$
%is an equivalence.
%\end{corollary}

\section{Derived de Rham cohomology and filtered Hochschild homology}\label{sec:HKR}

In this section, we use the universal properties of derived de Rham cohomology and Hochschild homology, to prove a version of the Hochschild-Kostant-Rosenberg Theorem for derived \textit{bornological} algebras, using the results of Section \ref{sec:exact}. We first recall the constructions of the categories where these complexes with their additional structures reside, leaving the reader to find all relevant details in \cite{raksit2020hochschild}.

\subsection{\(S^1\)-equivariant Hochschild homology}

Let \(R\) be a commutative, unital ring and let \(A\) be a flat, commutative \(R\)-algebra. The Hochschild homology \(\HH_*(A/R)\) of \(A\) with coefficients in \(R\) is defined as the homology of the chain complex \[ \mathsf{HH}(A/R) \defeq \cdots \to A^{\otimes_R n + 1} \overset{b_n}\to A^{\otimes_R n} \to \cdots \to A \otimes_R A \to A,\] where 

\begin{multline*}
b_n(a_0 \otimes \cdots \otimes a_n) = a_0a_1 \otimes \cdots a_n + \sum_{i=1}^{n-1}(-1)^i a_0 \otimes \cdots \otimes a_i a_{i+1} \otimes \cdots \otimes a_n \\
+ (-1)^n a_n a_0 \otimes \cdots \otimes a_{n-1}. 
\end{multline*}

By the Dold-Kan correspondence, this complex can equivalently be viewed as a simplicial \(R\)-module \(\mathsf{HH}(A/R) \defeq ([n] \mapsto A^{\otimes_R n+1})\), with face and degeneracy maps obtained from the Hochschild differential \(b_n \colon \mathsf{HH}(A/R)_n \to \mathsf{HH}(A/R)_{n-1}\) above. Since \(A\) is commutative, the face and degeneracy maps are actually \(R\)-algebra homomorphisms, so that \(\mathsf{HH}(A/R)\) is a simplicial, commutative \(R\)-algebra. Furthermore, we can view \(\mathsf{HH}(A/R)\) as a \textit{cyclic object} in the category \(\mathsf{CAlg}_R\) of commutative \(R\)-algebras, that is, as a functor \(\Lambda^\op \to \mathsf{CAlg}_R\) from Connes' cyclic category to the category of commutative \(R\)-algebras. By \cite{Loday}*{Theorem 7.1.4}, the geometric realisation \(\vert \mathsf{HH}(A/R) \vert\) has a canonical \(S^1\)-action. Applying the nerve construction, and using that \(\vert N(\Lambda) \vert \cong BS^1\),  we can view \(\mathsf{HH}(A/R)\) as an object of the \(\infty\)-category \(\mathsf{Fun}(BS^1, \mathbf{Ch}(\mathsf{Mod}_R))\) of chain complexes of \(R\)-modules with an \(S^1\)-action.

\begin{remark}
We can drop the assumption that \(A\) is flat by using a simplicial resolution \(P_\bullet \to A\) of flat \(R\)-algebras, and defining \(\mathsf{HH}(A/R) \defeq \vert \mathsf{HH}(P_\bullet/R) \vert\). More formally, it is the left Kan extension of the functor \(\mathsf{Poly}_R \ni A \mapsto \vert\mathsf{HH}(A/R)\vert \in \mathsf{Fun}(BS^1, \mathbf{Ch}(\mathsf{Mod}_R))\) along the inclusion \(\mathsf{Poly}_R \to \mathsf{sCAlg}_R\) of finitely generated polynomial algebras into simplicial commutative \(R\)-algebras.  
\end{remark}

Now suppose \(A\) is a (simplicial) commutative \(R\)-algebra. Then by the remark above, \(\mathsf{HH}(A/R)\) is an \(\mathbb{E}_\infty\)-\(R\)-algebra with an \(S^1\)-action. Here by \(\mathbb{E}_\infty\)-\(R\)-algebras, we mean the \(\infty\)-category \(\mathsf{CAlg}(\mathbf{Ch}(\mathsf{Mod}_R))\) of commutative monoids in the symmetric monoidal \(\infty\)-category of complexes \(\mathbf{Ch}(\mathsf{Mod}_R)\). The Hochschild homology algebra \(\mathsf{HH}(A/R)\) is the initial such object in the following sense:

\begin{theorem}\cite{BMS}*{Section 2.2}
Let \(X \in \mathsf{Fun}(BS^1, \mathsf{CAlg}(\mathbf{Ch}(\mathsf{Mod}_R))\) be an \(\mathbb{E}_\infty\)-\(R\)-algebra with a compatible \(S^1\)-action, and let \(A \to X\) be a morphism of \(\mathbb{E}_\infty\)-\(R\)-algebras. Then there exists a unique extension to an \(S^1\)-equivariant \(\mathbb{E}_\infty\)-\(R\)-algebra morphism \(\mathsf{HH}(A/R) \to X\).   
\end{theorem}

The main result in \cite{raksit2020hochschild} is the generalisation of the universal property above to the setting where \(\mathsf{HH}(A/R)\) is equipped with its \textit{HKR filtration} and correspondingly, a \textit{filtered \(S^1\)-action}. It is this generality we need in order to prove statements about global analytic spaces - that is, analytic spaces over a Banach ring that is not a field. 

In order to be able to do this, we first observe that we can rewrite a chain complex \(X \in \mathsf{Fun}(BS^1, \mathbf{Ch}(\mathsf{Mod}_\Z))\) as module over the group ring \(\Z[S^1]\). In general, let \(\mathbf{C}\) be any \(\Z\)-linear, stable, presentable, symmetric monoidal \(\infty\)-category. Denote by \(\mathbb{T}_{\mathbf{C}}\) the image of the group ring \(\Z[S^1]\) under the canonical symmetric monoidal functor \(\mathbf{Ch}(\mathsf{Mod}_\Z) \to \mathbf{C}\). Then there is an equivalence of \(\infty\)-categories \[\mathsf{Fun}(BS^1, \mathbf{C}) \cong \mathbf{Mod}_{\mathbb{T}_{\mathbf{C}}}(\mathbf{C}).\]

The group ring \(\mathbb{T}_{\mathbf{C}}\) is dualisable. To see this, we first note that \(\Z[S^1]\) is dualisable, where the dual \(\Z[S^1]^\vee\) is given by the colimit of the constant diagram ~\(S^1 \to \mathbf{Ch}(\mathsf{Mod}_\Z)\) with value \(\Z\).  Now denote by \(\mathbf{T}_{\mathbf{C}}^\vee\) the image of \(\Z[S^1]^\vee\) under the canonical map \(\mathbf{Ch}(\mathsf{Mod}_\Z) \to \mathbf{C}\). Furthermore, the group ring \(\Z[S^1]\) is a bicommutative bialgebra, where the bialgebra structure is induced by the group structure on \(S^1\). This induces a bicommutative bialgebra structure on \(\mathbb{T}_{\mathbf{C}}\), and hence on \(\mathbb{T}_{\mathbf{C}}^\vee\).

If \(\mathbf{C}\) is also a derived algebraic context, there is a \textit{derived} bicommutative, bialgebra structure on \(\mathbb{T}_{\mathbf{C}}^\vee\), promoting the bicommutative, bialgebra structure on \(\Z[S^1]^\vee\). This results in an equivalence of \(\infty\)-categories \[\Fun(BS^1, \mathsf{DAlg}_A(\mathbf{C})) \cong \mathsf{cMod}_{\mathbb{T}_{\mathbf{C}}^\vee}(\mathsf{DAlg}_A(\mathbf{C})),\] where \(A\) is a derived algebra object in \(\mathbf{C}\). The right hand side of the above equivalence denotes the \(\infty\)-category of \(\mathbb{T}_{\mathbf{C}}^\vee\)-comodule objects in \(\mathsf{DAlg}_A(\mathbf{C})\).

Now let \(\tau_{\geq *} \colon \mathbf{Ch}(\mathsf{Mod}_\Z) \to \mathbf{Fil}(\mathbf{Ch}(\mathsf{Mod}_\Z))\) be the \textit{Postnikov filtration functor}, equipping a chain complex \(X \in \mathbf{Ch}(\mathsf{Mod}_\Z)\) with its Postnikov filtration. The image of the Hochschild homology complex \(\mathsf{HH}(A/R)\) under the Postnikov filtration functor is called \textit{HKR-filtered Hochschild homology}. The images \(\Z[S^1]_\fil\) and \(\Z[S^1]_{\fil}^\vee\) of the group algebra \(\Z[S^1]\) and its dual under this functor define filtered algebras. The algebra \(\Z[S^1]_{\mathrm{fil}}\) is called the \textit{filtered circle}. These are bicommutative bialgebras, lifting the bicommutative bialgebra structures on \(\Z[S^1]\) and \(\Z[S^1]^\vee\). Similarly, there is also a \textit{derived} bicommutative bialgebra structure on \(\Z[S^1]_{\fil}^\vee\) that lifts the derived bicommutative bialgebra structure on \(\Z[S^1]^\vee\). 

Again, using the base change map \(\mathbf{Ch}(\mathsf{Mod}_\Z) \to \mathbf{C}\) into a \(\Z\)-linear, stable, presentable symmetric monoidal \(\infty\)-category as above, we can define the filtered circle \(\mathbb{T}_{\mathbf{C}, \mathrm{fil}}\) in \(\mathbf{C}\).  If additionally \(\mathbf{C}\) is a derived algebraic context, then given a derived commutative algebra object \(A\)  in \(\mathbf{C}\), we can define the \(\infty\)-category \[\mathsf{Fil}_{S^1}(\mathsf{DAlg}_A(\mathbf{C})) \defeq \mathsf{cMod}_{\Z[S^1]_{\fil}^\vee}(\mathsf{DAlg}_A(\mathbf{Fil}(\mathbf{C})))\] of \textit{filtered derived commutative \(A\)-algebras with filtered \(S^1\)-actions}.

Since the HKR filtration on Hochschild homology is non-negatively filtered (that is, the Postnikov filtration functor takes values in the non-negative part of the \(t\)-structure \(\mathbf{C}_{\geq 0}\)), we need to restrict to the \(\infty\)-category \[\mathsf{Fil}_{S^1}(\mathsf{DAlg}_A(\mathbf{C}))^{\geq 0} \defeq \mathsf{Fil}_{S^1}(\mathsf{DAlg}_A(\mathbf{C})) \times_{\mathsf{DAlg}_A(\mathbf{Fil}(\mathbf{C}))} \mathsf{DAlg}_A(\mathbf{Fil}(\mathbf{C}))^{\geq 0}\] of \textit{non-negatively filtered derived commutative \(A\)-algebras with filtered \(S^1\)-action}.   

The following definition is a consequence of \cite{raksit2020hochschild}*{Proposition 6.2.4}. It formulates the universal property of HKR-filtered Hochschild homology.

\begin{definition}\label{def:Hochschild_homology}
We define the \textit{HKR-filtered Hochschild functor} \(\mathsf{HH}(-/A)_\fil \colon \mathsf{DAlg}_A \to \mathbf{Fil}_{S^1}(\mathsf{DAlg}_A(\mathbf{C}))^{\geq 0}\) as the left adjoint of the forgetful functor 
\(U_\fil \colon \mathbf{Fil}_{S^1}(\mathsf{DAlg}_A(\mathbf{C}))^{\geq 0} \to \mathbf{Fil}(\mathsf{DAlg}_A(\mathbf{C}))^{\geq 0}\), composed with the evaluation functor \(\mathbf{Fil}(\mathsf{DAlg}_A(\mathbf{C}))^{\geq 0} \to \mathsf{DAlg}_A(\mathbf{C})\).
\end{definition}

Our primary interest is in the derived algebraic context \(\mathbf{Ch}(\mathcal{E})\), where \(\mathcal{E}\) is either the category \(\mathsf{CBorn}_R\) or \(\mathsf{Ind}(\mathsf{Ban}_R)\). Specialised to these categories, the universal property of Hochschild homology reads as follows: let \(A\) be a simplicial commutative complete bornological \(R\)-algebra, and let \(B\) be a simplicial commutative complete bornological \(A\)-algebra. Suppose \(X = (X_\bullet^n)_{n \in \Z}\) is a non-negatively filtered simplicial commutative complete bornological \(R\)-algebra with a filtered \(S^1\)-action, together with a morphism of simplicial \(R\)-algebras \(B \to X^0\), then there is a unique extension \(\mathsf{HH}_{\mathrm{fil}}(B/A) \to X_\bullet\) of filtered, \(S^1\)-equivariant simplicial commutative complete bornological \(R\)-algebras. 

%Mainly in this paper we will be interested in rational derived algebraic contexts.

\subsection{Homotopy coherent cochain complexes and derived de Rham cohomology}

Let \(R\) be a commutative ring and \(A\) a finitely generated polynomial \(R\)-algebra. The left Kan extension of the functor \(A \mapsto \Omega_{A/R}^1 \in \mathbf{Ch}(\mathsf{Mod}_R)\) to the \(\infty\)-category \(\mathsf{sCAlg}_R\) is called the \textit{cotangent complex} functor, denoted \(\mathsf{L}\Omega_{-/R}^1 \colon \mathsf{sCAlg}_R \to \mathbf{Ch}(\mathsf{Mod}_R)\).  Explicitly, we resolve a simplicial commutative \(R\)-algebra \(A\) by polynomial \(R\)-algebras \(P_\bullet \to A\), and set \(\mathsf{L}\Omega_{A/R}^1 = \Omega_{P_\bullet/R}^1 \otimes_{P_\bullet} A \in \mathbf{Ch}(\mathsf{Mod}_R)\). Similarly, we can define \(\mathsf{L}\Omega_{-/R}^n\) as left Kan extensions of the functor \(A \mapsto \Omega_{A/R}^n \defeq \bigwedge^n \Omega_{A/R}^1\). The modules of K\"ahler differentials \(\Omega_{A/R}^\bullet\), and their derived versions \(\mathsf{L}\Omega_{A/R}^\bullet\) assemble into an \(\mathbb{E}_\infty\)-\(R\)-algebra. Furthermore, we are also interested in the derived de Rham complex equipped with its \textit{Hodge filtration}, which is a complete, \(\N\)-indexed filtration. The resulting \textit{Hodge-completed derived de Rham complex} is a complete, filtered \(\mathbb{E}_\infty\)-\(R\)-algebra. This can be viewed as a graded \(\mathbb{E}_\infty\)-algebra with the structure of a \(\mathbb{D}_+ \defeq \mathsf{gr}(\Z[S^1]_\mathrm{fil})\)-module structure, capturing the higher algebraic analogue of a strictly commutative dg-algebra. The Hochschild-Kostant-Rosenberg Theorem then identifies this object with the associated graded of the HKR filtered Hochschild homology with its filtered \(S^1\)-action.

%Write forward references to definitions and results.

%In order to characterise the derived de Rham complex with its differential via a similar universal property as in the Hochschild homology case, we need an analogue of the \(S^1\)-action. Consider the \(\infty\)-category of chain complexes of abelian groups \(\mathbf{Ch}(\mathsf{Mod}_\Z)\), viewed as a derived algebraic context with its usual Postnikov \(t\)-structure. The positive and the negative \(t\)-structures on \(\Gr(\mathsf{Mod}_\Z)\) yield canonical lax symmetric monoidal embeddings \[\iota_{\pm} \colon \Gr(\mathsf{Mod}_\Z^\heartsuit)^{\otimes \kappa} \hookrightarrow \Gr(\mathsf{Mod}_\Z).\]

Let \(\mathbf{C}\) be a \(\Z\)-linear, stable, presentable, symmetric monoidal \(\infty\)-category. We define the graded algebra \(\mathbb{D}_+ \defeq \mathsf{gr}(\mathbb{T}_{\mathbf{C},\mathrm{fil}}) \in \mathbf{Gr}(\mathbf{C})\). It is a dualisable, bicommutative bialgebra object of \(\mathbf{Gr}(\mathbf{C})\) (see \cite{raksit2020hochschild}*{Construction 5.1.1}). When \(\mathbf{C} = \mathbf{Ch}(\mathsf{Mod}_\Z)\), the underlying graded \(\Z\)-module of \(\mathbb{D}_+\) is the split square-zero algebra \(\Z \oplus \Z[1]\). 

%\noindent Define two bicocommutative bialgebras \(\mathbb{D}_\pm \defeq \iota_{\pm}(\Z[x]/(x^2)) \in \Gr(\mathsf{Mod}_\Z)\). The underlying \(\Z\)-module of \(\mathbb{D}_\pm\) is the split square-zero algebra \(\Z \oplus \Z[\pm 1]\). If \(\mathcal{C}\) is additionally \(\Z\)-linear, then the canonical map \(\mathsf{Mod}_\Z \to \mathcal{C}\) induces a symmetric monoidal functor \(\Gr(\mathsf{Mod}_\Z) \to \Gr(\mathcal{C})\). Denote the image of \(\mathbb{D}_\pm\) in the category \(\Gr(\mathcal{C})\) by the same notation.  

\begin{definition}\label{def:htpy_coherent}
Let \(\mathbf{C}\) be a \(\Z\)-linear stable, presentable, symmetric monoidal \(\infty\)-category. The \(\infty\)-category \(\mathsf{DG}_+(\mathbf{C}) \defeq \mathsf{Mod}_{\mathbb{D}_+}(\Gr(\mathbf{C}))\) is called the \(\infty\)-category of \textit{\(h_+\)-differential graded objects} of \(\mathbf{C}\). 
\end{definition}

The \(\infty\)-category \(\mathsf{DG}_+(\mathbf{C})\) is a presentable, symmetric monoidal \(\infty\)-category, using the bicommutative bialgebra structures on \(\mathbb{D}_+\).

Now let \(\mathbf{C}\) be a derived algebraic context, and \(A\) a derived commutative algebra object of \(\mathbf{C}\). Consider the dual bialgebra $\mathbb{D}^{\vee}_{+}$ of $\mathbb{D}_{+}$. As in \cite{raksit2020hochschild} Definition 5.1.10 define
$$\mathsf{DG}_+(\mathsf{DAlg}_A(\mathbf{C})) \defeq \mathsf{cMod}_{\mathbb{D}^{\vee}_{+}}(\mathsf{DAlg}_{A}(\mathbf{Gr}(\mathbf{C})))$$

%Consider the \(\infty\)-category \(\mathsf{DG}_+(\mathbf{Ch}(\mathsf{Mod}_A))\) from Definition \ref{def:htpy_coherent} above. 

The objects of the \(\infty\)-category \(\mathsf{DG}_+(\mathsf{DAlg}_A(\mathbf{C}))\) are called \textit{\(h_+\)- differential graded commutative \(A\)-algebras in \(\mathbf{C}\)}. We can define its non-negative part as the \(\infty\)-category
% Why non-negative
\[\mathsf{DG}_+ \mathsf{DAlg}_A(\mathbf{C})^{\geq 0} \defeq \mathsf{DG}_+ \mathsf{DAlg}_A(\mathbf{C}) \times_{\mathsf{DAlg}_{A}(\mathbf{Gr}(\mathbf{C}))} \mathsf{DAlg}_{A}(\mathbf{Gr}(\mathbf{C})^{\ge0})\] as the \(\infty\)-category of \textit{non-negative \(h_+\)-differential graded derived commutative \(A\)-algebras}. %One can similarly talk about \(h_{-}\)-differential graded commutative \(A\)-algebras. It turns out that there is a symmetric monoidal equivalence of \(\infty\)-categories between \(\mathsf{DG}_{-}(\mathcal{C}) \cong \mathsf{DG}_+(\mathcal{C})\) for a \(\Z\)-linear, symmetric monoidal \(\infty\)-category \(\mathcal{C}\) (\cite{raksit2020hochschild}*{Remark 5.1.12}). 

% left adjoint's existence explicitly uses that \mathbb{D}_+ is dualisable.

The forgetful functor composed with the evaluation functor \[\mathsf{DG}_+\mathsf{DAlg}_A(\mathbf{C})^{\geq 0}  \overset{U}\to \mathsf{DAlg}_{A}(\mathbf{Gr}(\mathbf{C}))^{\geq 0} \overset{\ev^0}\to \mathsf{DAlg}_A(\mathbf{C})\] admits a left adjoint \[\mathsf{L}\Omega_{-/A}^{+\bullet} \colon \mathsf{DAlg}_A(\mathbf{C}) \to \mathsf{DG}_+^{\geq 0}(\mathsf{DAlg}_A(\mathbf{C}))\] called the \textit{derived de Rham functor}  over \(A\). For a derived commutative \(A\)-algebra \(B\),  the image of this functor \(\mathsf{L}\Omega_{B/A}^{+ \bullet}\) is called the \textit{derived de Rham complex of \(B\) over \(A\)}. By construction, it is a non-negative \(h_+\)-graded derived commutative algebra.

\subsubsection{Derivations, the Cotangent Complex, and de Rham Chomology}

Here we relate the formally defined derived de Rham complex \(\mathsf{L}\Omega_{B/A}^{+ \bullet}\) above to the usual definition involving the derived symmetric algebra \(\mathsf{LSym}_B(\mathsf{L}_{B/A})\) of the relative cotangent complex \(\mathsf{L}_{B/A}\). Let $A$ be a derived commutative algebra and $M$ an $A$-module.  \cite{raksit2020hochschild}*{Construction 5.1.1} defines for all such pairs $(A,M)$ the \textit{square-zero extension} $A\oplus M$ as an object of $\mathrm{DAlg}_{A}(\mathbf{C})$.

\begin{definition}\cite{raksit2020hochschild}*{Definition 4.4.7}
Let $A\in\mathrm{DAlg}(\mathbf{C}),B\in\mathrm{DAlg}_{A}(\mathbf{C})$, and $M\in\mathbf{Mod}_{B}$. Define 
$$\mathrm{Der}_{A}(B,M)\defeq\mathrm{Map}_{\mathrm{DAlg}_{A}\big\slash B}(B,B\oplus M)$$
An $A$-\textit{linear derivation of} $B$ \textit{into} $M$ is an element of $\pi_{0}(\mathrm{Der}_{A}(B,M))$.
\end{definition}

In \cite{raksit2020hochschild}*{Construction 4.4.10} Raksit shows the following.

\begin{proposition}
Let $A\in\mathrm{DAlg}(\mathbf{C})$ and $B\in\mathrm{DAlg}_{A}(\mathbf{C})$. The functor 
$$\mathbf{Mod}_{B}\rightarrow\mathbf{sSet},\;\; M\mapsto\mathrm{Der}_{A}(B,M)$$
is representable by a $B$-module $\mathsf{L}\Omega_{B/A}^{1}$. 
\end{proposition}
The \(B\)-module $\mathsf{L}\Omega_{B/A}^{1}$ is called the \textit{cotangent complex of }$B$\textit{ over }$A$. Note that the identity map $\mathsf{L}\Omega_{B/A}^{1}\rightarrow \mathsf{L}\Omega_{B/A}^{1}$ means that $\mathsf{L}\Omega_{B/A}^{1}$ has a universal $A$-linear derivation $d:B\rightarrow\mathsf{L}\Omega_{B/A}^{1}$. If $A$ and $B$ are connective, then the functor
$$\mathbf{Mod}^{cn}_{B}\rightarrow\mathbf{sSet},\;\; M\mapsto\mathrm{Der}_{A}(B,M)$$
is also representable by $\mathsf{L}\Omega_{B/A}^{1}$, where $\mathbf{Mod}^{cn}_{B}\defeq\mathbf{Mod}_{B}\times_{\mathbf{C}}\mathbf{C}_{\ge0}$.

\begin{theorem}\cite{raksit2020hochschild}*{Theorem 5.3.6}
Let \(\mathbf{C}\) be a derived algebraic context and \(A\) a derived commutative algebra object in \(\mathbf{C}\). Then for any derived commutative \(A\)-algebra \(B\), there is a canonical equivalence of graded commutative \(B\)-algebras

\[\mathsf{L}\Omega_{B/A}^{+\bullet} \cong \mathsf{LSym}_B(\mathsf{L}\Omega_{B/A}^1[1](1)),\]

\noindent inducing equivalences of \(B\)-modules \(\mathsf{L}\Omega_{B/A}^{+i} \cong (\bigwedge_B^i \mathsf{L}\Omega_{B/A}^1)[i]\), for \(i \geq 0\). Furthermore, under this equivalence, the first differential \[B \cong \mathsf{L}\Omega_{B/A}^{+0} \to \mathsf{L}\Omega_{B/A}^{+1}[-1] \cong \mathsf{L}_{B/A}\] of \(h_+\)-differential graded algebras is the universal \(A\)-linear derivation \(B \to \mathsf{L}_{B/A}\). 
\end{theorem}

To equip the derived de Rham complex with its Hodge structure, first consider the Koszul duality identification \[\coma{\mathsf{Fil}}(\mathbf{C}) \cong \mathsf{Mod}_{\mathbb{D}_-}(\mathbf{Gr}(\mathbf{C})).\] 

The latter \(\infty\)-category can be identified with \(\mathsf{DG}_+(\mathbf{C}) = \mathsf{Mod}_{\mathbb{D}_+}(\mathbf{Gr}(\mathbf{C}))\) by the shift equivalence functor \([2*] \colon (X^n)_{n \in \Z} \mapsto (X^n[2n])_{n \in \Z}\), which sends \(\mathbb{D}_- \mapsto \mathbb{D}_+\). The inverse of this functor is the reverse shift \([-2*] \colon (X^n)_{n \in \Z} \mapsto (X^n[-2n])_{n \in \Z}\). Combining these equivalences, the following objects can be defined:

\begin{definition}\cite{raksit2020hochschild}*{Definition 5.2.4}\label{def:brutal_filtration}
Let \(\mathbf{C}\) be a derived algebraic context and let \(X \in \mathsf{DG}_+(\mathbf{C})\).
\begin{itemize}
\item The \textit{brutal filtration} of \(X\) is the object \(\abs{X[-2*]}^{\geq *}  \in \coma{\mathsf{Fil}}(\mathbf{C})\) obtained from the composition of the reverse shift equivalence \([-2*]\) and the Koszul duality equivalence \(\vert - \vert \colon \mathsf{Mod}_{\mathbb{D}_-}(\mathbf{Gr}(\mathbf{C})) \to \coma{\mathsf{Fil}}(\mathbf{C})\);
\item The \textit{cohomology type} of \(X\) is the object \(\vert X \vert  \defeq \mathsf{colim} \abs{X[-2*]}^{\geq *} \in \mathbf{C}\).
\end{itemize}
\end{definition}

Applying Definition \ref{def:brutal_filtration} to  \(X = \mathsf{L}\Omega_{B/A}^{+\bullet}\), the cohomology type \[\mathsf{dR}_{B/A}^{\wedge} \defeq \vert \mathsf{L}\Omega_{B/A}^{+\bullet} \vert\] is called the \textit{Hodge-completed derived de Rham cohomology}, and its brutal filtration is called the \textit{Hodge filtration}. When \(\mathbf{C} = \mathbf{Ch}(\mathsf{Mod}_\Z)\) and \(A\) is an ordinary commutative ring, these definitions recover the usual Hodge-completed derived de Rham cohomology and Hodge filtrations by \cite{raksit2020hochschild}*{Corollary 5.3.9}. 

%Combining this with the identification \([-2 *]\colon \mathsf{DG}_{-}(\mathcal{C}) \cong \mathsf{DG}_+(\mathcal{C})\), we can associate to an object \(X \in \mathsf{DG}_-(\mathcal{C})\), its image \(\) under these identifications. Define \(\abs{X} \defeq \colim \abs{X[-2*]}^{\geq *} \in \mathcal{C}\). We call \(\abs{X}\) and \(\abs{X}^{\geq *}\) the \textit{cohomology type} of \(X\) and the \textit{brutal filtration} on \(X\), respectively. 

\begin{definition}[Bornological derived de Rham cohomology]
Let \(\mathbf{C}\) be the derived algebraic context \(\mathbf{Ch}(\mathcal{E})\), where \(\mathcal{E}\) is the quasi-abelian category \(\mathsf{CBorn}_R\) or \(\mathsf{Ind}(\mathsf{Ban}_R)\). Let \(A\) be a derived commutative algebra object in \(\mathbf{C}\). Then for a derived commutative \(A\)-algebra \(B\),  we call the corresponding derived de Rham complex, its cohomology type, and its brutal filtration, the \textit{bornological derived de Rham complex}, the \textit{bornologically Hodge-completed derived de Rham cohomology}, and the \textit{bornological Hodge filtration}, respectively. 
\end{definition}

\begin{remark}
The bornologically Hodge-completed derived de Rham cohomology is an object of \(\mathbb{E}_\infty\)-bornological \(A\)-algebras (i.e., \(\mathsf{CAlg}(\mathbf{Ch}(\mathsf{CBorn}_R))\)), and the bornological Hodge filtration is an object of complete, filtered \(\mathbb{E}_\infty\)-bornological \(A\)-algebras. 
\end{remark}

\subsection{The Hochschild-Kostant-Rosenberg Theorem}

What we have so far is as follows: let \(\mathbf{C}\) be a derived algebraic context and \(A\) a derived algebra object in \(\mathbf{C}\). There are two \(\infty\)-functors:

\begin{itemize}
\item the \textit{derived de Rham functor} with its Hodge filtration:

\[\mathsf{L}\Omega_{-/A}^{+\bullet} \colon \mathsf{DAlg}_A \to \mathsf{DG}_+^{\geq 0} \mathsf{DAlg}_A,\]

\item the HKR filtered Hochschild homology functor \[\HH(-/A)_\fil \colon \mathsf{DAlg}_A \to  \mathsf{Fil}_{S^1}^{\geq 0}(\mathsf{DAlg}_A).\]
\end{itemize}

\noindent Both these functors arise as left adjoint functors to appropriate forgetful functors, as previously described. Finally, the associated graded functor is adjoint to the functor \[\zeta \colon \mathsf{DG}_+ (\mathsf{DAlg}_A)^{\geq 0} \to \mathsf{Fil}_{S^1}^{\geq 0}(\mathsf{DAlg}_A)\]  that assigns to a filtered object \((X^n)_{n \in \Z}\), the trivially graded object \(\cdots \to X^{-1} \overset{0}\to X^0 \overset{0}\to X^1 \overset{0}\to \cdots\). Summarising,

\begin{theorem}[\cite{raksit2020hochschild}*{Theorem 6.2.6}]\label{theorem:HKR}
There is a commuting diagram of adjunctions:

\[ \begin{tikzcd}
\mathsf{Fil}_{S^1}(\mathsf{DAlg}_A(\mathbf{C})^{\geq 0} \arrow[r, yshift=0.7ex, "\mathsf{gr}"] \arrow[d, swap, "\ev^0"] & \mathsf{DG}_+ \mathsf{DAlg}_A(\mathbf{C})^{\geq 0} \arrow[d,swap,"\ev^0"] \arrow{l}{\zeta} \\%
\mathsf{DAlg}_A(\mathbf{C}) \arrow[u, xshift=0.8ex, swap,"\HH(-/A)_\fil" ]\arrow{r}{\mathsf{id}}& \mathsf{DAlg}_A(\mathbf{C}) \arrow[u,xshift=0.8ex, swap, "\mathsf{L}\Omega_{-/A}^{+\bullet}"].
\end{tikzcd}
\]

\end{theorem}

\begin{corollary}
Let \(\mathbf{C}\) be the derived algebraic context  \(\mathbf{Ch}(\mathsf{CBorn}_R)\), and let \(A\) be a simplicial commutative, complete bornological \(R\)-algebra. Then for any simplicial, commutative, complete bornological \(A\)-algebra, the associated graded of the HKR-filtered Hochschild homology complex \(\mathsf{HH}(B/A)_{\mathsf{fil}}\) agrees with the bornological derived de Rham complex \(\mathsf{L}(\Omega_{B/A}^{+\bullet})\). 
\end{corollary}

\begin{proposition}\label{prop:ratequiv}
If $\mathbf{C}$ is a rational derived algebraic context, then the HKR filtration is canonically split, and there is a natural equivalence
$$\mathsf{Sym}(\mathrm{L}\Omega^{1}_{B\big\slash A}[-1])\cong\mathsf{HH}(B/A).$$
\end{proposition}

\begin{proof}
The base change $\mathbb{Q}\otimes\mathbb{Z}[S_{1}]^{\vee}_{\mathsf{fil}}$ is split by \cite{toen-robalo} Theorem 1.2.1, and therefore isomorphic to $\mathsf{split}(\mathbb{D}^{\vee}_{+})$. We get a commutative diagram
\begin{displaymath}
\xymatrix{
\mathsf{Fil}_{S^1}(\mathsf{DAlg}_A(\mathbf{C})^{\geq 0}) \ar[d]\ar[r] & \mathsf{cMod}_{\mathrm{split}(\mathbb{D}^{\vee}_{+})}\mathsf{DAlg}_{A}(\mathsf{Fil}(\mathbf{C}))\ar[d]^{\mathsf{und}}\\
\mathsf{DAlg}_{A}(\mathsf{Fil}(\mathbf{C}))\ar[d] & \mathrm{DG}_{+}^{\ge0}\mathsf{DAlg}_{A}\ar[d]^{\mathrm{ev}^{0}\circ U}\\
\mathsf{DAlg}_{A}(\mathbf{C})\ar@{=}[r] & \mathsf{DAlg}_{A}(\mathbf{C})
}
\end{displaymath}
All maps are right adjoints, and the top and bottom maps are equivalences. The left adjoint of the path along the left-hand column is $\mathsf{HH}_{\mathsf{fil}}$, and the left adjoint along the right-hand path is $\mathrm{split}(\mathsf{L}\Omega_{-/A}^{+\bullet})$. 
\end{proof}

\subsection{Affinisation}

The (rational) HKR theorem we have proven can be viewed as the affine version of a geometric result. Let $\mathbf{M}$ be a locally copresentable (or more generally a cototal) $(\infty,1)$-category. The following is explained in great detail in \cite{toen2006champs} Section 2.

Consider the category $\mathbf{PreStk}(\mathbf{M})$ of prestacks valued in \(\infty\)-groupoids, and the Yoneda embedding which we write as $y:\mathbf{M}\rightarrow\mathbf{PreStk}(\mathbf{M})$. Since $\mathbf{M}$ is cototal, we have an adjunction
$$\adj{\mathbf{Aff}}{\mathbf{PreStk}(\mathbf{M})}{\mathbf{M}}{y}.$$
Roughly speaking, $\mathbf{Aff}$ is the functor which sends $\mathcal{X}\cong\colim y(M_{\alpha})$ to $\mathbf{Aff}(\mathcal{X})\defeq\colim M_{\alpha}$. We may regard $\textbf{Top}\subset\mathbf{PreStk}(\mathbf{M})$ as a full subcategory of constant prestacks. In particular we may regard the circle $S^{1}$ as a pre-stack on $\mathbf{M}$.
Write $\mathrm{Spec}:\mathbf{M}^{op}\rightarrow\mathbf{M}$ for the canonical anti-equivalence, and denote its inverse by $\mathcal{O}$. Let $k$ denote the initial object of $\mathbf{M}^{op}$.
Then $\mathrm{Spec}(k)$ is the terminal object of $\mathbf{PreStk}(\mathbf{M})$. We may write the circle as 
$$S^{1}\cong \mathrm{Spec}(k)\coprod_{\mathrm{Spec}(k)\coprod\mathrm{Spec}(k)}\mathrm{Spec}(k)$$

We therefore get a map in $\mathbf{PreStk}(\mathbf{M})$
$$S^{1}\rightarrow \mathrm{Spec}(k\prod_{k\prod k}k)$$
and therefore a map in $\mathbf{M}$
$$\mathbf{Aff}(S^{1})\rightarrow\mathrm{Spec}(k\prod_{k\prod k}k)$$
Thus we get a natural transformation of endofunctors of $\mathbf{M}$
$$\mathrm{Map}(\mathrm{Spec}(k\prod_{k\prod k}k),\mathrm{Spec}(-))\rightarrow\mathrm{Map}(S^{1},\mathrm{Spec}(-))$$

If $\mathbf{C}$ is a derived algebraic context, then taking $\mathbf{M}=\mathsf{DAlg}(\mathbf{C}_{\ge0})^{op}$, we have that $k$ is the monoidal unit and
 $$k\prod_{k\prod k}k\cong k\oplus k[1].$$
Indeed, by direct computation, this is true in $\mathbf{C}_{\mathrm{ab}}$ and remains true after base-changing to $\mathbf{C}$. 
The left-hand side of the map above is precisely the derived de Rham complex functor, and the right-hand side is the Hochschild homology functor.  The rational HKR theorem says precisely that when $\mathbf{C}$ is a rational derived algebraic context, this map is an equivalence. With this formulation it is clear that these functors and the natural transformation between them can be boosted to endofunctors on $\mathbf{PreStk}$. The geometric HKR results which we prove in the remaining sections extend this equivalence to larger subcategories of $\mathbf{PreStk}(\mathsf{DAlg}(\mathbf{C})^{op})$. These subcategories will be the schemes for certain geometric contexts modelled on $\mathsf{DAlg}(\mathbf{C})^{op}$, using the general setup of geometry relative to a monoidal $(\infty,1)$-category. 
In particular, if $R$ has characteristic zero, then applying our results to $\mathbf{C}_{R}$ will recover results of \cite{ben-zvinadler} (and the proof strategy is essentially a generalisation of Ben-Zvi-Nadler's work). We also explain how one can deduce a HKR theorem for derived analytic stacks, in the sense of forthcoming work \cite{BKK}. This is closely related to work of Anotnio, Petit, and Porta \cite{Antoniothesis}.

\section{A Categorical Context for Geometry}
\label{sec:categorical_geometry}

\subsection{A Basic Context for Geometry}

Before specialising to geometry relative to derived algebraic contexts, let us give a very general setup for geometry in the style of \cite{toen2005homotopical}. 

\begin{definition}[\cite{toen2005homotopical}*{Proposition 4.3.5}]
Let $\mathbf{M}$ be an $(\infty,1)$-category. A \textit{Grothendieck topology on }$\mathbf{M}$, is a Grothendieck topology on $\mathrm{Ho}(\mathbf{M})$ in the usual $1$-categorical sense. 
\end{definition}

Let $\mathbf{M}$ be an $(\infty,1)$-category and $\tau$ a topology on $\mathbf{M}$. Recall that the category of \textit{pre-stacks} (in $\infty$-groupoids) on $\mathbf{M}$ is the category of functors $\textbf{PrStk}(\mathbf{M})\defeq\textbf{Fun}(\mathbf{M}^{op},\textbf{Grpd})$ where $\textbf{Grpd}$ is the category of $\infty$-groupoids, which is presented by the model category of simplicial sets (with its Quillen model structure).
We define
$$\mathbf{Stk}(\mathbf{M},\tau)$$
to be the full subcategory of $\textbf{PrStk}(\mathbf{M})$ consisting of stacks which satisfy descent for hypercovers in $\tau$. $\mathbf{Stk}(\mathbf{M},\tau)$ is a reflective subcategory of $\mathbf{PreStk}(\mathbf{M})$ and is therefore a topos (see \cite{lurie}*{Chapter 6}). It is also closed monoidal for the cartesian monoidal structure, in particular one can form the mapping stack $\underline{\mathsf{Map}}(\mathcal{X},\mathcal{Y})$ between any two-stacks. 

\begin{definition}
Let $(\mathbf{M},\tau)$ be a site. An object $X$ of $\mathbf{M}$ is said to be \textit{admissible} if $\mathrm{Map}(-,X)$ is a stack.
\end{definition}

\begin{definition}
An $(\infty,1)$-\textit{pre-geometry triple} is a triple $(\mathbf{M},\tau,\mathbf{P})$ where
\begin{enumerate}
\item
$\mathbf{M}$ is an $(\infty,1)$-category.
\item
$\tau$ is a Grothendieck topology on $\mathbf{M}$.
\item
$\mathbf{P}$ is a class of maps in $\mathbf{M}$ 
\end{enumerate}
such that
\begin{enumerate}
\item
If $\{U_{i}\rightarrow U\}$ is a cover in $\tau$, then each $U_{i}\rightarrow U$ is in $\mathbf{P}$.
\item
If $f:U\rightarrow V$ is a map in $\mathbf{M}$, and $\{U_{i}\rightarrow U\}$ is a cover such that each composition $U_{i}\rightarrow V$ is in $\mathbf{P}$, then $f\in\mathbf{P}$. 
\item
The class $\mathbf{P}$ contains isomorphisms, and is closed under pullbacks and compositions.
\end{enumerate}
An $(\infty,1)$-pre-geometry triple is said to be an $(\infty,1)$-\textit{geometry triple} if every object of $\mathbf{M}$ is admissible.
\end{definition}

With a tremendous amount of caution, $\mathbf{P}$ will be regarded as a formalisation of a class of (local) immersions. Particularly for applications to derived analytic geometry, it is convenient to have a relative version of this.

\begin{definition}
\begin{enumerate}
\item
A \textit{relative} $(\infty,1)$-\textit{pre-geometry tuple} is a tuple $(\mathbf{M},\tau,\mathbf{P},\mathbf{A})$ where $(\mathbf{M},\tau,\mathbf{P})$ is an $(\infty,1)$-pre-geometry triple, and $\mathbf{A}\subset\mathbf{M}$ is a full subcategory such that if $f:X\rightarrow Y$ is a map in $\mathbf{P}\cap\mathbf{A}$, and $Z\rightarrow Y$ is any map with $Z$ in $\mathbf{A}$, then $X\times_{Y}Z$ is in $\mathbf{A}$. 
\item
A relative$(\infty,1)$-pre-geometry tuple $(\mathbf{M},\tau,\mathbf{P},\mathbf{A})$ is said to be \textit{strong} if whenever $\{U_{i}\rightarrow U\}$ is a cover in $\mathbf{M}$ (not necessarily in $\mathbf{A}$), and $Z\rightarrow U$ is any map with $Z$ in $\mathbf{A}$, then $\{U_{i}\times_{U}Z\rightarrow Z\}$ is a cover in $\tau|_{\mathbf{A}}$.
%
%$f:X\rightarrow Y$ is a map in $\mathbf{P}$ (not necessarily in $\mathbf{A}$), 

\end{enumerate}
\end{definition}

The idea behind this definition is that $\mathbf{A}$ is the collection of (affine) geometric objects in which we are primarily interested, and $\mathbf{M}$ is a larger class of spaces which permits more categorical constructions than are available in $\mathbf{A}$.

Note that in particular if $(\mathbf{M},\tau,\mathbf{P},\mathbf{A})$ is a relative $(\infty,1)$-pre-geometry tuple then $(\mathbf{A},\tau|_{\mathbf{A}},\mathbf{P}|_{\mathbf{A}})$ is an $(\infty,1)$-pre-geometry triple, where $\mathbf{P}|_{\mathbf{A}}$ denotes the class of maps $f:X\rightarrow Y$ in $\mathbf{A}$ which are in $\mathbf{P}$ when regarded as maps in $\mathbf{M}$, and $\tau|_{\mathbf{A}}$ denotes the restriction of $\tau$ to $\mathbf{A}$. 

\begin{remark}
If $(\mathbf{M},\tau,\mathbf{P},\mathbf{A})$ is a relative $(\infty,1)$-geometry tuple, then there is an associated strong relative $(\infty,1)$-geometry tuple defined as follows. Let $\mathbf{P}_{\mathbf{A}}\subset\mathbf{P}$ denote the class of maps $f:X\rightarrow Y$ such that whenever $Z\rightarrow Y$ is a map with $Z\in\mathbf{A}$ then $X\times_{Y}Z$ is in $\mathbf{A}$. Let $\tau_{\mathbf{A}}$ denote the class of covers $\{U_{i}\rightarrow V\}$ in $\tau$ such that each $U_{i}\rightarrow V$ is in $\mathbf{P}_{\mathbf{A}}$. Then $(\mathbf{M},\tau_{\mathbf{A}},\mathbf{P})$ is a strong relative $(\infty,1)$-geometry tuple. 
\end{remark}

\begin{definition}
Let $(\mathbf{M},\tau,\mathbf{P},\mathbf{A})$ be a (strong) relative $(\infty,1)$-pre-geometry tuple .
\begin{enumerate}
\item
An object $X$ of $\mathbf{M}$ is said to be $\mathbf{A}$-admissible if the restriction of $\mathrm{Map}(-,X)$ to $\mathbf{A}$ is a stack for $\tau|_{\textbf{A}}$.
\item
$(\mathbf{M},\tau,\mathbf{P},\mathbf{A})$  is said to be a \textit{(strong) relative} $(\infty,1)$-\textit{geometry tuple} if each $X\in\mathbf{A}$ is $\mathbf{A}$-admissible. 
\end{enumerate}
\end{definition}

We have a fully faithful functor, given by left Kan extension,

$$i:\mathbf{PreStk}(\mathbf{A})\rightarrow\mathbf{PreStk}(\mathbf{M})$$
which is left adjoint to the restriction functor. The restriction functor
$$(-)^{\mathbf{A}}:\mathbf{Stk}(\mathbf{M},\tau)\rightarrow\mathbf{Stk}(\mathbf{A},\tau|_{\mathbf{A}})$$
also has a left adjoint 
$$i^{\#}:\mathbf{Stk}(\mathbf{A},\tau|_{\mathbf{A}})\rightarrow\mathbf{Stk}(\mathbf{M},\tau),$$
which is the composition of the stackification functor with $i$.

\begin{proposition}
Let $(\mathbf{M},\tau,\mathbf{P},\mathbf{A})$ be a relative $(\infty,1)$-geometry tuple.
\begin{enumerate}
\item
If $\mathcal{X}\rightarrow\mathcal{Y}$ is an epimorphism in $\mathbf{Stk}(\mathbf{A},\tau|_{\mathbf{A}})$ then $i^{\#}(\mathcal{X})\rightarrow i^{\#}(\mathcal{Y})$ is an epimorphism in $\mathbf{Stk}(\mathbf{M},\tau)$.
\item
If the relative $(\infty,1)$-geometry tuple is strong then $i^{\#}$ is fully faithful. 
\end{enumerate}
\end{proposition}

\begin{proof}
\begin{enumerate}
\item
Let $\mathcal{Z}$ be a pre-stack on $\mathbf{M}$. If $\tilde{\mathcal{Z}}$ is its stackification then $\pi_{0}^{\tau}(\mathcal{Z})\rightarrow\pi_{0}^{\tau}(\tilde{\mathcal{Z}})$ is an equivalence. Thus it suffices to show that $\pi_{0}^{\tau}(i(\mathcal{X}))\rightarrow\pi_{0}^{\tau}(i(\mathcal{Y}))$ is an epimorphism. Now since $\pi_{0}:\textbf{sSet}\rightarrow\textbf{Set}$ and sheafification commute with colimits, this map is equivalent to
$$i^{\#}(\pi_{0}^{\tau|_{\mathbf{A}}}(\mathcal{X}))\rightarrow i^{\#}(\pi_{0}^{\tau|_{\mathbf{A}}}(\mathcal{Y}))$$
Since $i^{\#}$ is a left adjoint, this map is an epimorphism.
\item
It suffices to show that the unit map $\mathcal{X}\rightarrow(i^{\#}(\mathcal{X}))^{\mathbf{A}}$ is an isomorphism. Let $U\in\mathbf{A}$. We need to show that the map
$$\mathcal{X}(U)\rightarrow i^{\#}(\mathcal{X})(U)$$
is an isomorphism. Now 
$$i^{\#}(\mathcal{X})(U)\cong\colim_{K_{\bullet}\rightarrow U}\mathrm{Map}(K_{\bullet},\mathrm{LKE}(\mathcal{X}))$$
where $K_{\bullet}\rightarrow U$ is a hypercover. But since the relative $(\infty,1)$-geometry tuple is strong, a cover in $\tau$ of an object in $\mathbf{A}$ we may assume that it is a hypercover of the object in $\tau|_{\mathbf{A}}$. Thus the right-hand side is nothing but the stackification of $\mathcal{X}$, which is of course just $\mathcal{X}$.\qedhere
\end{enumerate}
\end{proof}

\subsection{Geometric Stacks}

In this section we introduce geometric stacks, following closely \cite{toen2008homotopical} Chapter 1.3. These are stacks which are glued together from affines in a controllable way. We fix  a relative $(\infty,1)$-pre-geometry tuple $(\mathbf{M},\tau,\mathbf{P},\mathbf{A})$.

\begin{definition}[\cite{toen2008homotopical} Definition 1.3.3.1]
\begin{enumerate}
\item
A prestack $\mathcal{X}$ in $\mathbf{Stk}(\mathbf{A},\tau|_{\mathbf{A}})$ is $(-1)$-\textit{geometric} if it is of the form $\mathcal{X}\cong\mathrm{Map}(-,M)$ for some $\mathbf{A}$-admissible $M\in\mathbf{M}$.
\item
A morphism $f:\mathcal{X}\rightarrow\mathcal{Y}$ in $\mathbf{Stk}(\mathbf{A},\tau|_{\mathbf{A}})$ is $(-1)$-\textit{representable} if for any map $U\rightarrow\mathcal{Y}$ with $U\in\mathbf{A}$ the pullback $\mathcal{X}\times_{\mathcal{Y}}U$ is $(-1)$-geometric.
\item
A morphism $f:\mathcal{X}\rightarrow\mathcal{Y}$ in  $\mathbf{Stk}(\mathbf{A},\tau|_{\mathbf{A}})$ is in $(-1)-\mathbf{P}$ if it is $(-1)$-representable and for any map $U\rightarrow\mathcal{Y}$ with $U\in\mathbf{A}$, the map $\mathcal{X}\times_{\mathcal{Y}}U\rightarrow U$ is in $\mathbf{P}$.
\end{enumerate}
\end{definition}

\begin{definition}[\cite{toen2008homotopical} Definition 1.3.3.1]\label{defn:ngeom}
Let $n\ge0$.
\begin{enumerate}
\item
Let $\mathcal{X}$ be in  $\mathbf{Stk}(\mathbf{A},\tau|_{\mathbf{A}})$. An $n$\textit{-atlas} for $\mathcal{X}$ is a set of morphisms $\{U_{i}\rightarrow \mathcal{X}\}$ such that 
\begin{enumerate}
\item
Each $U_{i}$ is $(-1)$-geometric.
\item
Each map $U_{i}\rightarrow\mathcal{X}$ is in $(n-1)$-$\mathbf{P}$ and is an epimorphism of stacks.
\end{enumerate}
\item
A stack $\mathcal{X}$ is $n$-\textit{geometric}  if
\begin{enumerate}
\item
The map $\mathcal{X}\rightarrow\mathcal{X}\times\mathcal{X}$ is $(n-1)$-representable.
\item
$\mathcal{X}$ admits an $n$-atlas.
\end{enumerate}
\item
A morphism of stacks $f:\mathcal{X}\rightarrow\mathcal{Y}$ is $n$-\textit{representable} if for any map  $U\rightarrow \mathcal{Y}$ with $U\in\mathbf{A}$ the pullback $\mathcal{X}\times_{\mathcal{Y}}U$ is $n$-geometric.
\item
A morphism of stacks $\mathcal{X}\rightarrow\mathcal{Y}$ is in $n$-$\mathbf{P}$ if it is $n$-representable, and for any map $U\rightarrow\mathcal{Y}$ with $U\in\mathbf{A}$, there is an $n$-atlas $\{U_{i}\rightarrow\mathcal{X}\times_{\mathcal{Y}}U\}_{i\in\mathcal{I}}$ such that each map $U_{i}\rightarrow U$ is in $\mathbf{P}$.
\end{enumerate}
\end{definition}

The full subcategory of $\mathbf{Stk}(\mathbf{A},\tau|_{\mathbf{A}})$ consisting of $n$-geometric  stacks is denoted $\mathbf{Stk}_{n}(\mathbf{M},\tau,\mathbf{P},\mathbf{A})$. As in \cite{toen2008homotopical}*{Proposition 1.3.3.6}, we have  $\mathbf{Stk}_{m}(\mathbf{M},\tau,\mathbf{P},\mathbf{A})\subseteq \mathbf{Stk}_{n}(\mathbf{M},\tau,\mathbf{P},\mathbf{A})$ for $-1\le m\le n$. The full subcategory of geometric stacks is $\mathbf{Stk}_{geom}(\mathbf{M},\tau,\mathbf{P},\mathbf{A})\defeq\bigcup_{n=-1}^{\infty}\mathbf{Stk}_{n}(\mathbf{M},\tau,\mathbf{P},\mathbf{A})$. As in \cite{toen2008homotopical}*{Corollary 1.3.3.5}, $\mathbf{Stk}_{n}(\mathbf{M},\tau,\mathbf{P},\mathbf{A})$ is closed under pullbacks.

\begin{remark}
Note that since we have not made any quasi-compactness assumptions on our topology disjoint unions of geometric stacks need not be geometric.
\end{remark}

If $(\mathbf{M},\tau,\mathbf{P})$ is an $(\infty,1)$-geometry triple then $(\mathbf{M},\tau,\mathbf{P},\mathbf{M})$ is a relative $(\infty,1)$-pre-geometry tuple. In this case we write
$$\mathbf{Stk}_{m}(\mathbf{M},\tau,\mathbf{P})\defeq\mathbf{Stk}_{m}(\mathbf{M},\tau,\mathbf{P},\mathbf{M})$$
and we say that map $f:\mathcal{X}\rightarrow\mathcal{Y}$ is $n$-representable, or in $n-\mathbf{P}$ if it is so in the sense of Definition \ref{defn:ngeom} for the relative $(\infty,1)$-pre-geometry tuple $(\mathbf{M},\tau,\mathbf{P},\mathbf{M})$. The following is immediate:

\begin{proposition}
Let $(\mathbf{M},\tau,\mathbf{P},\mathbf{A})$ be a relative $(\infty,1)$-geometry tuple. Then for each $m\ge -1$ we have $\mathbf{Stk}_{m}(\mathbf{A},\tau|_{\mathbf{A}},\mathbf{P}|_{\mathbf{A}})\subseteq \mathbf{Stk}_{m}(\mathbf{M},\tau,\mathbf{P},\mathbf{A})$.
\end{proposition}

If $(\mathbf{M},\tau,\mathbf{P},\mathbf{A})$ is a relative $(\infty,1)$-geometry tuple such that $(\mathbf{M},\tau,\mathbf{P})$ is an $(\infty,1)$-geometry triple (i.e. all objects of $\mathbf{M}$ are admissible) then we have a permanency property for geometric stacks.
Consider the functor $i^{\#}:\mathbf{Stk}(\mathbf{A},\tau|_{\mathbf{A}})\rightarrow\mathbf{Stk}(\mathbf{M},\tau)$. 

\begin{lemma}
Let $(\mathbf{M},\tau,\mathbf{P},\mathbf{A})$ be a relative $(\infty,1)$-geometry tuple such that $(\mathbf{M},\tau,\mathbf{P})$ is an $(\infty,1)$-geometry triple. Let $f:\mathcal{X}\rightarrow\mathcal{Y}$ be an map in $\mathbf{Stk}_{n}(\mathbf{A},\tau|_{\mathbf{A}},\mathbf{P}|_{\mathbf{A}})$ which is in $n-\mathbf{P}|_{\mathbf{A}}$. Then $i^{\#}(f)$ is in $n-\mathbf{P}$.
\end{lemma}

\begin{proof}
The proof is by induction on $n$. For $n=-1$ the claim is clear. Suppose the claim has been proven for some $(n-1)$ with $n\ge1$, and consider $n$. Let $\mathrm{Map}(-,X)\rightarrow i^{\#}(\mathcal{Y})$ be a map. There is a cover $\{\mathrm{Map}(-,X_{i})\rightarrow\mathrm{Map}(-,X)\}_{i\in\mathcal{I}}$ such that each $\mathrm{Map}(-,X_{i})\rightarrow i^{\#}(\mathcal{Y})$ factors through $i(\mathcal{Y})\rightarrow i^{\#}(\mathcal{Y})$. But then each map $\mathrm{Map}(-,X_{i})\rightarrow i(\mathcal{Y})$ factors through some map $\mathrm{Map}(-,A_{i})\rightarrow i(\mathcal{Y})$ with $A_{i}\in\mathbf{A}$. Consider the fibre product $\mathrm{Map}(-,A_{i})\times_{\mathcal{Y}}\mathcal{X}$. There is an an $n$-atlas $\{\mathrm{Map}(-,A_{ij})\rightarrow\mathrm{Map}(-,A_{i})\times_{\mathcal{Y}}\mathcal{X}$ such that each composition $\mathrm{Map}(-,A_{ij})\rightarrow\mathrm{Map}(-,A_{i})$ is in $\mathbf{P}$.  Now both the left Kan extension and stackification commute with finite limits. Thus $\mathrm{Map}(-,A_{ij})\rightarrow\mathrm{Map}(-,A_{i})\times_{i^{\#}(\mathcal{Y})}i^{\#}(\mathcal{X})$ is an $n$-atlas. Write $X_{ij}\defeq X_{i}\times_{A_{i}}A_{ij}$. Then $\{\mathrm{Map}(-,X_{ij})\rightarrow\mathrm{Map}(-,X_{i})\times_{i^{\#}(\mathcal{Y})}i^{\#}(\mathcal{X})\}$ is an $n$-atlas, and the composition $\mathrm{Map}(-,X_{ij})\rightarrow\mathrm{Map}(-,X_{i})$ is in $\mathbf{P}$. The result now follows from \cite{toen2008homotopical}*{Proposition 1.3.3.4}.
\end{proof}

\begin{corollary}
Let $(\mathbf{M},\tau,\mathbf{P},\mathbf{A})$ be a relative $(\infty,1)$-geometry tuple such that $(\mathbf{M},\tau,\mathbf{P})$ is an $(\infty,1)$-geometry triple. The functor  $i^{\#}:\mathbf{Stk}(\mathbf{A},\tau|_{\mathbf{A}})\rightarrow\mathbf{Stk}(\mathbf{M},\tau)$ induces a functor 
 $$i^{\#}:\mathbf{Stk}_{n}(\mathbf{A},\tau|_{\mathbf{A}},\mathbf{P}|_{\mathbf{A}})\rightarrow\mathbf{Stk}_{n}(\mathbf{M},\tau,\mathbf{P})$$
 for each $n$, which is fully faithful if the relative $(\infty,1)$-geometry tuple is strong.
\end{corollary}

%\begin{proposition}
%The full subcategory of admissible geometric stacks in $\mathbf{Stk}(\mathbf{M},\tau)$ is closed under fibre products. 
%\end{proposition}
%
%We can also define the category $\mathbf{Stk}_{n}^{\mathbf{A}}(\mathbf{M},\tau,\mathbf{P},\mathbf{A})\subset\mathbf{Stk}(\mathbf{M},\tau)$ as the full subcategory of $\mathbf{Stk}(\mathbf{M},\tau)$ consisting of stacks $\mathcal{X}$ such that $\mathcal{X}^{\mathbf{A}}$ is in $\mathbf{Stk}_{n}(\mathbf{A},\tau|_{\mathbf{A}},\mathbf{P}|_{\mathbf{A}})$. $\mathbf{Stk}_{n}^{\textbf{adm}}(\mathbf{M},\tau,\mathbf{P},\mathbf{A})$ is in particular a full subcategory of \(\mathbf{Stk}_{n}^{\mathbf{A}}(\mathbf{M},\tau,\mathbf{P},\mathbf{A})\).

\subsubsection{Segal Groupoids}

For quasi-compact topologies there is an alternative extremely useful description of geometric stacks using Segal groupoids. This is explained in \cite{toen2008homotopical} specifically for geometry relative to so-called HAG contexts. Although the important Proposition 1.3.4.2 in \cite{toen2008homotopical} does not quite go through in complete generality , we can still say something useful.

\begin{definition}[\cite{toen2008homotopical}*{Definition 1.3.1.6}]
Let $\textbf{N}$ be an $(\infty,1)$-category. A \textit{Segal groupoid in }$\textbf{N}$ is a simplicial object $X_{*}:\Delta^{op}\rightarrow\textbf{N}$ such that 
\begin{enumerate}
\item
For any $n>0$ the natural morphism
$$\prod_{0\le i<n}\sigma_{i}:X_{n}\rightarrow X_{1}\times_{X_{0}}X_{1}\times_{X_{0}}\ldots\times_{X_{0}}X_{1}$$
is an isomorphism.
\item
The morphism 
$$d_{0}\times d_{1}:X_{2}\rightarrow X_{1}\times_{X_{0}}X_{1}$$
is an isomorphism. 
\end{enumerate}
\end{definition}

Let $\mathcal{X}$ be an $n$-geometric stack with $n$-atlas $\{U_{i}\rightarrow\mathcal{X}\}_{i\in\mathcal{I}}$. For $\underline{i}=(i_{0},\ldots,i_{n})\in\mathcal{I}^{n+1}$, write
$$U_{\underline{i}}\defeq U_{i_{0}}\times_{\mathcal{X}}U_{i_{1}}\times_{\mathcal{X}}\times\ldots\times_{\mathcal{X}}U_{i_{n}}$$
 Consider the simplicial object $U_{\bullet}$ with 
$$U_{n}=\coprod_{\underline{i}\in\mathcal{I}^{n+1}}U_{\underline{i}}$$
As in \cite{toen2008homotopical}, it is easily shown that the map
$$|U_{\bullet}|\rightarrow\mathcal{X}$$
is an equivalence (in fact this uses only the fact that $\coprod_{i\in\mathcal{I}}U_{i}\rightarrow\mathcal{X}$ is an epimorphism). Moreover, each $U_{\underline{i}}$ is an $(n-1)$-geometric stack. 

%\begin{definition}
%Let $\mathcal{X}\in\mathbf{Stk}(\mathbf{A},\tau|_{\mathbf{A}})$. We say that $\mathcal{X}$ is a \textit{hypergeometric stack} if there is hypercover $K_{\bullet}\rightarrow\mathcal{X}$ with each $K_{n}$ a disjoint union of objects of $\mathbf{A}$. 
%\end{definition}

%Note that if $\mathcal{X}$ is an admissible $n$-geometric stack, then atlas can be chosen such that each $U_{\underline{i}}$ is an admissible $(n-1)$-geometric stack. 

\subsubsection{Monomorphisms and Schemes}

For the purposes of geometric HKR, we will be particularly interested in schemes.

\begin{definition}
A map $f:\mathcal{X}\rightarrow\mathcal{Y}$ in an $(\infty,1)$-category $\mathbf{M}$ is said to be a \textit{monomorphism} if the map $\mathcal{X}\rightarrow\mathcal{X}\times_{\mathcal{Y}}\mathcal{X}$ is an isomorphism.
\end{definition}

It is easy to see that the class $\textbf{Mon}$ of all monomorphisms is in fact itself closed under composition and pullbacks, and contains all isomorphisms. 

%\begin{definition}
%A map $f:\mathcal{X}\rightarrow\mathcal{Y}$ in $\mathbf{PreStk}(\mathbf{M})$ is said to be an $\mathbf{A}$-\textit{monomorphism} if $f^{\mathbf{A}}$ is a monomorphism in $\mathbf{PreStk}(\mathbf{A})$. 
%\end{definition}

\begin{definition}
Let $(\mathbf{M},\tau,\mathbf{P},\mathbf{A})$ be a relative $(\infty,1)$-geometry tuple. An $n$-geometric -stack $\mathcal{X}$ is said to be an $n$-\textit{scheme} if 
\begin{enumerate}
\item
it has an  $n$-atlas $\{U_{i}\rightarrow\mathcal{X}\}$ such that each $U_{i}\rightarrow\mathcal{X}$ is a monomorphism in $\mathbf{PreStk}(\mathbf{A},\tau|_{\mathbf{A}})$,
\item
the map $\mathcal{X}\rightarrow\mathcal{X}\rightarrow\mathcal{X}$ is a monomorphism in $\mathbf{PreStk}(\mathbf{A},\tau|_{\mathbf{A}})$.
\end{enumerate}
\end{definition}

The full subcategory of $\mathbf{Stk}(\mathbf{A},\tau|_{\mathbf{A}})$ consisting of $n$-schemes is denoted $\mathbf{Sch}_{n}(\mathbf{M},\tau,\mathbf{P},\mathbf{A})$. We also write $\mathbf{Sch}(\mathbf{M},\tau,\mathbf{P},\mathbf{A})\defeq\bigcup_{n=-1}^{\infty}\mathbf{Sch}_{n}(\mathbf{M},\tau,\mathbf{P},\mathbf{A})$.

\subsubsection{Equivalences of Geometric Stacks}

One useful consequence of the inductive definition of geometric stacks (and schemes) is that it allows us to `boost' equivalences defined at the level of affines to equivalences of $n$-geometric stacks. For schemes in particular we have the following.

\begin{proposition}\label{prop:geomequiv}
Let $F,G:\mathbf{Stk}(\mathbf{A},\tau|_{\mathbf{A}})\rightarrow\mathbf{Stk}(\mathbf{N},\gamma)$ be functors, and $\eta:F\rightarrow G$ a natural transformation. Suppose that
\begin{enumerate}
\item
$\eta$ is an equivalence when restricted to $(-1)$-geometric stacks;
\item
$F$ and $G$ commute with fibre products of $(-1)$-geometric stacks;
\item
If $\{U_{i}\rightarrow\mathcal{X}\}_{i\in\mathcal{I}}$ is an $n$-atlas with each $U_{i}\rightarrow\mathcal{X}$ a monomorphism then $\{F(U_{i})\rightarrow F(\mathcal{X})\}_{i\in\mathcal{I}}$ and $\{G(U_{i})\rightarrow G(\mathcal{X})\}_{i\in\mathcal{I}}$ are atlases;
\end{enumerate}
Then $\eta$ is an equivalence when restricted to $\mathbf{Sch}(\mathbf{M},\tau,\mathbf{P},\mathbf{A})$.
\end{proposition}

\begin{proof}
The proof is by induction on $n$. The $n=-1$ case is by assumption. Suppose the result has been proven for some $n-1$ with $n> 0$, and let $\mathcal{X}$ be $n$-geometric. Fix an atlas $\{U_{i}\rightarrow\mathcal{X}\}_{i\in\mathcal{I}}$. Then the map $|U_{\bullet}|\rightarrow\mathcal{X}$ is an equivalence, where $U_{n}=\coprod_{(i_{1},\ldots,i_{n})\in\mathcal{I}^{n}}U_{i_{1}}\times_{\mathcal{X}}\ldots\times_{\mathcal{X}}U_{i_{n}}$. Note that each $U_{i_{1}}\times_{\mathcal{X}}\ldots\times_{\mathcal{X}}U_{i_{n}}$ is $(n-1)$-geometric. Thus the map $\coprod_{(i_{1},\ldots,i_{n})\in\mathcal{I}^{n}}F(U_{i_{1}})\times_{F(\mathcal{X})}\ldots\times_{F(\mathcal{X})}F(U_{i_{n}})\rightarrow\coprod_{(i_{1},\ldots,i_{n})\in\mathcal{I}^{n}}G(U_{i_{1}})\times_{G(\mathcal{X})}\ldots\times_{G(\mathcal{X})}G(U_{i_{n}})$ is an isomorphism. Moreover, since we have assumed that $\{F(U_{i})\rightarrow F(\mathcal{X})\}_{i\in\mathcal{I}}$ and $\{G(U_{i})\rightarrow G(\mathcal{X})\}_{i\in\mathcal{I}}$ are atlases, the maps \newline
$|\coprod_{(i_{1},\ldots,i_{n})\in\mathcal{I}^{n}}F(U_{i_{1}})\times_{F(\mathcal{X})}\ldots\times_{F(\mathcal{X})}F(U_{i_{n}})|\rightarrow F(\mathcal{X})$ and $|\coprod_{(i_{1},\ldots,i_{n})\in\mathcal{I}^{n}}G(U_{i_{1}})\times_{G(\mathcal{X})}\ldots\times_{G(\mathcal{X})}G(U_{i_{n}})|\rightarrow G(\mathcal{X})$ are isomorphisms. 
\end{proof}

To prove the geometric HKR theorem, we will essentially show that for geometry relative to derived algebraic contexts, the conditions of Proposition \ref{prop:geomequiv} for the loop stack and shifted tangent stack functors are satisfied.

\subsection{The Loop Stack}
One side of the geometric HKR theorem will involve the loop stack construction. We may embed the category $\textbf{Grpd}$ in $\mathbf{Stk}(\mathbf{M},\tau)$  by regarding any simplicial set $S$ as a constant functor on $\mathbf{M}$ with value $S$. For $S^{1}$ the simplicial circle and $\mathcal{X}\in\mathbf{PreStk}(\mathbf{M},\tau)$ we define the \textit{loop pre-stack} 
$$\mathcal{L}(\mathcal{X}) \defeq \underline{\mathsf{Map}}(S^1, \mathcal{X}) \in \mathbf{PreStk}(\mathbf{M},\tau)$$
 Note that if $\mathcal{X}$ is a stack then so is $\mathcal{L}(\mathcal{X})$.

\begin{lemma}
There is an equivalence of pre-stacks \(\mathcal{L}(X) \cong X \times_{X \times X} X\).
\end{lemma}

\begin{proof}
This follows from writing the circle as a homotopy fibered coproduct \(S^1 \cong * \underset{* \coprod *} \coprod *\) and using that internal \(\Hom\) takes a homotopy pushout to a homotopy pullback. 
\end{proof}

%In particular if $\mathcal{X}\cong\mathsf{Spec}(A)$ then we have 
%$$\mathcal{L}(\mathsf{Spec}(A))\cong\mathsf{Spec}(A\otimes_{A\otimes A}A)$$

Suppose now that $(\mathbf{M},\tau,\mathbf{P},\mathbf{A})$ is a relative $(\infty,1)$-geometry tuple, and that $\mathcal{X}$ is an  $n$-geometric stack. As a fibre product $\mathcal{L}(X)$ is also an $n$-geometric stack. 

\begin{remark}
If $\mathcal{X}\cong\mathrm{Map}(-,U)$ is representable then $\mathcal{L}(\mathcal{X})\cong\mathrm{Map}(-,U\times_{U\times U}U)$ is representable. However even if $U\in\mathbf{A}$ then $U\times_{U\times U}U$ is not necessarily in $\mathbf{A}$. It is however still $\mathbf{A}$-admissible, as the category of stacks is closed under limits in the category of prestacks.
\end{remark}

Let $\mathcal{X}$ be an $n$-geometric stack which is the realisation of an $(n,\mathbf{P})$-hypergroupoid $\coprod_{\mathcal{I}_{n}}U_{i_{m}}$. Then $\mathcal{X}_{\mathcal{X}\times\mathcal{X}}\mathcal{X}$ is the realisation of the hypegroupoid in $(n-1)$-geometric stacks $\coprod_{i_{m},j_{n}\in\mathcal{I}_{n}}U_{i_{m}}\times_{\mathcal{X}\times\mathcal{X}}U_{j_{n}}$.

\begin{proposition}
Suppose that each $U_{i_{n}}\rightarrow\mathcal{X}$ is a monomorphism, i.e. the map $U_{i_{n}}\times_{\mathcal{X}}U_{i_{n}}\rightarrow U_{i_{n}}$ is an equivalence. Then $U_{i}\times_{\mathcal{X}\times\mathcal{X}}U_{j}\cong (U_{i}\times_{\mathcal{X}}U_{j})\times_{\mathcal{X}\times\mathcal{X}}(U_{i}\times_{\mathcal{X}}U_{j})$. 
\end{proposition}

\begin{proof}
Let $U\rightarrow\mathcal{X}$ and $V\rightarrow\mathcal{X}$ be monomorphisms. Then 
\begin{align*}
U\times_{\mathcal{X}\times\mathcal{X}}V & \cong (U\times_{\mathcal{X}}U)\times_{\mathcal{X}\times\mathcal{X}}(V\times_{\mathcal{X}}V)\\
&\cong (U\times_{\mathcal{X}}V)\times_{\mathcal{X}\times\mathcal{X}}(U\times_{\mathcal{X}}V)
\end{align*}
Now $(U\times_{\mathcal{X}}V)\rightarrow\mathcal{X}$ is a monomorphism, so $(U\times_{\mathcal{X}}V)\times (U\times_{\mathcal{X}}V)\rightarrow\mathcal{X}\times\mathcal{X}$ is as well. Thus the map 
$$(U\times_{\mathcal{X}}V)\times_{U\times_{\mathcal{X}}V\times U\times_{\mathcal{X}}V}(U\times_{\mathcal{X}}V)\rightarrow(U\times_{\mathcal{X}}V)\times_{\mathcal{X}\times\mathcal{X}}(U\times_{\mathcal{X}}V)$$
is an equivalence.
\end{proof}

\begin{corollary}\label{cor:epiloop}
Let $\mathcal{X}$ be an $n$-scheme with atlas $\{U_{i}\rightarrow\mathcal{X}\}_{i\in\mathcal{I}}$ such that for each $i\in\mathcal{I}$ the map $U_{i}\rightarrow\mathcal{X}$ is a monomorphism. Then $\mathcal{L}(\mathcal{X})$ is the geometric realisation of the Segal groupoid
$$n\mapsto\coprod_{(i_{1},\ldots,i_{n})}\mathcal{L}(U_{i_{1}}\times_{\mathcal{X}}\ldots\times_{\mathcal{X}}U_{i_{n}})$$
% $n\mapsto\coprod_{(i_{1},\ldots,i_{n}),(j_{1},\ldots,j_{n})\in\mathcal{I}^{n}}\mathcal{L}(U_{i_{1}}\times_{\mathcal{X}}\ldots\times_{\mathc\al{X}}U_{i_{n}}\times_{\mathcal{X}}U_{j_{1}}\times_{\mathcal{X}}\ldots\times_{\mathcal{X}}U_{j_{n}})$.
\end{corollary}

\begin{proof}
$\mathcal{X}\times_{\mathcal{X}\times\mathcal{X}}\mathcal{X}$ is the geometric realisation of the hypergroupoid 
$$n\mapsto\coprod_{(i_{1},\ldots,i_{n}),(j_{1},\ldots,j_{n})\in\mathcal{I}^{n}}(U_{i_{1}}\times_{\mathcal{X}}\ldots\times_{\mathcal{X}}U_{i_{n}})\times_{\mathcal{X}\times\mathcal{X}}(U_{j_{1}}\times_{\mathcal{X}}\ldots\times_{\mathcal{X}}U_{j_{n}})$$
For convenience of notation, we write $U_{\underline{i}}\defeq U_{i_{1}}\times_{\mathcal{X}}\ldots\times_{\mathcal{X}}U_{i_{n}}$ and $U_{\underline{j}}\defeq U_{j_{1}}\times_{\mathcal{X}}\ldots\times_{\mathcal{X}}U_{j_{n}}$. Now $U_{\underline{i}}\times_{\mathcal{X}\times\mathcal{X}}U_{\underline{j}}$ is equivalent to $U_{\underline{i}}\times_{\mathcal{X}}U_{\underline{j}}\times_{U_{\underline{i}}\times_{\mathcal{X}}U_{\underline{j}}\times U_{\underline{i}}\times_{\mathcal{X}}U_{\underline{j}}}U_{\underline{i}}\times_{\mathcal{X}}U_{\underline{j}}\cong\mathcal{L}(U_{\underline{i}}\times_{\mathcal{X}}U_{\underline{j}})$.
\end{proof}

\begin{corollary}\label{cor:atlasloop}
Let $\mathcal{X}$ be an $n$-scheme with $n$-atlas $\{U_{i}\rightarrow\mathcal{X}\}_{i\in\mathcal{I}}$ such that for each $i\in\mathcal{I}$ the map $U_{i}\rightarrow\mathcal{X}$ is a monomorphism. Then $\{\mathcal{L}(U_{i})\rightarrow\mathcal{L}(\mathcal{X})\}_{i\in\mathcal{I}}$ is an $n$-atlas.
\end{corollary}

\begin{proof}
Corollary \ref{cor:epiloop} implies that $\coprod_{\mathcal{I}}\mathcal{L}(U_{i})\rightarrow\mathcal{L}(\mathcal{X})$ is an epimorphism. Consider the commutative diagram
\begin{displaymath}
\xymatrix{
\mathcal{L}(U_{i})\ar[d]\ar[r] & \mathcal{L}(\mathcal{X})\times_{\mathcal{X}}U_{i}\ar[d]\ar[r] & U_{i}\ar[d]\\
U_{i}\times_{\mathcal{X}}\mathcal{L}(\mathcal{X})\ar[d]\ar[r] & \mathcal{L}(\mathcal{X})\ar[d]\ar[r] & \mathcal{X}\ar[d]\\
U_{i}\ar[r] & \mathcal{X}\ar[r] & \mathcal{X}\times\mathcal{X}
}
\end{displaymath}
all squares in this diagram are cartesian. Since the map $U_{i}\rightarrow\mathcal{X}$ is in $n-\mathbf{P}$, it follows that $\mathcal{L}(U_{i})\rightarrow\mathcal{L}(\mathcal{X})$ is in $n-\mathbf{P}$.
\end{proof}

\subsection{The Shifted Tangent Stack}

The other side of the geometric HKR equivalence involves the shifted tangent stack. We can define this completely abstractly.
$$T(-)[-1]:\mathbf{PreStk}(\mathbf{M})\rightarrow\mathbf{PreStk}(\mathbf{M})$$
$$\mathcal{X}\mapsto\mathrm{Map}(\mathsf{Spec}(k\prod_{k\prod k}k),\mathcal{X})$$
However, not much can be said about the geometry of this object at this level of generality. This brings us to the next section.

\section{Geometry Relative to a Derived Algebraic Context}\label{sec:monoidal_geometry}

We now specialise to geometry relative to a derived algebraic context $(\mathbf{C},\mathbf{C}_{\ge0},\mathbf{C}_{\le0},\mathbf{C}^{0})$. The category of affine stacks relative to $\mathbf{C}$ is
$$\mathbf{Aff}^{cn}_{\mathbf{C}}\defeq(\mathsf{DAlg}^{cn}(\mathbf{C}))^{op}$$
For the remainder of the paper, this will function as the category $\mathbf{M}$ from the previous section.

\subsection{\'{E}tale Maps}

Let us define some properties of maps which will be useful for defining topologies on affine stacks. 
\begin{definition}[\cite{toen2008homotopical} Definition 1.2.6.1]
Let $f:A\rightarrow B$ be a map in $\mathsf{DAlg}^{cn}(\mathbf{C})$.
\begin{enumerate}
\item
$f$ is said to be an \textit{epimorphism} if the map $B\otimes_{A}B\rightarrow B$ is an equivalence;
\item
$f$ is said to be \textit{formally unramified} if $\mathsf{L}\Omega^{1}_{B\big\slash A}\cong 0$; 
\item
$f$ is said to be \textit{formally \'{e}tale} if the natural map
$$\mathsf{L}\Omega^{1}_{A}\otimes_{A}B\rightarrow\mathsf{L}\Omega^{1}_{B}$$
is an equivalence.
\end{enumerate}
\end{definition}

Since the suspension functor on $\mathbf{C}_{\ge0}$ is fully faithful, we get the following results.

\begin{proposition}[\cite{toen2008homotopical} Proposition 1.2.6.5 (1),(3)]
\begin{enumerate}
\item
An epimorphism is formally \'{e}tale;
\item
Let $f:A\rightarrow B$ and $g:B\rightarrow C$ be maps of commutative monoids with $g$ formally \'{e}tale. Then the map $\mathsf{L}\Omega^{1}_{B\big\slash A}\otimes_{B}C\rightarrow\mathsf{L}\Omega^{1}_{C\big\slash A}$ is an equivalence. In particular, in this case a map is formally unramified if and only if it is formally \'{e}tale.
\end{enumerate}
\end{proposition}

\subsection{Weak relative DAG contexts}\label{subsubsec:RDAG}

Now we fix a full subcategory $\mathbf{A}\subset\mathbf{Aff}^{cn}_{\mathbf{C}}$ stable by equivalences. Before establishing our formalism for relative derived geometry, we need some terminology.

\begin{definition}
Let $(\mathbf{C},\mathbf{C}_{\ge0},\mathbf{C}_{\le0},\mathbf{C}^{0})$ be a derived algebraic context, $\tau$ a topology on $\mathbf{Aff}^{cn}_{\mathbf{C}}$ and $B\in\mathrm{DAlg}^{cn}(\mathbf{C})$. A $B$-module $M$ is said to \textit{satisfy  weak \v{C}ech descent} if for any cover $\{\mathrm{Spec}(A_{i})\rightarrow\mathrm{Spec}(A)\}_{i\in\mathcal{I}}$ in $\tau$ the natural map from $M$ to the homotopy limit of the cosimplicial diagram
$$n\mapsto\prod_{(i_{0},\ldots,i_{n})\in\mathcal{I}^{n+1}}M\otimes_{A}A_{i_{0}}\otimes_{A}A_{i_{1}}\otimes_{A}\ldots\otimes_{A}A_{i_{n}}$$
is an equivalence. $\tau$ is said to \textit{satisfy weak \v{C}ech descent} if for any $B\in\mathrm{DAlg}^{cn}(\mathbf{C})$, any $B$-module satisfies weak \v{C}ech descent.
\end{definition}

Let $f:A\rightarrow B$ be a map of commutative monoids, $M$ a stable $B$-module, and $d\in\pi_{0}(\mathrm{Der}_{A}(B,M))$. This defines a map of commutative monoids $d:B\rightarrow B\oplus M$. There is also the zero section map $s:B\rightarrow B\oplus M$. We consider the homotopy pullback
\begin{displaymath}
\xymatrix{
C\ar[d]\ar[r] & B\ar[d]^{d}\\
B\ar[r]^{s} & B\oplus M
}
\end{displaymath}
As explained in \cite{toen2008homotopical} Section 1.2.1. The homotopy fibre of the map $C\rightarrow B$ is equivalent to the loops module $\Omega M$, so that $C$ is a square-zero extension of $B$ by $\Omega M$, which we denote by $B\oplus_{d}\Omega M$. Here $\Omega M$ is the \textit{loop module} of $M$ given by the pullback $0\times_{M}0$. 

\begin{definition}
A \textit{good system of module categories} $\mathbf{F}$ on $\mathbf{A}\subseteq\mathbf{Aff}^{cn}_{\mathbf{C}}$ is an assignment to each $\mathrm{Spec}(B)\in \mathbf{A}$ a full subcategory $\mathbf{F}_{B}\subset\mathbf{Mod}_{B}$ such that
\begin{enumerate}
\item
$A\in\mathbf{F}_{A}$ for any $\mathrm{Spec}(A)\in\mathbf{A}$;
\item
$\mathsf{L}\Omega^{1}_{A}\in\mathbf{F}_{A}$ for any $\mathrm{Spec}(A)\in\mathbf{A}$;
\item
$\mathbf{F}_{A}$ is closed under isomorphisms, finite colimits, and retracts;
\item
for any map $f:B\rightarrow C$ and $M\in\mathbf{F}_{B}$, $C\otimes_{B}M\in\mathbf{F}_{C}$;
\item
any $F\in\mathbf{F}_{A}$ satisfies weak \v{C}ech descent;
\item
For $M\in\mathbf{F}_{A}^{1}$ and $\mathrm{Spec}(A)\in\mathbf{A}$, $\mathrm{Spec}(A\oplus M)\in\mathbf{A}$. 
\end{enumerate}
\end{definition}

For integers $n$, we shall write
$$\mathbf{F}_{A}^{n}\defeq\mathbf{F}_{A}\times_{\mathbf{Mod}_{A}}\mathbf{Mod}^{n}_{A}$$
and in particular, $\mathbf{F}_{A}^{cn}\defeq\mathbf{F}_{A}^{0}$.

\begin{definition}
A \textit{weak relative DAG context} is a tuple $(\mathbf{C},\mathbf{C}_{\ge0},\mathbf{C}_{\le0},\mathbf{C}^{0},\tau,\mathbf{P},\mathbf{A},\mathbf{F})$ where
\begin{enumerate}
\item
$(\mathbf{C},\mathbf{C}_{\ge0},\mathbf{C}_{\le0},\mathbf{C}^{0})$ is a derived algebraic context;
\item
$\tau$ is a Grothendieck topology on $\mathbf{Aff}^{cn}_{\mathbf{C}}$;
\item
$\mathbf{P}$ is a class of maps in $\mathsf{Ho}(\mathbf{Aff}^{cn}_{\mathbf{C}})$ which is stable by composition, pullback, and contains isomorphisms;
\item
$\mathbf{P}$ consists of formally \'{e}tale morphisms;
\item
$\mathbf{F}$ is a good system of modules on $\mathbf{A}$,
%\item
%$\tau^{qc}$ satisfies weak \v{C}ech descent.
%\item
%If $\mathrm{Spec}(A)$ in $\mathbf{A}$ then $\mathsf{L}\Omega^{1}_{A}\in\mathbf{Mod}^{\tau}_{A}$.
%\item
%$\tau$ satisfies weak \v{C}ech descent relative to $\mathbf{A}$ and $\mathbf{F}$.
\end{enumerate}
such that $(\mathbf{Aff}^{cn}_{\mathbf{C}},\tau,\mathbf{P},\mathbf{A})$ is a strong relative $(\infty,1)$-geometry tuple. 
\end{definition}

\subsection{Relative Cotangent Complexes and Tangent Stacks}

In the setting of geometry relative to a weak relative DAG context, it is possible to meaningfully define cotangent complexes of stacks, and tangent stacks. Here we follow \cite{toen2008homotopical} Section 1.4.1.

\begin{definition}[\cite{toen2008homotopical}, Definition 1.4.1.4]
Let $\mathcal{X}\in\mathbf{PreStk}(\mathbf{Aff}^{cn}_{\mathbf{C}},)$, $\mathrm{Spec}(\mathbf{A})$, and $x:\mathrm{Spec}(A)\rightarrow\mathcal{X}$ a map. Define 
$$\mathrm{Der}_{\mathcal{X}}(\mathrm{Spec}(A),-):\mathbf{F}^{cn}_{A}\rightarrow\textbf{sSet}$$
to be the functor
$$\mathrm{Map}_{{}_{\mathrm{Spec}(A)\big\backslash\mathbf{Aff}^{cn}_{\mathbf{C}}}}(\mathrm{Spec}(A\oplus M),\mathcal{X}).$$
\end{definition}

\begin{definition}[\cite{toen2008homotopical} Section 1.4.1]
Let $f:\mathcal{X}\rightarrow\mathcal{Y}$ be a map in $\mathbf{PreStk}(\mathbf{Aff}^{cn}_{\mathbf{C}})$ and $x:\mathrm{Spec}(A)\rightarrow\mathcal{X}$ a map. By composition we get a map
$$\mathrm{Der}_{\mathcal{X}}(\mathrm{Spec}(A),-)\rightarrow \mathrm{Der}_{\mathcal{Y}}(\mathrm{Spec}(A),-).$$
We define 
$$\mathrm{Der}_{\mathcal{X}\big\slash\mathcal{Y}}(\mathrm{Spec}(A),-):\mathbf{F}^{\mathrm{cn}}_{A}\rightarrow\textbf{sSet}$$
to be the fibre of the map $\mathrm{Der}_{\mathcal{X}}(\mathrm{Spec}(A),-)\rightarrow \mathrm{Der}_{\mathcal{Y}}(\mathrm{Spec}(A),-)$.
\end{definition}

%\begin{definition}
%A map $f:\mathcal{X}\rightarrow\mathcal{Y}$ in $\mathbf{Stk}(\mathbf{Aff}^{cn}_{\mathbf{C}},\tau)$ is said to \textit{have a relative cotangent complex at }$x:\mathrm{Spec}(A)\rightarrow\mathcal{X}$\textit{ relative to }$\mathbf{F}$ if there is a connective $A$-module $\mathsf{L}\Omega^{1}_{\mathcal{X}\big\slash\mathcal{Y},x}$ in $\mathbf{F}_{\mathbf{A}}$ and an equivalence 
%$$\mathrm{Der}_{\mathcal{X}\big\slash\mathcal{Y}}(\mathrm{Spec}(A),-)\cong\mathrm{Map}_{\mathbf{Mod}_{A}}(\mathsf{L}\Omega^{1}_{\mathcal{X}\big\slash\mathcal{Y},x},-)$$
%\end{definition}

\begin{definition}[\cite{toen2008homotopical}, Definition 1.4.1.14]
A map $f:\mathcal{X}\rightarrow\mathcal{Y}$ in $\mathbf{PreStk}(\mathbf{Aff}^{cn}_{\mathbf{C}})$ is said to \textit{have a relative cotangent complex at }$x:\mathrm{Spec}(A)\rightarrow\mathcal{X}$ if there is a $(-n)$-connective $A$-module $\mathsf{L}\Omega^{1}_{\mathcal{X}\big\slash\mathcal{Y},x}\in\mathbf{F}_{A}$ and an equivalence 
$$\mathrm{Der}_{\mathcal{X}\big\slash\mathcal{Y}}(\mathrm{Spec}(A),-)\cong\mathrm{Map}_{\mathbf{Mod}_{A}}(\mathsf{L}\Omega^{1}_{\mathcal{X}\big\slash\mathcal{Y},x},-)$$
\end{definition}
Let

\begin{displaymath}
\xymatrix{
\mathrm{Spec}(B)\ar[dr]^{y}\ar[rr]^{u} & & \mathrm{Spec}(A)\ar[dl]^{x}\\
& \mathcal{X} & 
}
\end{displaymath}
be a commutative diagram with $u$ a morphism in $\mathbf{A}$. Suppose that $f:\mathcal{X}\rightarrow\mathcal{Y}$ has a relative cotangent complex at both $x$ and $y$. Using the universal properties of $\mathsf{L}\Omega^{1}_{\mathcal{X}\big\slash\mathcal{Y},x}$ and $\mathsf{L}\Omega^{1}_{\mathcal{X}\big\slash\mathcal{Y},y}$, there is a natural map.
$$u^{*}:\mathsf{L}\Omega^{1}_{\mathcal{X}\big\slash\mathcal{Y},x}\otimes_{A}B\rightarrow\mathsf{L}\Omega^{1}_{\mathcal{X}\big\slash\mathcal{Y},y}$$

%\begin{definition}[\cite{toen2008homotopical} Definition 1.4.1.15]
%A morphism $f:\mathcal{X}\rightarrow\mathcal{Y}$ of stacks is said to have a \textit{global cotangent complex relative to }$\mathbf{A}$\textit{ and }$\mathbf{F}$
%\begin{enumerate}
%\item
%for any $A\in\mathbf{A}$ and any $x:\mathrm{Spec}(A)\rightarrow\mathcal{X}$, $f$ has a cotangent complex at $x$ relative to $\mathbf{F}$.
%\item
%For any morphism $u:\mathrm{Spec}(B)\rightarrow\mathrm{Spec}(A)$ in $\mathbf{A}$, and any commutative diagram
%\begin{displaymath}
%\xymatrix{
%\mathrm{Spec}(B)\ar[dr]^{y}\ar[rr]^{u} & & \mathrm{Spec}(A)\ar[dl]^{x}\\
%& \mathcal{X} & 
%}
%\end{displaymath}
%the map 
%$$u^{*}:\mathsf{L}\Omega^{1}_{\mathcal{X}\big\slash\mathcal{Y},x}\otimes_{A}B\rightarrow\mathsf{L}\Omega^{1}_{\mathcal{X}\big\slash\mathcal{Y},y}$$
%is an equivalence. 
%\end{enumerate}
%\end{definition}

\begin{definition}[\cite{toen2008homotopical} Definition 1.4.1.15]
A morphism $f:\mathcal{X}\rightarrow\mathcal{Y}$ of prestacks is said to have a \textit{global cotangent complex relative to }$\mathbf{A}$ if
\begin{enumerate}
\item
for any $A\in\mathbf{A}$ and any $x:\mathrm{Spec}(A)\rightarrow\mathcal{X}$, $f$ has a relative cotangent complex at $x$ 
\item
For any morphism $u:\mathrm{Spec}(B)\rightarrow\mathrm{Spec}(A)$ in $\mathbf{A}$, and any commutative diagram
\begin{displaymath}
\xymatrix{
\mathrm{Spec}(B)\ar[dr]^{y}\ar[rr]^{u} & & \mathrm{Spec}(A)\ar[dl]^{x}\\
& \mathcal{X} & 
}
\end{displaymath}
the map 
$$u^{*}:\mathsf{L}\Omega^{1}_{\mathcal{X}\big\slash\mathcal{Y},x}\otimes_{A}B\rightarrow\mathsf{L}\Omega^{1}_{\mathcal{X}\big\slash\mathcal{Y},y}$$
is an equivalence. 
\end{enumerate}
\end{definition}

The following is proven exactly as in \cite{toen2008homotopical}.

\begin{proposition}[\cite{toen2008homotopical}*{Lemma 1.4.2.16 (3)}]\label{prop:globalcotinfad}
Let $f:\mathcal{X}\rightarrow\mathcal{Y}$ be a map of prestacks such that for any map $\mathrm{Spec}(A)\rightarrow\mathcal{Y}$ with $\mathrm{Spec}(A)$ in $\mathbf{A}$, the map
$$\mathcal{X}\times_{\mathcal{Y}}\mathrm{Spec}(A)\rightarrow\mathcal{Y}$$
has a global contangent complex relative to $\mathbf{A}$. Then $f$ has a global cotangent complex relative to $\mathbf{A}$, and there is a natural isomorphism
$$\mathsf{L}\Omega^{1}_{\mathcal{X}\big\slash\mathcal{Y},x}\cong\mathsf{L}\Omega^{1}_{\mathcal{X}\times_{\mathcal{Y}}\mathrm{Spec}(A)\big\slash\mathrm{Spec}(A),x}$$
\end{proposition}

The following is essentially  tautological.

\begin{proposition}[\cite{toen2008homotopical} Lemma 1.4.1.16 (1)]
Let $f:\mathcal{X}\rightarrow\mathcal{Y}$ be a map in $\mathbf{PreStk}(\mathbf{Aff}^{cn}_{\mathbf{C}},\tau)$ and let $x:\mathrm{Spec}(A)\rightarrow\mathcal{X}$ be a map. If $\mathcal{X}$ and $\mathcal{Y}$ have contangent complexes at $x$, then $f:\mathcal{X}\rightarrow\mathcal{Y}$ has a contangent complex at $x$, and there is a  cofibre sequence.
$$\mathsf{L}\Omega^{1}_{\mathcal{Y},f(x)}\rightarrow\mathsf{L}\Omega^{1}_{\mathcal{X},x}\rightarrow\mathsf{L}\Omega^{1}_{\mathcal{X}\big\slash\mathcal{Y},x}$$
\end{proposition}

With our assumptions, the next result can be proven formally as in \cite{toen2008homotopical}. 

\begin{lemma}[\cite{toen2008homotopical} Lemma 1.4.1.12]\label{lem:pullbackgt}
Let
\begin{displaymath}
\xymatrix{
\mathrm{Spec}(B)\ar[d]^{u}\ar[r]^{y} & \mathcal{X}'\ar[d]\ar[r] & \mathcal{Y}'\ar[d]\\
\mathrm{Spec}(A)\ar[r]^{x} & \mathcal{X}\ar[r] & \mathcal{Y}
}
\end{displaymath}
be a commutative diagram in $\mathbf{PreStk}(\mathbf{Aff}^{cn}_{\mathbf{C}},\tau)$ with the right-hand square being cartesian and $u$ being a map in $\mathbf{A}$. Suppose that $\mathcal{X}$ and $\mathcal{Y}$ have global cotangent complexes relative to $\mathbf{A}$. Then the square
\begin{displaymath}
\xymatrix{
B\otimes_{A}\mathsf{L}\Omega^{1}_{\mathcal{Y},x}\ar[d]\ar[r] &B\otimes_{A}\mathsf{L}\Omega^{1}_{\mathcal{X},x}\ar[d]\\
\mathsf{L}\Omega^{1}_{\mathcal{Y'},y}\ar[r] & \mathsf{L}\Omega^{1}_{\mathcal{X}',y}
}
\end{displaymath}
is cartesian in $\mathbf{Mod}_{B}$ (and hence in $\mathbf{F}_{B}$). 
\end{lemma}

\begin{proof}
This is as in \cite{toen2008homotopical}, and follows from the pullback square
\begin{displaymath}
\xymatrix{
\mathrm{Der}_{\mathcal{X}'}(\mathrm{Spec}(B),M)\ar[d]\ar[r] & \mathrm{Der}_{\mathcal{Y}'}(\mathrm{Spec}(B),M)\ar[d]\\
\mathrm{Der}_{\mathcal{X}}(\mathrm{Spec}(B),M)\cong\mathrm{Der}_{\mathcal{X}}(\mathrm{Spec}(A),M)\ar[r] & \mathrm{Der}_{\mathcal{Y}}(\mathrm{Spec}(B),M)\cong\mathrm{Der}_{\mathcal{Y}}(\mathrm{Spec}(A),M).
}
\end{displaymath}
\end{proof}

An important class of maps which have relative cotangent complexes is the class of geometrically \'{e}tale maps.

\begin{definition}
A map $f:\mathcal{X}\rightarrow\mathcal{Y}$ of prestacks is said to be \textit{geometrically \'{e}tale} if for any map $x:\mathrm{Spec}(A)\rightarrow\mathcal{X}$ with $\mathrm{Spec}(A)\in\mathbf{A}$, $\mathsf{L}\Omega^{1}_{\mathcal{X}\big\slash\mathcal{Y},x}$ is zero. 
\end{definition}

Note that the class of geometrically \'{e}tale maps is closed under pullback and composition, and contains isomorphisms.
%\begin{proposition}
%Let $f:\mathcal{X}\rightarrow\mathcal{Y}$ be a map in $\mathbf{Stk}(\mathbf{Aff}^{cn}_{\mathbf{C}},\tau^{qc})$ which is $n-\mathbf{P}^{\'{e}t}$ for some $n$. Then $f$ is geometrically \'{e}tale.
%\end{proposition}

%\begin{proof}
%Let $\mathrm{Spec}(B)\rightarrow\mathcal{Y}$ be a map, and consider the pullback $\mathcal{X}\times_{\mathcal{Y}}\mathrm{Spec}(B)\rightarrow\mathrm{Spec}(B)$. Let $x:\mathrm{Spec}(A)\rightarrow\mathcal{X}\times_{\mathcal{Y}}\mathrm{Spec}(B)$ be a map. We then have an isomorphism
%$$\mathsf{L}\Omega^{1}_{\mathcal{X}\big\slash\mathcal{Y},x}\cong\mathsf{L}\Omega^{1}_{\mathcal{X}\times_{\mathcal{Y}}\mathrm{Spec}(B)\big\slash\mathrm{Spec}(B),x}$$
%Thus in the proof we may assume that both $\mathcal{X}$ and $\mathcal{Y}$ are geometric, and in fact that $\mathcal{Y}$ is affine. 
%
%Suppose that $f$ is formally \'{e}tale. It follows from Proposition \ref{prop:formallyetcot} that $\mathsf{L}\Omega^{1}_{\mathcal{X}\big\slash\mathcal{Y},x}$ for any $\mathrm{Spec}(A)\rightarrow\mathcal{Y}$.
%
%Conversely suppose that $\mathsf{L}\Omega^{1}_{\mathcal{X}\big\slash\mathcal{Y},x}\cong 0$ for any map $x:\mathrm{Spec}(A)\rightarrow\mathcal{X}$ with $\mathrm{Spec}(A)\in\mathbf{A}$. 
%\end{proof}

\begin{example}
Let $f:\mathcal{X}\rightarrow\mathcal{Y}$ be a monomorphism in $\mathbf{PreStk}(\mathbf{A})$. Then $f$ is geometrically \'{e}tale. Indeed, by construction of the contangent complex, for any $x:\mathrm{Spec}(A)\rightarrow\mathcal{X}$ the fibre of $\mathcal{X}\rightarrow\mathcal{Y}$ over $x$ is contractible. Thus the fibre of $\mathrm{Der}_{\mathcal{X}}(\mathrm{Spec}(A),-)\rightarrow\mathrm{Der}_{\mathcal{Y}}(\mathrm{Spec}(A),-)$ is contractible, and
$\mathsf{L}\Omega^{1}_{\mathcal{X}\big\slash\mathcal{Y},x}\cong0$. 
\end{example}

\begin{definition}\label{defn:geomet}
An $n$-geometric stack $\mathcal{X}$ in $\mathbf{Stk}_{n}(\mathbf{M},\tau,\mathbf{P},\mathbf{A})$ is said to be \textit{geometrically \'{e}tale} if
\begin{enumerate}
\item
 it has an atlas $\{U_{i}\rightarrow\mathcal{X}\}_{i\in\mathcal{I}}$ such that each $U_{i}\rightarrow\mathcal{X}$ is geometrically \'{e}tale.
  \item
  The map $\mathcal{X}\rightarrow\mathcal{X}\times\mathcal{X}$ is geometrically \'{e}tale.
\end{enumerate}
\end{definition}

%\begin{prop}
%The category of geometrically \'{e}tale stacks is closed under fibre products. 
%\end{prop}
% Later we will show that in many circumstances formally \'{e}tale stacks are geometrically \'{e}tale. In complete generality schemes are geometrically \'{e}tale.

\begin{corollary}
Schemes in $\mathbf{Sch}(\mathbf{M},\tau,\mathbf{P},\mathbf{A})$ are geometrically \'{e}tale.
\end{corollary}

In the next section we will show that many more geometric stacks and representable maps of interest have cotangent complexes.

\section{Obstruction Theories}\label{sec:obstruction_theories}

In this section we introduce shifted tangent stacks and obstruction theories, following \cite{toen2008homotopical} Section 1.4.2. This will allow us to put a geometric structure on the shifted tangent stack of a geometric stack. Shifted tangent stacks geometrise the (shifted) cotangent complex and in particular, the de Rham complex. We fix a weak relative DAG context $(\mathbf{C},\mathbf{C}_{\ge0},\mathbf{C}_{\le0},\mathbf{C}^{0},\tau,\mathbf{P},\mathbf{A},\mathbf{F})$.

\subsection{Obstruction Theories}

In this section we discuss obstruction theories, which we will need to construct an atlas for the shifted tangent stack of a scheme.

\begin{definition}[ \cite{toen2008homotopical} Definition 1.4.2.1]
\begin{enumerate}
\item
A pre-stack $\mathcal{X}\in\mathbf{PreStk}(\mathbf{Aff}^{cn}_{\mathbf{C}})$ is said to be \textit{inf-cartesian relative to }$\mathbf{A}$ if for any $A$ with $\mathrm{Spec}(A)\in\mathbf{A}$, any $1$-connective $M\in\mathbf{F}^{k}_{\mathbf{A}}$, and any derivation $d\in\pi_{0}\mathrm{Der}(A,M)$, the square 
\begin{displaymath}
\xymatrix{
\mathcal{X}(A\oplus_{d}\Omega M)\ar[d]\ar[r] & \mathcal{X}(A)\ar[d]^{d}\\
\mathcal{X}(A)\ar[r]^{s} & \mathcal{X}(A\oplus M)
}
\end{displaymath}
is cartesian in $\textbf{sSet}$.
\item
A pre-stack $\mathcal{X}$ \textit{has an obstruction theory relative to }$\mathbf{A}$ if it has a global cotangent complex relative to $\mathbf{A}$, and it is inf-cartesian relative to $\mathbf{A}$.
\end{enumerate}
\end{definition}

Let us show that any $\widetilde{\mathrm{Spec}}(A)\in\mathbf{A}$ has an obstruction theory relative to $\mathbf{A}$, where $\widetilde{\mathrm{Spec}}(A)$ is the stackification of $\mathrm{Spec}(A)$ in $\mathbf{Stk}(\mathbf{Aff}^{cn}_{\mathbf{C}},\tau)$. The following is proven exactly as in \cite{toen2008homotopical} Proposition 1.4.2.4 (1).
\begin{proposition}
Let $\mathrm{Spec}(A)\in\mathbf{A}$, $M\in\mathbf{F}_{A}^{1}$ and $d:A\rightarrow M$ a derivation. Then for any $\mathrm{Spec}(B)\in\mathbf{A}$ the diagram
\begin{displaymath}
\xymatrix{
\mathrm{Map}_{\mathrm{DAlg}^{cn}(\mathbf{C})}(B,A\oplus_{d}\Omega M)\ar[d]\ar[r] & \mathrm{Map}_{\mathrm{DAlg}^{cn}(\mathbf{C})}(B,A)\ar[d]\\
\mathrm{Map}_{\mathrm{DAlg}^{cn}(\mathbf{C})}(B,A)\ar[r]& \mathrm{Map}_{\mathrm{DAlg}^{cn}(\mathbf{C})}(B,A\oplus M)
}
\end{displaymath}
is cartesian in $\mathbf{sSet}$.
\end{proposition}

The next result is a simplification of Lemma 1.4.3.9 in \cite{toen2008homotopical}. Let $\mathrm{Spec}(A)$ be in $\mathbf{A}$, $M\in\mathbf{F}_{A}^{1}$,  $d:A\rightarrow M$ a derivation, and $K'_{\bullet}\rightarrow\mathrm{Spec}(A)$ a hypercover. Now each $K'_{n}$ is a coproduct of representables, $K'_{n}\cong\coprod_{\mathcal{I}_{n}}\mathrm{Spec}(A_{i_{n}})$. Denote $K^{d}$ the simplicial object with $K_{n}=\coprod_{\mathcal{I}_{n}}\mathrm{Spec}(A_{i_{n}}\oplus _{d_{i_{n}}}\Omega(A_{i_{n}}\otimes_{A}M))$, where $d_{i_{n}}:A_{i_{n}}\rightarrow A_{i_{n}}\otimes_{A}M$ is the unique derivation extending $d$. 

\begin{proposition}
Let $\mathrm{Spec}(A)\in\mathbf{A}$, $M\in\mathbf{F}_{A}^{1}$ and $d:A\rightarrow M$ a derivation. Then the assignment $K\mapsto K^{d}$ furnishes an equivalence of categories between the category of hypercovers of $\mathrm{Spec}(A)$, and the category of hypercovers of $\mathrm{Spec}(A\oplus_{d}\Omega M)$.
\end{proposition}

This allows us to prove the following.

\begin{corollary}
Let $\mathrm{Spec}(B)\in\mathbf{A}$. Then $\widetilde{\mathrm{Spec}(B)}$ has an obstruction theory relative to $\mathbf{A}$. In particular for any $\mathrm{Spec}(A)$in $\mathbf{A}$, $M\in\mathbf{F}_{A}^{1}$ and $d:A\rightarrow M$ a derivation the map
$$\mathrm{Spec}(B)(A\oplus_{d}\Omega M)\rightarrow\widetilde{\mathrm{Spec}(B)}(A\oplus_{d}\Omega M)$$
is an equivalence.
\end{corollary}

\begin{proof}
We have
\begin{align*}
\widetilde{\mathrm{Spec}(B)}(A\oplus_{d}\Omega M) &\cong\colim_{K_{\bullet}\rightarrow\mathrm{Spec}(A)}\mathrm{Spec}(B)(K^{d}_{\bullet})\\
&\cong\colim_{K_{\bullet}\rightarrow\mathrm{Spec}(A)}\mathrm{Spec}(B)(K_{\bullet})\times_{\mathrm{Spec}(B)(K^{M}_{\bullet})}\mathrm{Spec}(B)(K_{\bullet})
\end{align*}
The colimit runs over the category opposite to the category of hypercovers of $\mathrm{Spec}(A)$. This category is filtered, so this colimit is equivalent to
$$(\colim_{K_{\bullet}\rightarrow\mathrm{Spec}(A)}\mathrm{Spec}(B)(K_{\bullet}))\times_{\colim_{K_{\bullet}\rightarrow\mathrm{Spec}(A)}\mathrm{Spec}(B)(K_{\bullet}^{M})}(\colim_{K_{\bullet}\rightarrow\mathrm{Spec}(A)}\mathrm{Spec}(B)(K_{\bullet}))$$
$K_{\bullet}^{M}\rightarrow \mathrm{Spec}(A\oplus M)$ is a hypercover in $\mathbf{A}$. Thus $\mathrm{Spec}(B)(K_{\bullet}^{M})\cong\mathrm{Spec}(B)(A\oplus M)$. Similarly $\mathrm{Spec}(B)(K_{\bullet})\cong\mathrm{Spec}(B)(A)$. This gives
$$\widetilde{\mathrm{Spec}(B)}(A\oplus_{d}\Omega M)\cong\mathrm{Spec}(B)(A)\times_{\mathrm{Spec}(B)(A\oplus M)}\mathrm{Spec}(B)(A)\cong\mathrm{Spec}(B)(A\oplus_{d}\Omega M).$$
\end{proof}

\begin{proposition}\label{prop:formallyetcot}
Let $f:\mathcal{X}\rightarrow\mathcal{Y}$ be a formally \'{e}tale map between stacks in $\mathbf{Stk}(\mathbf{A},\tau|_{\mathbf{A}})$. Suppose that $\mathcal{X}\in\mathbf{Stk}_{n}(\mathbf{A},\tau|_{\mathbf{A}},\textbf{P}|_{\mathbf{A}})$ with atlas.
$$\{U_{i}\rightarrow\mathcal{X}\}_{i\in\mathcal{I}}$$
Let $y:V=\mathrm{Spec}(A)\rightarrow\mathcal{X}$ be a map with $\mathrm{Spec}(A)$ in $\mathbf{A}$. If $f$ has a global contangent complex relative to $\mathbf{A}$ then $\mathsf{L}\Omega^{1}_{\mathcal{X}\big\slash\mathcal{Y},y}\cong0$.
\end{proposition}

\begin{proof}
The proof is formally identical to \cite{toen2008homotopical}*{Lemma 2.2.5.5}. Without loss of generality, we may assume that $\mathcal{Y}$ is affine. If $\mathcal{X}$ is affine   this is the definition of a formally \'{e}tale map. Suppose the claim has been proven for all $m<n$, and let $\mathcal{X}$ be $n$-geometric. Fix an $n$-atlas $\coprod_{i\in\mathcal{I}}U_{i}\rightarrow\mathcal{X}$, and let $x:\mathrm{Spec}(A)\rightarrow\mathcal{X}$ be a map. We may choose a covering $\{\mathrm{Spec}(A_{j})\rightarrow\mathrm{Spec}(A)\}$ such that each $x_{j}:\mathrm{Spec}(A_{j})\rightarrow\mathcal{X}$ factors through some $U_{i(j)}$. Now $A_{j}\otimes_{A}\mathsf{L}\Omega^{1}_{\mathcal{X}\big\slash\mathcal{Y},x} \cong\mathsf{L}\Omega^{1}_{\mathcal{X}\big\slash\mathcal{Y},x_{j}}$. Because $\mathsf{L}\Omega^{1}_{\mathcal{X}\big\slash\mathcal{Y},x}\in\mathbf{F}_{A}$, it suffices to show that each $\mathsf{L}\Omega^{1}_{\mathcal{X}\big\slash\mathcal{Y},x_{j}}\cong 0$. We write $U=U_{i(j)}$ and replace $x$ by $x_{j}$.

We get a fibration sequence of stable $A$-modules.
$$\mathsf{L}\Omega^{1}_{\mathcal{X}\big\slash\mathcal{Y},x}\rightarrow\mathsf{L}\Omega^{1}_{U\big\slash\mathcal{Y},x}\rightarrow\mathsf{L}\Omega^{1}_{U\big\slash\mathcal{X},x}$$
Since $U\rightarrow \mathcal{Y}$ is formally unramified, $\mathsf{L}\Omega^{1}_{U\big\slash\mathcal{Y},x}\cong0$. On the other hand, $\mathsf{L}\Omega^{1}_{U\big\slash\mathcal{X},x}$ is equivalent to $\mathsf{L}\Omega^{1}_{U\times_{\mathcal{X}}V\big\slash V,s}$ where $s:Y\rightarrow U\times_{\mathcal{X}}V$ is induced by the map $V\rightarrow U$. Since $U\times_{\mathcal{X}}V\rightarrow V$ is formally unramified  $\mathsf{L}\Omega^{1}_{U\times_{\mathcal{X}}V\big\slash V,s}\cong 0$. Thus $\mathsf{L}\Omega^{1}_{\mathcal{X}\big\slash\mathcal{Y},x}\cong0$.  
\end{proof}

The next two results can be proven exactly as in \cite{toen2008homotopical}. The first result uses the fact that each $\mathbf{F}_{A}$ is closed under finite limits and colimits.

\begin{proposition}[ \cite{toen2008homotopical}*{Proposition 1.4.1.11}]\label{prop:exactnesscot}
Let $\mathcal{X}$ be an object of $\mathbf{Stk}_{n}(\mathbf{A},\tau|_{\mathbf{A}},\mathbf{P}|_{\mathbf{A}})$. Assume that for any $\mathrm{Spec}(A)\rightarrow\mathcal{X}$ with $\mathrm{Spec}(A)\in\mathbf{A}$, and any $M\in\mathbf{F}^{cn}_{A}$ the map
$$\mathrm{Der}_{\mathcal{X}}(\mathrm{Spec}(A),\Omega\Sigma M)\rightarrow\Omega\mathrm{Der}_{\mathcal{X}}(\mathrm{Spec}(A),\Sigma M)$$
is an isomorphism. Then $\mathcal{X}$ has a global cotangent complex relative to $\mathbf{A}$which is $(-n)$-connective.
\end{proposition}

\begin{corollary}[\cite{toen2008homotopical}*{Proposition 1.4.2.7}]
Let $\mathcal{X}$ be an object of $\mathbf{Stk}(\mathbf{A},\tau|_{\mathbf{A}})$ such that the diagonal $\mathcal{X}\rightarrow\mathcal{X}\times\mathcal{X}$ is $n$-geometric for some $n$. Then $\mathcal{X}$ has a obstruction theory relative to $\mathbf{A}$ if and only if it is inf-cartesian relative to $\mathbf{A}$.
\end{corollary}

%\begin{proof}
%Consider the cartesian square
%\begin{displaymath}
%\xymatrix{
%A\oplus M\ar[d]\ar[r] & A\ar[d]\\
%A\ar[r] & A\oplus \Sigma(M)
%}
%\end{displaymath}
%The assumption that $\mathcal{X}$ is inf-cartesian means that the square
%\begin{displaymath}
%\xymatrix{
%\mathcal{X}(A\oplus M)\ar[d]\ar[r] & \mathcal{X}(A)\ar[d]\\
%\mathcal{X}(A)\ar[r] & \mathcal{X}(A\oplus \Sigma(M))
%}
%\end{displaymath}
%is cartesian. This in turn implies that 
%$$\mathrm{Der}_{\mathcal{X}}(\mathrm{Spec}(A),M)\cong\Omega\mathrm{Der}_{\mathcal{X}}(\mathrm{Spec}(A),\Sigma(M)).$$
%\end{proof}

The proof of the following is identical to \cite{toen2008homotopical}*{Proposition 1.4.2.6}.

%\begin{proposition}[\cite{toen2008homotopical}*{Proposition 1.4.2.6}]\label{prop:fibrecomp}
%Let $f:\mathcal{X}\rightarrow\mathcal{Y}$ be a map in $\mathbf{Stk}(\mathbf{Aff}^{cn}_{\mathbf{C}},\tau)$ which has an obstruction theory relative to $\mathbf{A}$ and $\mathbf{F}$. Let $\mathrm{Spec}(A)\in\mathbf{A}$, $M\in\mathbf{Mod}_{A}$ be $1$-connective, and $d\in\pi_{0}(\mathrm{Der}(A,M))$. Let $x\in\pi_{0}(\mathcal{X}(A\oplus_{d}\Omega M)\rightarrow\pi_{0}(\mathcal{X}(A)\times_{\mathcal{Y}(A\oplus_{d}\Omega M)}\mathcal{Y}(A))$ with projection $y\in\mathcal{X}(A)$, and let $L(x)$ be the  fibre taken at $x$ of the map 
%$$\pi_{0}(\mathcal{X}(A\oplus_{d}\Omega M)\rightarrow\pi_{0}(\mathcal{X}(A)\times_{\mathcal{Y}(A)}\mathcal{Y}(A\oplus_{d}\Omega M))$$
%Then there exists a natural $\alpha(x)\in\mathrm{Map}_{\mathbf{Mod}_{A}}(\mathsf{L}\Omega^{1}_{\mathcal{X}\big\slash\mathcal{Y},y},M)$ and a natural isomorphism 
%$$L(x)\cong\Omega_{\alpha(x),0}\mathrm{Map}_{\mathbf{Mod}_{A}}(\mathsf{L}\Omega^{1}_{\mathcal{X}\big\slash\mathcal{Y},y},M)$$
%where $\Omega_{\alpha(x),0}\mathrm{Map}_{\mathbf{Mod}_{A}}(\mathsf{L}\Omega^{1}_{\mathcal{X}\big\slash\mathcal{Y},y},M)$ is the simplicial set of paths from $\alpha(x)$ to $0$.
%\end{proposition}

\begin{proposition}[\cite{toen2008homotopical}*{Proposition 1.4.2.6}]\label{prop:fibrecomp}
Let $f:\mathcal{X}\rightarrow\mathcal{Y}$ be a map in $\mathbf{PreStk}(\mathbf{Aff}^{cn}_{\mathbf{C}})$ which has a obstruction theory relative to $\mathbf{A}$. Let $\mathrm{Spec}(A)\in\mathbf{A}$, $M\in\mathbf{F}_{A}$ be $1$-connective, and $d\in\pi_{0}(\mathrm{Der}(A,M))$. Let $x\in\pi_{0}(\mathcal{X}(A\oplus_{d}\Omega M)\rightarrow\pi_{0}(\mathcal{X}(A)\times_{\mathcal{Y}(A\oplus_{d}\Omega M)}\mathcal{Y}(A))$ with projection $y\in\mathcal{X}(A)$, and let $L(x)$ be the  fibre taken at $x$ of the map 
$$\pi_{0}(\mathcal{X}(A\oplus_{d}\Omega M)\rightarrow\pi_{0}(\mathcal{X}(A)\times_{\mathcal{Y}(A)}\mathcal{Y}(A\oplus_{d}\Omega M))$$
Then there exists a natural $\alpha(x)\in\mathrm{Map}_{\mathbf{Mod}_{A}}(\mathsf{L}\Omega^{1}_{\mathcal{X}\big\slash\mathcal{Y},y},M)$ and a natural isomorphism 
$$L(x)\cong\Omega_{\alpha(x),0}\mathrm{Map}_{\mathbf{Mod}_{A}}(\mathsf{L}\Omega^{1}_{\mathcal{X}\big\slash\mathcal{Y},y},M)$$
where $\Omega_{\alpha(x),0}\mathrm{Map}_{\mathbf{Mod}_{A}}(\mathsf{L}\Omega^{1}_{\mathcal{X}\big\slash\mathcal{Y},y},M)$ is the simplicial set of paths from $\alpha(x)$ to $0$.
\end{proposition}

The proof of the following is a straightforward modification of \cite{toen2008homotopical}*{Lemma 1.4.2.3 (3)}.

\begin{proposition}
Let $f:\mathcal{X}\rightarrow\mathcal{Y}$ be a map of stacks such that for any map $\mathrm{Spec}(A\oplus_{d}\Omega M)\rightarrow\mathcal{Y}$ with $\mathrm{Spec}(A)$ in $\mathbf{A}$ and $M\in\mathbf{F}_{A}$, the map
$$\mathcal{X}\times_{\mathcal{Y}}\mathrm{Spec}(A\oplus_{d}\Omega M)\rightarrow\mathrm{Spec}(A\oplus_{d}\Omega M)$$
is inf-cartesian relative to $\mathbf{A}$. Then $f$ is inf-cartesian relative to $\mathbf{A}$.
\end{proposition}

\begin{corollary}\label{cor:-1inf}
Any $(-1)$-representable map of stacks is inf-cartesian.
\end{corollary}

The following is essentially \cite{toen2008homotopical} Sub-Lemma 1.4.3.11.

\begin{lemma}
Let $f:\mathcal{Y}\rightarrow\mathcal{X}$ be a geometrically \'{e}tale map with an obstruction theory. Let $\mathrm{Spec}(A)\in\mathbf{A}$ and $M$ be a $1$-connective $A$-module in $\mathbf{F}_{A}$. Let $x\in\pi_{0}(\mathcal{X}(A)\times_{\mathcal{X}(A\oplus M)}\mathcal{X}(A))$ with projection $x_{1}$ onto the first factor. Suppose that the $x_{1}:\mathrm{Spec}(A)\rightarrow\mathcal{X}$ factors through $y_{1}:\mathrm{Spec}(A)\rightarrow\mathcal{Y}$. The point $x$ lifts to $y\in\pi_{0}(\mathcal{Y}(A)\times_{\mathcal{Y}(A\oplus M)}\mathcal{Y}(A))$ which projects to $y_{1}$.
\end{lemma}

\begin{proof}
Consider the commutative diagram
\begin{displaymath}
\xymatrix{
\mathcal{Y}(A)\times_{\mathcal{Y}(A\oplus M)}\mathcal{Y}(A)\ar[d]^{p}\ar[r]^{f} & \mathcal{X}(A)\times_{\mathcal{X}(A\oplus M)}\mathcal{X}(A)\ar[d]^{q}\\
\mathcal{Y}(A)\ar[r] & \mathcal{X}(A).
}
\end{displaymath}

Write $F(p)$ and $F(q)$ for the homotopy fibres of $p$ at $y_{1}$ and $q$ at $x_{1}$ respectively. By commutativity of the diagram there is a natural map $g:F(p)\rightarrow F(q)$, and it suffices to show that the fibre of $g$ is non-empty. This morphism is equivalent to the morphism
$$\Omega_{d,0}\mathrm{Der}_{\mathrm{Spec}(A)}(\mathcal{Y},M)\rightarrow\Omega_{d,0}\mathrm{Der}_{\mathrm{Spec}(A)}(\mathcal{X},M)$$
where $d$ is the derivation given by the image of $y_{1}$. The homotopy fibre is 
$$\Omega_{d,0}\mathrm{Der}_{\mathrm{Spec}(A)}(\mathcal{Y}\big\slash\mathcal{X},M)\cong\Omega_{d,0}\mathrm{Map}(\mathsf{L}\Omega^{1}_{\mathcal{Y}\big\slash\mathcal{X},y_{1}},M)\cong0,$$
where we have used that $\mathcal{Y}\rightarrow\mathcal{X}$ is \'{e}tale.
\end{proof}

\begin{proposition}
Let $f:\mathcal{X}\rightarrow\mathcal{Y}$ be a representable map in $\mathbf{Stk}(\mathbf{A},\tau|_{\mathbf{A}},\mathbf{P}|_{\mathbf{A}})$. Then $f$ has an obstruction theory relative to $\mathbf{A}$.
\end{proposition}

\begin{proof}
This proof is essentially identical to \cite{toen2008homotopical}*{Theorem 1.4.3.2} (It is in fact simpler, since we do not need to appeal to some form of Artin's conditions, as all maps are assumed to be formally \'{e}tale). It suffices to show that any formally \'{e}tale geometric stack is inf-cartesian. The proof is by induction on $n$. The $n=-1$ case follows immediately from Proposition \ref{prop:globalcotinfad} and Corollary \ref{cor:-1inf}.

 Let $\mathrm{Spec}(A)\in\mathbf{A}$, $M\in\mathbf{F}_{A}^{1}$ and $d\in\pi_{0}(\mathrm{Der}(A,M))$ a derivation. Let $x\in\pi_{0}(\mathcal{X}(A)\times_{\mathcal{X}(A\oplus M)}\mathcal{X}(A))$ with projection $x_{1}$ onto the first factor $\pi_{0}(\mathcal{X}(A))$. We need to show that the homotopy fibre
$$\mathcal{X}(A\oplus_{d}\Omega M)\rightarrow\mathcal{X}(A)\times_{\mathcal{X}(A\oplus M)}\mathcal{X}(A)$$
is (non-empty and) contractible. Fix an $n$-atlas $\coprod_{i\in\mathcal{I}}U_{i}\rightarrow\mathcal{X}$ with each map formally \'{e}tale. Since $\mathbf{P}$ consists of formally \'{e}tale morphisms we may in fact  assume that $x_{1}$ lifts to a point $y_{1}:\mathrm{Spec}(A)\rightarrow U$ for some $U=U_{i}$. Moreover since $U\rightarrow\mathcal{X}$ is geometrically \'{e}tale we know that $x$ lifts to a point $y\in\pi_{0}(U(A)\times_{U(A\oplus M)}U(A))$. 

Now consider the commutative diagram
\begin{displaymath}
\xymatrix{
U(A\oplus_{d}\Omega M)\ar[d]\ar[r] & U(A)\times_{U(A\oplus M)}U(A)\ar[d]\\
\mathcal{X}(A\oplus_{d}\Omega M)\ar[r] & \mathcal{X}(A)\times_{\mathcal{X}(A\oplus M)}\mathcal{X}(A)
}
\end{displaymath}
The top map is an equivalence because $U$ is representable. By the inductive hypothesis the square is cartesian. Since we have observed the fibre at $x$ is non-empty, it is also contractible. This completes the proof.
\end{proof}

We arrive at the most important result of this section.

\begin{lemma}[ \cite{toen2008homotopical}, Proposition 2.2.5.4 (2)]\label{lem:geometpullback}
Let $f:\mathcal{X}\rightarrow\mathcal{Y}$ be a geometrically \'{e}tale morphism in $\mathbf{Stk}(\mathbf{Aff}^{cn}_{\mathbf{C}},\tau)$ with an obstruction theory. Then for $\mathrm{Spec}(A)\in\mathbf{A}$, any $1$-connective stable $A$-module $M$ in $\mathbf{F}_{A}^{1}$, and any derivation $d\in\pi_{0}(\mathrm{Der}(A,M))$ the map $A\oplus_{d}\Omega M\rightarrow A$ induces an
equivalence 
$$\mathcal{X}(A\oplus_{d}\Omega M)\rightarrow\mathcal{X}(A)\times_{\mathcal{Y}(A)}\mathcal{Y}(A\oplus_{d}\Omega M).$$
\end{lemma}

\begin{proof}
The proof is as in \cite{toen2008homotopical} Proposition 2.2.5.4. Let $x\in\pi_{0}(\mathcal{Y}(A\oplus_{d}\Omega M)\times_{\mathcal{Y}(A)}\mathcal{X}(A))$ which projects to $y$ in $\pi_{0}(\mathcal{X}(A))$. We need to show that the fibre of this map at $x$ is non-empty and contractible. This follows from Definition \ref{defn:geomet} and Proposition \ref{prop:fibrecomp}.
\end{proof}

\subsection{Deformations of Geometric Stacks}
Let $M$ be a connective object of $\mathbf{F}^{cn}_{k}$. Consider the square-zero extension $k\oplus M$,

% and let $d:k\rightarrow M\in\pi_{0}(\mathrm{Der}(k,M))$.
 For $\mathcal{X}\in\mathbf{PreStk}(\mathbf{Aff}^{cn}_{\mathbf{C}},\tau)$ , denote by $T^{M}\mathcal{X}$ the functor
$$\mathrm{Spec}(A)\mapsto\mathcal{X}(A\oplus (A\otimes M))$$
There is an equivalence
$$T^{M}\mathcal{X}\cong\underline{\mathbf{Map}}(\mathrm{Spec}(k\oplus M),\mathcal{X})$$
This implies the following.

\begin{proposition}
If $\mathcal{X}\in\mathbf{Stk}(\mathbf{Aff}^{cn}_{\mathbf{C}},\tau)$ then so is $T^{d}(\mathcal{X})$. In particular we get a functor
$$T^{M}(-):\mathbf{Stk}(\mathbf{Aff}^{cn}_{\mathbf{C}},\tau)\rightarrow\mathbf{Stk}(\mathbf{Aff}^{cn}_{\mathbf{C}},\tau)$$
\end{proposition}

When $\mathcal{X}$ is representable, then $T^{M}(\mathcal{X})$ is often representable itself, once we impose some finiteness restrictions of the module $M$.

\begin{definition}
An object $M$ in $\mathbf{F}^{cn}_{k}$ is said to be \textit{nuclear} if for any object $L$ in $\mathbf{C}$, the natural map
$$M^{\vee}\otimes L\rightarrow\mathrm{Map}_{\mathbf{C}}(M,L)$$
is an isomorphism.
\end{definition}

\begin{proposition}
Let $\mathcal{X}=\mathrm{Spec}(B)$ be representable and $M\in\mathbf{F}^{cn}_{k}$ be nuclear. Then there is an equivalence, natural in $B$, 
$$T^{M}(\mathrm{Spec}(B))\cong\mathrm{Spec}(\mathrm{Sym}(\mathsf{L}\Omega^{1}_{B}\otimes M^{\vee}))$$
%In particular for $d:k\rightarrow M$ a derivation, we have
%$$T^{d}(\mathrm{Spec}(B))\cong\mathrm{Spec}(B\otimes_{\mathrm{Sym}(\mathsf{L}\Omega^{1}_{B}\otimes M^{\vee})}B).$$
\end{proposition}

\begin{proof}
We have
\begin{align*}
(T^{M}(\mathrm{Spec}(B)))(\mathrm{Spec}(A))&=\mathrm{Map}_{\mathrm{DAlg}^{cn}(\mathbf{C})}(B,A\oplus A\otimes M)\\
&\cong\mathrm{Map}_{\mathsf{DAlg}(\mathbf{C})_{\big\slash\textbf{B}}}(B,B\oplus A\otimes M)\\
&\cong\mathrm{Map}_{\mathbf{Mod}_{B}}(\mathsf{L}\Omega^{1}_{B},|A|\otimes M)\\
&\cong\mathrm{Map}_{\mathbf{Mod}_{B}}(\mathsf{L}\Omega^{1}_{B}\otimes M^{\vee},|A|)\\
&\cong\mathrm{Map}_{\mathsf{DAlg}(\mathbf{C})}(\mathrm{Sym}_{B}(\mathsf{L}\Omega^{1}_{B}\otimes M^{\vee}),A)
\end{align*}
% Now let $d:k\rightarrow M$ be a derivation. The universal property furnishes a map of modules $\phi_{d}:\mathsf{L}\Omega^{1}_{B}\rightarrow B\otimes M$, and therefore a map of unital commutative monoids $\tilde{\phi}_{d}:\mathrm{Sym}_{B}(\mathsf{L}\Omega^{1}_{B}\otimes M^{\vee})\rightarrow B$. In particular for $d=0$ we also get a map $\tilde{\phi}_{0}:\mathrm{Sym}_{B}(\mathsf{L}\Omega^{1}_{B}\otimes M^{\vee})\rightarrow B$. Let $C_{d}$ be given by the pushout diagram
% \begin{displaymath}
% \xymatrix{
% \mathrm{Sym}_{B}(\mathsf{L}\Omega^{1}_{B}\otimes M^{\vee})\ar[d]^{\tilde{\phi}_{d}}\ar[r]^{\;\;\;\;\;\;\;\;\;\;\;\;\;\;\;\;\;\tilde{\phi}_{0}} & B\ar[d]\\
% B\ar[r] & C_{d}
% }
% \end{displaymath}
% It is checked directly that $T^{d}(\mathrm{Spec}(B))\cong\mathrm{Spec}(C_{d})$. 
\end{proof}

The following is an immediate consequence of Lemma \ref{lem:geometpullback}.

\begin{lemma}\label{lem:geometpullback2}
Let $f:\mathcal{X}\rightarrow\mathcal{Y}$ be a geometrically \'{e}tale map between stacks in $\mathbf{Stk}(\mathbf{Aff}^{cn}_{\mathbf{C}},\tau)$ which has an obstruction theory relative to $\mathbf{A}$.Then the following diagram is cartesian
\begin{displaymath}
\xymatrix{
T^{M}\mathcal{X}\ar[d]\ar[r] & T^{M}\mathcal{Y}\ar[d]\\
\mathcal{X}\ar[r] &\mathcal{Y}.
}
\end{displaymath}
\end{lemma}

\begin{corollary}\label{cor:Tepi}
Suppose that $\{f_{i}:\mathcal{X}_{i}\rightarrow\mathcal{Y}\}_{i\in\mathcal{I}}$ be a set of morphisms in $\mathbf{Stk}(\mathbf{A},\tau|_{\mathbf{A}})$ such that each $f_{i}$ is geometrically \'{e}tale and has an obstruction theory relative to $\mathbf{A}$. Consider the induced map $f:\coprod_{i\in\mathcal{I}}\mathcal{X}_{i}\rightarrow\mathcal{Y}$. If $f$ is an epimorphism in $\mathbf{Stk}(\mathbf{A},\tau|_{\mathbf{A}})$ then the map $T^{M}f:\coprod_{i\in\mathcal{I}}T^{M}\mathcal{X}_{i}\rightarrow T^{M}\mathcal{Y}$ is an epimorphism in $\mathbf{Stk}(\mathbf{A},\tau|_{\mathbf{A}})$.
\end{corollary}

\begin{proof}
By Lemma \ref{lem:geometpullback2}, for each $i\in\mathcal{I}$ there is an equivalence
$$T^{M}(\mathcal{X}_{i})\cong T^{M}(\mathcal{Y})\times_{\mathcal{Y}}\mathcal{X}_{i}$$

By assumption $\coprod_{i\in\mathcal{I}}\mathcal{X}_{i}\rightarrow\mathcal{Y}$ is an epimorphism. Since $\mathbf{Stk}(\mathbf{A},\tau|_{\mathbf{A}})$ is a topos,  $T^{M}\mathcal{Y}\times_{\mathcal{Y}}\coprod_{i\in\mathcal{I}}\mathcal{X}\cong\coprod_{i\in\mathcal{I}}T^{M}\mathcal{X}_{i}\rightarrow T^{M}\mathcal{Y}$ is an epimorphism. 
\end{proof}

\begin{corollary}\label{cor:atlastangent}
Let $\mathcal{X}$ be an object of $\mathcal{X}\in\mathbf{Stk}^{\mathbf{A}}_{n}(\mathbf{M},\tau,\textbf{P},\mathbf{A})$ for $n\ge-1$, with atlas $\{U_{i}\rightarrow\mathcal{X}\}_{i\in\mathcal{I}}$. Then $T^{M}\mathcal{X}$ is $n$-geometric, with $n$-atlas $\{T^{M}U_{i}\rightarrow T^{M}\mathcal{X}\}_{i\in\mathcal{I}}$. 
\end{corollary}

\begin{proof}
 By Corollary \ref{cor:Tepi} the map $\coprod_{i\in\mathcal{I}}T^{M}U_{i}\rightarrow T^{M}\mathcal{X}$ is an epimorphism. Moreover $T^{M}U_{i}\rightarrow T^{M}\mathcal{X}$ is $n$-representable and in $\mathbf{P}$ since it is a pullback of the map $U_{i}\rightarrow\mathcal{X}$.
 \end{proof}

\section{Derived Algebraic Geometry Contexts and Geometric HKR}
\label{sec:Geometric_HKR}
In this section we prove the geometric version of HKR, at least for rational derived algebraic contexts, and give examples.

%\begin{definition}\label{def:model_RDAC}
%A \textit{model relative derived algebraic geometry context} is a model derived algebraic context $(\mathpzc{C},\mathpzc{C}_{\ge0},\mathpzc{C}_{\le0},\mathsf{M}^{0})$ together with
%\begin{enumerate}
%\item
%A Grothendieck topology $\tau$ on $\mathpzc{Aff}_{\mathpzc{C}_{\ge0}}$.
%\item
%A class of maps $\mathbf{P}$ in $\mathrm{Ho}(\mathpzc{Aff}_{\mathpzc{C}_{\ge0}})$
%\item
%A full subcategory $\mathpzc{A}\subset\mathpzc{Aff}_{\mathpzc{C}_{\ge0}}$
%\item
%A full subcategory $\mathpzc{F}$ of $\mathpzc{C}$
%\end{enumerate}
%such that 
%$$(\mathpzc{C}_{\ge0},\tau,\mathpzc{P},\mathpzc{A},\mathpzc{F})$$
%is a weak relative HAG context.
%\end{definition}

%Let $(\mathbf{C},\mathbf{C}_{\ge0},\mathbf{C}_{\le0},\mathbf{C}^{0},\tau,\mathbf{P},\mathbf{A})$  be a weak relative derived algebraic geometry context.

\subsection{HKR for Geometric Stacks}

\begin{theorem}
Let 
$$(\mathbf{C},\mathbf{C}_{\ge0},\mathbf{C}_{\le0},\mathbf{C}^{0},\tau,\mathbf{P},\mathbf{A},\mathbf{F})$$ 
 be a weak relative derived algebraic geometry context enriched over $\mathbb{Q}$, and let $\mathcal{X}$ be a scheme for the relative $(\infty,1)$-geometry tuple $(\mathbf{C}_{\ge0},\tau,\mathbf{P},\mathbf{A})$. Then there is an equivalence, natural in $\mathcal{X}$,
$$\mathcal{L}(\mathcal{X})\cong T\mathcal{X}[-1]$$
\end{theorem}

\begin{proof}
We use Proposition \ref{prop:geomequiv}. The fact that $\eta$ is an equivalence when restricted to $0$-geometric stacks is Proposition \ref{prop:ratequiv}. As stacks of the form $\underline{\mathrm{Map}}(\mathcal{Z},-)$, both $\mathcal{L}(-)$ and $T(-)[-1]$ commute with fibre products. Whenever $\{U_{i}\rightarrow\mathcal{X}\}_{i\in\mathcal{I}}$ is an $n$-atlas, then by Corollary \ref{cor:atlasloop} $\{\mathcal{L}(U_{i})\rightarrow\mathcal{L}(\mathcal{X})\}_{i\in\mathcal{I}}$ is an $n$-atlas and by Corollary \ref{cor:atlastangent} $\{TU_{i}[-1]\rightarrow T\mathcal{X}[-1]\}_{i\in\mathcal{I}}$ is an $n$-atlas.
\end{proof}

%\subsection{HKR via Topoi}
%
%\begin{definition}
%Let $\mathcal{X}\in\mathbf{Stk}(\mathbf{A},\tau|_{\mathbf{A}})$. 
%\begin{enumerate}
%\item
%We say that $\mathcal{X}$ is a \textit{hypergeometric stack} if there is hypercover $K_{\bullet}\rightarrow\mathcal{X}$ with each $K_{n}$ a disjoint union of objects of $\mathbf{A}$. 
%\item
%We say that $\mathcal{X}$ is a \textit{hypergeometric scheme} if there is a  hypercover $K_{\bullet}\rightarrow\mathcal{X}$ with each $K_{n}$ a disjoint union of objects of $\mathbf{A}$, and for each $n$ the map
%FIX: FINISH
%\end{enumerate}
%\end{definition}
%
%\begin{definition}
%Let $\mathcal{X}\in\mathbf{Stk}(\mathbf{A},\tau|_{\mathbf{A}})$. A hypercover $K_{\bullet}\rightarrow\mathcal{X}$ with each $K_{n}$ a disjoint union of affines is said to \textit{have an obstruction theory} if for 
%$$K_{n}\cong\coprod_{\mathcal{I}_{n}}\mathrm{Spec}(A_{i_{n}})$$, 
%each map 
%FIX: FINISH
%%$\mathrm{Spec}(A_{i_{n}})\rightarrow\mathrm{Map}(\partial\Delta^{n},K)\times_{\mathrm{Map}(\partial\Delta^{n},\mathcal{X})\times 
%\end{definition}
%
%
%
%\begin{proposition}
%Let $\mathcal{X}\in\mathbf{Stk}(\mathbf{A},\tau|_{\mathbf{A}})$. Suppose that there is a hypercover $K_{\bullet}\rightarrow\mathcal{X}$ such that 
%\begin{enumerate}
%\item
%each $K_{n}$ is a disjoint union of affines, and 
%\item
%the hypercover has an obstruction theory.
%\end{enumerate}
%Let $M\in\mathbf{C}^{cn}$ be nuclear, and $d:k\rightarrow M$ a derivation. Then 
%$$T^{d}(K_{\bullet})\rightarrow T^{d}(\mathcal{X})$$
%is a hypercover. 
%\end{proposition}

\subsection{HKR in (Weak) DAG Contexts}

%Let $(\mathbf{C},\mathbf{C}_{\ge0},\mathbf{C}_{\le0},\mathbf{C}^{0})$ be a derived algebraic context, $\tau$ a topology on $\mathbf{Aff}^{cn}_{\mathbf{C}}$, and $\mathbf{P}$ a class of morphisms in $\mathbf{Aff}^{cn}_{\mathbf{C}}$ containing all isomorphisms, stable by pullback 
\begin{definition}
A \textit{(weak) DAG context} is a tuple $(\mathbf{C},\mathbf{C}_{\ge0},\mathbf{C}_{\le0},\mathbf{C}^{0},\tau,\mathbf{P})$ where 
\begin{enumerate}
\item
$(\mathbf{C},\mathbf{C}_{\ge0},\mathbf{C}_{\le0},\mathbf{C}^{0})$ is a derived algebraic context;
\item
$\tau$ is a topology on $\mathbf{Aff}^{cn}_{\mathbf{C}}$;
\item
$\mathbf{P}$ is a collection of maps in $\mathbf{Aff}^{cn}_{\mathbf{C}}$,
\end{enumerate}
such that $(\mathbf{C},\mathbf{C}_{\ge0},\mathbf{C}_{\le0},\mathbf{C}^{0},\tau,\mathbf{P},\mathrm{DAlg}^{cn}(\mathbf{C}),\mathbf{Mod})$ is a (weak) relative DAG context.
\end{definition}

\begin{remark}
If $\tau$ is a topology on $\mathbf{Aff}^{cn}_{\mathbf{C}}$ and $\mathbf{P}$ a class of maps in $\mathbf{Aff}^{cn}_{\mathbf{C}}$ such that, for example, $(\mathbf{C}_{\ge0},\tau,\mathbf{P})$ is a HAG context in the sense of \cite{toen2008homotopical}, then by assumption one has hyperdescent for modules. Thus in this case $(\mathbf{C},\mathbf{C}_{\ge0},\mathbf{C}_{\le0},\mathbf{C}^{0},\tau,\mathbf{P})$ is a DAG context.
\end{remark}

Let $(\mathbf{C},\mathbf{C}_{\ge0},\mathbf{C}_{\le0},\mathbf{C}^{0},\tau,\mathbf{P})$  be a (weak) DAG context, and let $\mathbf{A}\subset\mathbf{Aff}^{cn}_{\mathbf{C}}$ be a full subcategory such that $(\mathbf{Aff}^{cn}_{\mathbf{C}},\tau,\mathbf{P},\mathbf{A})$ is a relative $(\infty,1)$-geometry tuple. Consider the induced strong relative $(\infty,1)$-geometry tuple $(\mathbf{Aff}_{\mathbf{C}},\tau_{\mathbf{A}},\mathbf{P}_{\mathbf{A}},\mathbf{A})$. Then $(\mathbf{C},\mathbf{C}_{\ge0},\mathbf{C}_{\le0},\mathbf{C}^{0},\tau_{\mathbf{A}},\mathbf{P})$ is itself a (weak) DAG context. Thus we get the following.

\begin{theorem}
Let $(\mathbf{C},\mathbf{C}_{\ge0},\mathbf{C}_{\le0},\mathbf{C}^{0},\tau,\mathbf{P})$ be a weak DAG context with $(\mathbf{C},\mathbf{C}_{\ge0},\mathbf{C}_{\le0},\mathbf{C}^{0})$ being rational, and $\mathbf{A}\subset\mathbf{Aff}^{cn}_{\mathbf{C}}$ a full subcategory such that $(\mathbf{Aff}^{cn}_{\mathbf{C}},\tau,\mathbf{P},\mathbf{A})$ is a relative $(\infty,1)$-geometry tuple. Let $\mathcal{X}\in\mathbf{Sch}(\mathbf{A},\tau|_{\mathbf{A}},\mathbf{P}|_{\mathbf{A}})$. Then there is an equivalence 
$$\mathcal{L}(\mathcal{X})\cong T\mathcal{X}[-1]$$
that is natural in $\mathcal{X}$.
\end{theorem} 
%
%In the To\"en-Vezzosi setup of relative derived algebraic geometry \cite{toen2008homotopical}, one assumes (hyper)descent for modules. In particular if $\mathbf{A}\subset\mathbf{Aff}^{cn}_{\mathbf{C}}$ is any subcategory closed under pullback along maps in $\mathbf{P}|_{\mathbf{A}}$. 
%
%$$(\mathbf{C},\mathbf{C}_{\ge0},\mathbf{C}_{\le0},\mathbf{C}^{0},\tau,\textbf{P},\mathrm{DAlg}^{cn}(\mathbf{C}),\mathbf{Mod})$$
%is a weak relative DAG context. 

\subsubsection{Example: Algebraic Geometry}

Let $R$ be a unital commutative $\mathbb{Q}$-algebra. Consider the derived algebraic context $\mathbf{C}_{R}$. It is presented by the model derived algebraic context  $(\mathrm{Ch}(\mathsf{Mod}_{R}),\mathrm{Ch}_{\ge0}(\mathsf{Mod}_{R}),\mathrm{Ch}_{\le0}(\mathsf{Mod}_{R}),\mathcal{P}^{0})$ where $\mathrm{Ch}(\mathsf{Mod}_{R})$ and $\mathrm{Ch}_{\ge0}(\mathsf{Mod}_{R})$ are equipped with the projective model structures, the $t$-structure is the Postnikov one, and $\mathcal{P}^{0}$ is the category of free, finitely generated $R$-modules. Let $\mathbf{P}^{sm}$ denote the class of smooth morphisms, as defined in \cite{toen2008homotopical}*{Definition 1.2.6.7}. Let $\tau^{\'{e}t}$ denote the topology consisting of covers
$$\{\mathrm{Spec}(B_{i})\rightarrow\mathrm{Spec}(A)\}_{i\in\mathcal{I}}$$
in such that each $\mathrm{Spec}(B_{i})\rightarrow\mathrm{Spec}(A)$ is an \'{e}tale morphism.  Then $(\mathrm{Ch}_{\ge0}(\mathsf{Mod}_{R}),\tau^{\'{e}t},\mathbf{P}^{sm})$ is a HAG context. Thus 
$$(\mathbf{Ch}(\mathsf{Mod}_{R}),\mathbf{Ch}_{\ge0}(\mathsf{Mod}_{R}),\mathbf{Ch}_{\le0}(\mathsf{Mod}_{R}),\mathcal{P}^{0},\tau^{\'{e}t},\mathbf{P}^{\'{e}t},\mathrm{DAlg}^{cn}(\mathbf{C}_{R}))$$ 
is a (weak relative) derived algebraic geometry context. In particular we recover \cite{ben-zvinadler}*{Proposition 1.1}.

\begin{proposition}
Let $R$ be a unital commutative $\mathbb{Q}$-algebra, and let $\mathcal{X}$ be a derived scheme over $R$. Then there is an equivalence
$$T\mathcal{X}[-1]\cong\mathcal{L}(\mathcal{X}),$$ which is natural in $\mathcal{X}$.
\end{proposition}

\subsection{Analytic Geometry}\label{subsec:analytic_geometry}

Here we explain an application to analytic geometry, based on forthcoming work of the first two authors and Oren Ben-Bassat \cite{BKK}.

\subsubsection{The Monomorphism Topology HAG Context}

Let $(\mathbf{C},\mathbf{C}_{\ge0},\mathbf{C}_{\le0},\mathbf{C}^{0})$ be a derived algebraic context. For $A\in\mathrm{DAlg}^{cn}(\mathbf{C})$ say that a collection of maps $\{\mathrm{Spec}(A_{i})\rightarrow\mathrm{Spec}(A)\}_{i\in\mathcal{I}}$ is a \textit{cover in the homotopy monomorphism topology} if there exists a finite subset $\mathcal{J}\subset\mathcal{I}$ such that 
\begin{enumerate}
\item
For each $j\in\mathcal{J}$ the map $\mathrm{Spec}(A_{j})\rightarrow\mathrm{Spec}(A)$ is a monomorphism.
\item
If $g:M\rightarrow N$ is a map in $\mathbf{Mod}_{A}$ such that $A_{j}\otimes_{A}g$ is an equivalence for each $j\in\mathcal{J}$ then $g$ is an equivalence.
\end{enumerate}
This is easily seen to define a topology on $\mathbf{Aff}^{cn}_{\mathbf{C}}$ which we call the monomorphism topology. Moreover, as will be explained in \cite{BKK} this topology satisfies \v{C}ech descent, and in particular weak \v{C}ech descent. Thus 
$$(\mathbf{C},\mathbf{C}_{\ge0},\mathbf{C}_{\le0},\mathbf{C}^{0},\tau|_{\mathbf{A}}^{hm},\mathbf{P}^{\'{e}t})$$
is a relative derived algebraic context.

%\begin{theorem}
%Let $(\mathbf{C},\mathbf{C}_{\ge0},\mathbf{C}_{\le0},\mathbf{C}^{0})$ be a derived algebraic context. Let $\mathcal{X}$ be a scheme for the weak relative derived context
%$$(\mathbf{C},\mathbf{C}_{\ge0},\mathbf{C}_{\le0},\mathbf{C}^{0},\tau^{hm},\mathbf{P}^{\'{e}t},\mathbf{A})$$
%Then there is an equivalence, natural in $\mathcal{X}$,
%$$\mathcal{L}(\mathcal{X})\cong T\mathcal{X}[-1]$$
%\end{theorem}

\subsubsection{Analytic HKR}

Let $R$ be a non-trivially valued Banach field of characteristic $0$. We consider the derived algebraic context
 $$\mathbf{C}^{IB}_{R}\defeq(\mathbf{Ch}(\mathrm{Ind(Ban_{R})}),\mathbf{Ch}_{\ge0}((\mathrm{Ind(Ban_{R})}),\mathbf{Ch}_{\le0}(\mathrm{Ind(Ban_{R})}),\mathcal{P}_{0}).$$ 

For $r=(r_{1},\ldots,r_{n})$ a vector in $\mathbb{R}_{>0}^{n}$, we define $T^{n}_{R}(r)$ to be 
 $$\{\sum_{I\in\mathbb{N}^{n}}a_{I}X^{I}:\sum_{I\rightarrow\infty}|a_{I}|r^{I}|<\infty\}$$
 with the norm $||\sum_{I\in\mathbb{N}^{n}}a_{I}X^{I}||\defeq\sum_{I\in\mathbb{N}^{n}}|a_{I}|r^{I}$ for $R$ Archimedean, and to be 
  $$\{\sum_{I\in\mathbb{N}^{n}}a_{I}X^{I}:\lim_{I\rightarrow\infty}|a_{I}|r^{I}|\rightarrow0\}$$
 with the norm $||\sum_{I\in\mathbb{N}^{n}}a_{I}X^{I}||\defeq\textrm{sup}_{I\in\mathbb{N}^{n}}|a_{I}|r^{I}$ if $R$ is non-Archimedean. For $\rho=(\rho_{1},\ldots,\rho_{n})\in\mathbb{R}^{n}_{>0}$ we define $\mathcal{W}^{n}_{R}(\rho)$ to be the complete bornological $R$-algebra defined by the  colimit
 $$\mathcal{W}^{n}_{R}(\rho)=R<\rho_{1}^{-1}X_{1},\ldots,\rho_{n}^{-1}X_{n}>^{\dagger}\defeq\colim_{r>\rho}T^{n}_{k}(r)$$
 
\begin{definition}[\cite{bambozzi2016dagger} Definition 4.7]
An algebra $A$ in $\mathsf{DAlg}(\mathsf{Ind(Ban_{R})})$ is called a 
\textit{dagger affinoid algebra} if there is an isomorphism 
$$A\cong\mathcal{W}^{n}_{k}(\rho)\big\slash I$$ for some ideal $I\subset\mathcal{W}^{n}_{k}(\rho)$. 
\end{definition}
 
Note that the tensor product of two dagger affinoid algebras is again dagger affinoid. This follows from the fact that $\mathcal{W}^{m}_{R}(\rho)\haotimes_R \mathcal{W}^{n}_{R}(\rho')\cong\mathcal{W}^{m+n}_{R}((\rho,\rho'))$. In particular for any dagger affinoid $A$,
$$A<\rho_{1}^{-1}X_{1},\ldots,\rho_{n}^{-1}X_{n}>\defeq A \haotimes_R \mathcal{W}^{n}_{R}(\rho)$$
is dagger affinoid. 

\begin{definition}[\cite{BaBK} Definition 4.9]
A morphism $f:A\rightarrow B$ of dagger affinoid $R$-algebras is said to be a \textit{Weierstrass localisation} if there is an isomorphism
$$B\cong A<\rho_{1}^{-1}X_{1},\ldots,\rho_{n}^{-1}X_{n}>^{\dagger}\big\slash(X_{1}-f_{1},\ldots,X_{n}-f_{n})$$
for some $f_{1},\ldots,f_{n}\in A$ and $\rho=(\rho_{1},\ldots,\rho_{n})\in\mathbb{R}^{n}_{>0}$ such that under this isomorphism, $A\rightarrow A<\rho_{1}^{-1}X_{1},\ldots,\rho_{n}^{-1}X_{n}>^{\dagger}\big\slash(X_{1}-f_{1},\ldots,X_{n}-f_{n})$ is the obvious map. 
\end{definition}

\begin{remark}
By Lemma 5.1 in \cite{bambozzi2016dagger}, Weierstrass localisations are in particular homotopy epimorphisms.
\end{remark}

For a dagger affinoid algebra $A$, write $\mathcal{M}(A)$ for the set of equivalence classes of bounded algebra morphisms $A\rightarrow\overline{R}$, where $\overline{R}$ is a valued (field) extension of $R$, and equip it with the weak topology. Note that this is a (contravariant) functorial construction.

\begin{definition}[\cite{BaBK} Definition 4.12]
A \textit{dagger Stein algebra} is a bornological algebra of the form $\lim_{n}A_{n}$ where 
$$\ldots A_{n+1}\rightarrow A_{n}\rightarrow\ldots\rightarrow A_{1}\rightarrow A_{0}$$
is a sequence with each $A_{n}$ a dagger affinoid, and each $A_{n+1}\rightarrow A_{n}$ a Weierstrass localisation and $\mathcal{M}(A_{i})$ is contained within the interior of $\mathcal{M}(A_{i+1})$. 
\end{definition}

\begin{remark}\label{rem:alternative_dagger}
In \cite{Meyer-Mukherjee:Bornological_tf}, an alternative, coordinate-free definition of dagger algebras is provided in the non-Archimedean setting. Here, one starts with an arbitrary torsion-free \(\Z_p\)-algebra and completes it in a certain bornology, which encodes for the overconvergence condition implicit in a dagger affinoid algebra. This coincides with the usual notion of a dagger affinoid algebra over \(\Z_p\), when one restricts to finite-type commutative algebras.  
\end{remark}

\begin{definition}
Let $A$ be a dagger Stein algebra. A discrete $A$-module $M$ is said to be \textit{RR-quasicoherent} if for all epimorphisms $A\rightarrow B$ with $B$ dagger Stein, the map $B\otimes_{A}M\rightarrow\pi_{0}(B\otimes_{A}M)$ is an equivalence. 
\end{definition}

\begin{remark}
By Theorems $A$ and $B$ for dagger Steins (\cite{bambozzi2019theorems}), coherent $A$-modules are in particular RR-quasicoherent.
\end{remark}

\begin{definition}
A \textit{derived dagger Stein algebra} is an object $A$ in $\mathrm{DAlg}^{cn}(\mathbf{C}^{IB}_{R})$ such that
\begin{enumerate}
\item
$\pi_{0}(A)$ is a dagger Stein.
\item
$\pi_{n}(A)$ is an $RR$-quasi-coherent $\pi_{0}(A)$-module. 
\end{enumerate}
The full subcategory of $\mathrm{DAlg}^{cn}(\mathbf{C}^{IB}_{R})$ consisting of derived dagger Steins is denoted $\mathbf{dStn}_{alg}^{\dagger}$.
\end{definition}

\begin{definition}
Let $A$ be a derived dagger Stein algebra. An $A$-module $M$ is said to be 
\begin{enumerate}
\item
 \textit{RR-quasicoherent}  if each  $\pi_{i}(M)$ is an $RR$-quasicoherent $\pi_{0}(A)$-module. The full subcategory of $\mathbf{Mod}_{A}$ consisting of such modules is denoted $\mathbf{QCoh}_{A}$. 
 \item
 \textit{coherent} if each $\pi_{i}(M)$ is a coherent $\pi_{0}(A)$-module. The full subcategory of $\mathbf{Mod}_{A}$ consisting of such modules is denoted $\mathbf{Coh}_{A}$. 
 \end{enumerate}
\end{definition}
Note that $\mathbf{Coh}_{A}\subset\mathbf{QCoh}_{A}$.

Using results of \cite{BaBK} (particularly Section 5), in \cite{BKK} we will show the following.
\begin{enumerate}
\item
If $A\rightarrow B$ and $A\rightarrow C$ are maps in $\mathbf{dStn}_{alg}^{\dagger}$ with $A\rightarrow B$ an epimorphism, then $A\otimes_{C}B$ is in $\mathbf{dStn}_{alg}^{\dagger}$.
\item
The homotopy monomorphism topology on $\mathbf{Aff}_{\mathbf{C}^{IB}_{R}}$ restricts to a topology on $\mathbf{dStn}_{alg}^{\dagger}$.
\end{enumerate}

We consider the weak relative derived algebraic geometry context
$$(\mathbf{Ch}(\mathrm{Ind(Ban_{R})}),\mathbf{Ch}_{\ge0}((\mathrm{Ind(Ban_{R})})),\mathbf{Ch}_{\le0}(\mathrm{Ind(Ban_{R})}),\mathcal{P}_{0},\tau^{hm},\mathbf{dStn}_{alg}^{\dagger},\mathbf{P}^{\'{e}t},\mathbf{Mod})$$ 
A scheme in this setting will be called a \textit{locally finitely coverable analytic scheme}.

We get the following
\begin{theorem}\label{thm:analyticHKR}
Let $k$ be either the field of complex numbers with the Euclidean norm, or a non-Archimedean field of characteristic $0$ with a non-trivial valuation. Let $\mathcal{X}$ be a locally finitely coverable derived $k$-analytic scheme. Then there is an equivalence of derived $k$-analytic schemes
$$T\mathcal{X}[-1]\cong\mathcal{L}(\mathcal{X}),$$ that is natural in \(\mathcal{X}\).
\end{theorem}

The category of locally finitely coverable analytic schemes contains, as a full subcategory, the category of locally finitely coverable  (underived) analytic spaces (in the usual sense when $R$ is $\mathbb{C}$, and in the Berkovich sense for $R$ non-Archimedean). An analytic space $X$ is locally finitely coverable if it has an atlas $\{U_{i}\rightarrow X\}_{i\in\mathcal{I}}$ consisting of dagger Stein subspaces such that for any map $V\rightarrow X$ there is a \textit{finite} cover by dagger Steins $\{V_{j}\rightarrow V\}_{j\in\mathcal{J}}$ such that the composition $V_{j}\rightarrow V\rightarrow X$ factors through some $U_{i(j)}$. Any derived analytic scheme which has a finite atlas is in particular locally finitely coverable. This includes any derived analytic scheme $\mathcal{X}$ such that the truncated analytic scheme $t_{0}(\mathcal{X})$ is a second-countable and Hausdorff analytic space. This covers the overwhelmingly vast majority of schemes in which we are interested. We can however extend the HKR theorem to certain non-locally finitely coverable spaces. 

\begin{definition}
 Let $\mathbf{dStn}_{alg}^{\dagger,coh}$ denote the full subcategory of $\mathbf{dStn}_{alg}^{\dagger}$ consisting of derived dagger Steins $A$ such that each $\pi_{n}(A)$ is a \textit{coherent} $\pi_{0}(A)$-module. 
\end{definition}

\begin{definition}
\begin{enumerate}
\item
Let $\mathrm{C}$ be a symmetric monoidal complete and cocomplete category. An object $X$ of $\mathrm{C}$ is said to be \textit{formally }$\aleph_{1}$-\textit{filtered} if for any countable collection $\{Y_{n}\}_{n\in\mathbb{N}}$ the map
$$X\otimes\prod_{n\in\mathbb{N}}Y_{n}\rightarrow\prod_{n\in\mathbb{N}}X\otimes Y_{n}$$
is an isomorphism.
\item
Let $A$ be a derived dagger Stein algebra. A connective $A$-module $M$ is said to be $\aleph_{1}$-RR-\textit{quasi-coherent} if it is RR-quasi-coherent and each $\pi_{n}(M)$ is formally $\aleph_{1}$-filtered in $\mathrm{Mod}_{\pi_{0}(A)}(\mathrm{Ind(\mathrm{Ban}_{R})})$. The full subcategory of $\mathbf{Mod}_{A}$ consisting of $\aleph_{1}$-RR-quasi-coherent modules is denoted $\mathbf{QCoh}^{\aleph_{1}}_{A}$.
\end{enumerate}
\end{definition}

\begin{remark}
One has $\mathbf{Coh}_{A}\subset\mathbf{QCoh}^{\aleph_{1}}_{A}\subset\mathbf{QCoh}_{A}$. 
\end{remark}

In \cite{BKK}, we will define the  countable version $\tau^{hm}_{\aleph_{1}}$  of $\tau^{hm}$. 
whereby a collection of maps $\{\mathrm{Spec}(A_{i})\rightarrow\mathrm{Spec}(A)\}_{i\in\mathcal{I}}$ is a cover precisely if there is a countable subset $\mathcal{J}\subset\mathcal{I}$ such that
\begin{enumerate}
\item
each map $A\rightarrow A_{j}$ with $j\in\mathcal{J}$ is a monomorphism.
\item
If $M$ is a module in $\mathbf{QCoh}_{A}^{\aleph_{1}}$ such that $A_{j}\otimes_{A}M\cong0$ for all $j\in\mathcal{J}$ then $M\cong0$. 
\end{enumerate}
 
We will then show the following.

$$(\mathbf{Ch}(\mathrm{Ind(Ban_{R})}),\mathbf{Ch}_{\ge0}((\mathrm{Ind(Ban_{R})})),\mathbf{Ch}_{\le0}(\mathrm{Ind(Ban_{R})}),\mathcal{P}_{0},\tau_{\aleph_{1}}^{hm},\mathbf{dStn}_{alg}^{\dagger,coh},\mathbf{P}^{\'{e}t},\mathbf{Coh})$$ 
is a weak relative DAG context. A scheme in this context will be called a \textit{locally coherent derived analytic scheme}.

\begin{theorem}\label{thm:analyticHKR2}
Let $k$ be either the field of complex numbers with the Euclidean norm, or a non-Archimedean field of characteristic $0$ with a non-trivial valuation. Let $\mathcal{X}$ be a locally coherent derived $k$-analytic scheme. Then there is an equivalence of derived $k$-analytic schemes
$$T\mathcal{X}[-1]\cong\mathcal{L}(\mathcal{X}).$$
\end{theorem}

Again, as will be explained in \cite{BKK}, the category of ($n$-localic for some $n$) analytic spaces $\mathbf{dAn}^{loc}_{R}$ of Porta and Yue Yu  (\cites{porta2016higher,porta2015derived1, porta2015derived, porta2017derived,porta2018derived, porta2017representability,porta2018derivedhom}), embeds fully faithfully in the category of locally coherent derived analytic schemes. In particular we recover the HKR theorem of Antonio, Petit, Porta in their setup for localic derived analytic spaces, which is \cite{antonio2019derived}*{Theorem 1.2.1.5}. In a future version of this paper, we shall explain how to remove the localic assumption, and recover their full theorem.

\section{A Sketch of Condensed HKR}\label{sec:condensed}

With some modifications which shouldn't be particularly onerous, it is possible to prove a HKR theorem for analytic rings (and analytic spaces), in the sense of Condensed Mathematics of Clausen and Scholze (\cite{clausenscholze1}, \cite{clausenscholze2}). The main hurdle is that the category of analytic rings is not quite a full subcategory of commutative monoids in some symmetric monoidal category. However it is close. Recall (\cite{clausenscholze1}*{Definition 2.1}) that for $\kappa$ an uncountable strong limit cardinal, $\mathcal{C}$ a category with $\kappa$-small limits and colimits, the category of $\kappa$-\textit{condensed} $\mathcal{C}$-\textit{objects}, denoted $\mathrm{Cond}_{\kappa}(\mathcal{C})$, is the category of $\mathpzc{C}$-valued sheaves on the site $*_{\kappa-\mathrm{pro\'{e}t}}$ of $\kappa$-small profinite sets $S$ with covers given by finite families of jointly surjective maps. For $\kappa<\kappa'$ there is an adjunction
$$\adj{LKE_{\kappa<\kappa'}}{\mathrm{Cond}_{\kappa}(\mathcal{C})}{\mathrm{Cond}_{\kappa'}(\mathcal{C})}{U_{\kappa<\kappa'}}$$
where the right adjoint is the forgetful functor and the left is Kan extension (and sheafification). Then one defines
$$\mathrm{Cond}(\mathpzc{C})\defeq\colim_{\kappa}\mathrm{Cond}_{\kappa}(\mathcal{C}).$$
 We need the following, which is essentially the content of Lecture II of \cite{clausenscholze1} 

\begin{theorem}
For each $\kappa$, $\mathrm{Cond}_{\kappa}(\mathrm{Ab})$ is a strongly monoidal elementary abelian category with symmetric projectives.
\end{theorem}

\begin{proof}
That it is abelian is clear, since it is the category of sheaves of abelian groups on a site. The existence of a generating set of tiny projectives is Theorem 2.2 of \cite{clausenscholze1}. The construction goes as follows. For $T$ a $\kappa$-condensed set consider the $\kappa$-condensed abelian group $\mathbb{Z}[T]$, which is the sheafification of the functor which sends $S\in*_{\kappa-\mathrm{pro\'{e}t}}$ to the free abelian group on $\mathbb{Z}[T(S)]$. It is shown in \cite{clausenscholze1} that this is a flat, compact, projective object of $\mathrm{Cond}(\mathrm{Ab})$. Moreover they show that the functor $\mathbb{Z}[-]:\mathrm{Cond}_{\kappa}(\mathrm{Set})\rightarrow\mathrm{Cond}(\mathrm{Ab})$ is symmetric monoidal. It is also a left adjoint. Thus for $T$ and $T'$ we have
$$\mathbb{Z}[T\times T']\cong \mathbb{Z}[T]\otimes\mathbb{Z}[T']$$
so that the tensor product of two projectives is projective, and we have
$$S^{n}(\mathbb{Z}[T])\cong\mathbb{Z}[S^{n}(T)]$$
so projectives are symmetric.
\end{proof}

Thus we have a model derived algebraic context
$$(Ch(\mathrm{Cond}_{\kappa}(\mathrm{Ab})),Ch_{\ge0}(\mathrm{Cond}_{\kappa}(\mathrm{Ab})),Ch_{\le0}(\mathrm{Cond}_{\kappa}(\mathrm{Ab})),\mathcal{P}^{0}),$$
where $\mathcal{P}^{0}$ is the set $\{\mathbb{Z}[S]:S\in*_{\kappa-\mathrm{pro\'{e}t}}\}$, presenting a derived algebraic context $\mathbf{C}_{\mathrm{Cond_{\kappa}(Ab)}}$. By considering condensed $\mathbb{Q}$-modules, we also get a rational model derived algebraic context $\mathbf{C}_{\mathrm{Cond_{\kappa}({}_{\mathbb{Q}}\mathrm{Mod})}}$. 

$Ch(\mathrm{Cond}(\mathrm{Ab}))$ does not have the structure of a derived algebraic context, since there is no small, compact generating set. However since all computations happen in a `level' $Ch(\mathrm{Cond}_{\kappa}(\mathrm{Ab}))$ for some $\kappa$, we can and will treat it as a derived algebraic context.

In condensed mathematics however, one is not interested in the category of condensed rings (or condensed animated rings in the terminology of \cite{clausenscholze2}) \[\mathsf{DAlg}^{cn}(\mathbf{Ch}(\mathrm{Cond}(\mathrm{Ab}))).\] Rather, one is interested in the category of \textit{analytic rings}. 

\begin{definition}[\cite{clausenscholze2}*{Definition 12.1}]
\begin{enumerate}
\item
 A \textit{pre-analytic ring} is a pair $(\underline{R},R[-])$ where $\underline{R}$ is a condensed ring, and $R[-]$ is a functor from the category of extremally disconnected sets to the category $\mathbf{Mod}^{cn}_{\underline{R}}$ which commutes with finite coproducts, together with a natural transformation $S\rightarrow R[S]$.
 \item
 An \textit{analytic ring} is a pre-analytic ring $(\underline{R},R[-])$ such that for all object $C\in\mathbf{Mod}^{cn}_{\underline{R}}$ which is a sifted colimit of objects of the form $R[S]$, the map
  $$\mathbf{Map}_{{}_{\underline{R}}\mathrm{Mod}}(R[S],C)\rightarrow\mathbf{Map}_{{}_{\underline{R}}\mathbf{Mod}}(\underline{R}[S],C)$$
  is an equivalence for all extremally disconnected profinte sets $S$.
\end{enumerate}
\end{definition}
 
 \begin{definition}[\cite{clausenscholze2} Definition 12.3]
Let $(\underline{R},R[-])$ be an analytic ring, and denote by $\mathbf{Mod}^{cn}_{R}\subset\mathbf{Mod}^{cn}_{\underline{R}}$ the full subcategory consisting of $R$-modules such that for any extremally disconnected set $S$, the map
$$\mathrm{Map}_{\mathbf{Mod}_{\underline{R}}}(R[S],M)\rightarrow \mathrm{Map}_{\mathbf{Mod}_{\underline{R}}}(\underline{R}[S],M)$$
is an isomorphism.
\end{definition}

The following is \cite{clausenscholze2}*{Proposition 12.20}.

\begin{proposition}\label{prop:characterisean}
The full subcategory $\mathbf{Mod}^{cn}_{R}\subset\mathbf{Mod}^{cn}_{\underline{R}}$ satisfies the following.
\begin{enumerate}
\item
$\mathbf{Mod}^{cn}_{R}\subset\mathbf{Mod}^{cn}_{\underline{R}}$ is stable under all limits and colimits. 
\item
 the inclusion $\mathbf{Mod}^{cn}_{R}\rightarrow \mathbf{Mod}^{cn}_{\underline{R}}$ admits a left adjoint.
 \item
 for $M\in\mathbf{Mod}^{cn}_{R}$ and $N$ in $ \mathbf{Mod}^{cn}_{\underline{R}}$, $\underline{\mathrm{Map}}_{\mathbf{Mod}^{cn}_{\underline{R}}}(M,N)$ is in $\mathbf{Mod}^{cn}_{R}$. 
 \end{enumerate}
 Conversely, given a full subcategory $\mathbf{M}\subset\mathbf{Mod}^{cn}_{\underline{R}}$ satisfying these properties, there is a unique analytic ring structure $(\underline{R},R)$ such that $\mathbf{M}=\mathbf{Mod}^{cn}_{R}$. 
\end{proposition} 

Note in particular that $\mathbf{Mod}_{R}$ is itself a closed symmetric monoidal category such that the functor $\mathbf{Mod}_{\underline{R}}\rightarrow \mathbf{Mod}_{R}$ is strong monoidal.

\begin{definition}[\cite{clausenscholze1} Lecture VII]
A map $f:(\underline{R},R[-])\rightarrow(\underline{T},T[-])$ is a map $\underline{f}:\underline{R}\rightarrow\underline{T}$ of condensed rings such that  any $M\in\mathbf{Mod}^{cn}_{T}$ is in $\mathbf{Mod}^{cn}_{R}$ when $M$ is regarded as a condensed $\underline{R}$-module via the map $\underline{R}\rightarrow\underline{T}$.
\end{definition}

If $\underline{R}$ is a commutative condensed ring, then to ensure the existence of a derived symmetric algebra monad on $\mathbf{Mod}_{R}^{cn}$ one needs an extra assumption.

\begin{definition}[\cite{clausenscholze2} Definition 12.10]
An \textit{analytic commutative ring} is an analytic associative ring $(\underline{R},R)$ such that $\underline{R}$ is in $\mathrm{DAlg}(\mathbf{Cond}(\mathrm{Ab}))$, and for all primes $p$, the Frobenius map $\phi_{p}:\underline{R}\rightarrow\underline{R}\big\slash^{\mathbb{L}}p$ induces a map of analytic rings
$$(\underline{R},R)\rightarrow(\underline{R}\big\slash^{\mathbb{L}}p,R\big\slash^{\mathbb{L}}p).$$
\end{definition}

In \cite{clausenscholze2} Appendix to Lecture XII, it is shown that this so-called Frobenius condition guarantees that for $(\underline{R},R)$  an analytic commutative ring, the monad $\mathrm{LSym}_{\underline{R}}$ on $\mathbf{Mod}_{\underline{R}}^{cn}$ localises to a monad on $\mathbf{Mod}^{cn}_{R}$.

Let us give an obvious formalisation of this setup, in which we essentially take Proposition \ref{prop:characterisean} as a definition. We begin with a presentable closed symmetric monoidal $(\infty,1)$-category $\mathbf{C}$. 

\begin{definition}
\begin{enumerate}
\item
A \textit{decorated ring} is a triple $(A,\mathbf{M}_{A},l_{A})$ where 
\begin{enumerate}
\item
$A\in\mathsf{DAlg}^{cn}(\mathbf{C})$.
\item
$\mathbf{M}_{A}\subset\mathbf{Mod}^{cn}_{A}(\mathbf{C})$ is a reflective subcategory with reflector $l_{A}$.
\item
$\mathbf{M}_{A}$ is closed under all limits and colimits in $\mathbf{Mod}^{cn}_{A}$.
\item
For any $P\in\mathbf{C}$ and $M\in\mathbf{M}_{A}$, $\underline{\mathrm{Map}}(P,M)\in\mathbf{M}_{A}$.
\end{enumerate}
% is a reflective subcategory, with reflector $l_{A}$  such that for any $M\in\mathbf{M}_{A}$ and $N$ in $\mathbf{Mod}_{A}$, 
%$$\underline{\mathrm{Map}}_{\mathbf{Mod}_{A}}(M,N)$$
%is in $\mathbf{M}_{A}$.
\item
A decorated ring is said to be \textit{normalised} if $A\in\mathbf{M}_{A}$. 
\item
A \textit{map of decorated rings} $f:(A,\mathbf{M}_{A},l_{A})\rightarrow (B,\mathbf{M}_{B},l_{B})$ is a map of rings $\underline{f}:A\rightarrow B$ such that any $M\in\mathbf{M}_{B}$ is in $\mathbf{M}_{A}$ when regarded as an $A$-module via the map $\underline{f}$.
\end{enumerate}
To ease the notational burden, we shall usually suppress the left adjoint $l_{A}$. 
\end{definition}

This implies that $\mathbf{M}_{A}$ is a closed symmetric monoidal category with tensor product $\overline{\otimes}_{A}$ such that the functor $\mathbf{Mod}^{cn}_{A}\rightarrow\mathbf{M}^{cn}_{A}$ is strong monoidal. The internal hom is the restriction of the one on $\mathbf{Mod}_{A}$. The forgetful functor
$$\mathbf{M}_{B}\rightarrow\mathbf{M}_{A}$$
also has a left adjoint which is the composition 
\begin{displaymath}
\xymatrix{
\mathbf{M}_{A}\ar[r] & \mathbf{Mod}^{cn}_{A}\ar[r]^{B\otimes_{A}(-)} & \mathbf{Mod}^{cn}_{B}\ar[r]^{l_{B}} & \mathbf{M}_{B}
}
\end{displaymath}
The full subcategory of $\mathsf{DAlg}^{cn}(\mathbf{C})$ consisting of decorated rings is denoted $\mathsf{DAlg}^{dec}(\mathbf{C})$.  If $A\in\mathsf{DAlg}(\mathbf{C})$, then $(A,\mathbf{Mod}^{cn}_{A},Id_{A})$ is an object of $\mathsf{DAlg}^{dec}(\mathbf{C})$. This gives a fully faithful functor $\mathsf{DAlg}^{cn}(\mathbf{C})\rightarrow\mathsf{DAlg}^{dec}(\mathbf{C})$. This functor has a right adjoint which sends $(B,\mathbf{M}_{B},l_{B})$ to $B$. More generally there is a functor.
$$\mathsf{DAlg}(\mathbf{M}_{A})\rightarrow{}_{(A,\mathbf{M}_{A})\big\backslash}\mathsf{DAlg}^{dec}(\mathbf{C})$$
which sends $B$ to $(B,\mathbf{Mod}_{B}(\mathbf{M}_{A}))$. This again has an obvious right adjoint.

\begin{definition}
A decorated ring $(A,\mathbf{M}_{A},l_{A})$ is said to be \textit{admissible} if the monad
$$\mathrm{LSym}_{A}:\mathbf{Mod}^{cn}_{A}\rightarrow\mathbf{Mod}^{cn}_{A}$$
descends to a monad 
$$\mathrm{LSym}_{\mathbf{M}_{A}}:\mathbf{M}_{A}\rightarrow\mathbf{M}_{A}$$
such that the category of $\mathrm{LSym}_{\mathbf{M}_{A}}$-algebras is equivalent to the category of $\mathrm{LSym}_{A}$-algebras whose underlying module is in $\mathbf{M}_{A}$. 
\end{definition}
%
%Let $(A,\mathbf{M}_{A})\in\mathsf{DAlg}^{dec}(\mathbf{C})$ be admissible. 

Let $\mathsf{DAlg}^{dec}(\mathbf{C})$ denote the category of decorated rings, $\mathsf{DAlg}^{dec,n}(\mathbf{C})$ the category of normalised decorated rings, $\mathsf{DAlg}^{adm}(\mathbf{C})$ of admissible decorated rings, and $\mathsf{DAlg}^{adm,n}(\mathbf{C})$ the category of normalised admissible decorated rings. As explained in \cite{clausenscholze2} (Appendix to Lecture XII), any admissible decorated ring $(A,\mathbf{M}_{A})$ has a normalisation $(\mathrm{LSym}_{\mathbf{M}_{A}}(0),\mathbf{M}_{A})$. This means that the inclusion
$$\mathsf{DAlg}^{adm,n}(\mathbf{C})\rightarrow \mathsf{DAlg}^{adm}(\mathbf{C})$$
has a left adjoint.

% We write
%$$\mathbf{Aff}^{adm,n}_{\mathbf{C}}\defeq(\mathsf{DAlg}^{adm,n}(\mathbf{C}))^{op}$$
%for the category of decorated affine stacks.

\begin{proposition}[\cite{clausenscholze2}*{Proposition 12.12}]
\begin{enumerate}
\item
The initial object is $(k,\mathbf{C})$.
\item
$\mathsf{DAlg}^{dec}(\mathbf{C})$ has pushouts. If $(B,\mathbf{M}_{B})\leftarrow(A,\mathbf{M}_{A})\rightarrow(C,\mathbf{M}_{C})$ is a diagram then the pushouts is given by $(B\otimes_{A}C,\mathbf{M}_{B\otimes_{A}C})$, where a module $M$ is in $\mathbf{M}_{B\otimes_{A}C}$ precisely if it is in $\mathbf{M}_{A}$ when regarded as an $A$-module, and $\mathbf{M}_{B}$ when regarded as a $B$-module. 
\end{enumerate}
\end{proposition}

\begin{definition}
A full subcategory $\textbf{R}\subset\mathsf{DAlg}^{adm,n}(\mathbf{C})$ is said to be \textit{geometric} if the following conditions hold.
\begin{enumerate}
\item
If $A\in\mathsf{DAlg}^{cn}(\mathbf{C})$ then $(A,\mathbf{Mod}_{A}^{cn}(\mathbf{C}))\in\textbf{R}$.
\item
If $(A,\mathbf{M}_{A})\in\textbf{R}$ and $B\in\mathbf{M}_{A}\cap\mathsf{DAlg}^{cn}(\mathbf{C})$ then the normalisation of $(B,\mathbf{Mod}_{B}(\mathbf{M}_{A}))$ is in $\mathbf{R}$.
\item
If $(B,\mathbf{M}_{B})\leftarrow(A,\mathbf{M}_{A})\rightarrow(C,\mathbf{M}_{C})$ is a diagram in $\textbf{R}$, then $(B,\mathbf{M}_{B})\otimes_{(A,\mathbf{M}_{A})}(C,\mathbf{M}_{C})$ is admissible, and its normalisation is in $\textbf{R}$. 
\end{enumerate}
\end{definition}

Let $(B,\mathbf{M}_{B})$ be an object in $\textbf{R}$ and $M\in\mathbf{M}_{B}$. We define the square-zero extension ring in $\textbf{R}$ to be $(B\oplus M,\mathbf{Mod}_{B\oplus M}(\mathbf{M}_{B}))$. Let $f:(A,\mathbf{M}_{A})\rightarrow(B,\mathbf{M}_{B})$ be a map in $\textbf{R}$. 
Define 
$$\mathrm{Der}_{(A,\mathbf{M}_{A})}((B,\mathbf{M}_{B}),-):\mathbf{M}_{B},\rightarrow\textbf{sSet}$$
 to be the functor
$$\mathrm{Map}_{(A,\mathbf{M}_{A})\big\backslash\textbf{R}\big\slash(B,\mathbf{M}_{B})}((B,\mathbf{M}_{B}),B\oplus(-)\otimes_{B}\mathbf{M}_{B}).$$

\begin{lemma}
The functor $\mathrm{Der}_{(A,\mathbf{M}_{A})}((B,\mathbf{M}_{B}),-)$ is representable by an object $\mathsf{L}\Omega^{1}_{(B,\mathbf{M}_{B})\big\slash (A,\mathbf{M}_{A})}$ of $\mathbf{M}_{B}$.
\end{lemma}

\begin{proof}[Sketch proof]
First assume that $(A,\mathbf{M}_{A})=(k,\mathbf{C})$ is the initial decorated ring. Let $(\mathsf{L}\Omega^{1})^{\mathbf{C}}_{B}$ denote the contangent complex of $B$ computed in $\mathbf{Mod}_{B}(\mathbf{C})$. Set $\mathsf{L}\Omega^{1}_{B}\defeq l_{B}((\mathsf{L}\Omega^{1})^{\mathbf{C}}_{B})$. It is easy to see that this satisfies the necessary universal property. For arbitrary morphisms $f:(A,\mathbf{M}_{A})\rightarrow(B,\mathbf{M}_{B})$, we define $\mathsf{L}\Omega^{1}_{(B,\mathbf{M}_{B})\big\slash (A,\mathbf{M}_{A})}$ to be the cofibre in $\mathbf{M}_{B}$ of the map
$$\mathsf{L}\Omega^{1}_{(A,\mathbf{M}_{A})}\otimes_{(A,\mathbf{M}_{A})}(B,\mathbf{M}_{B})\rightarrow\mathsf{L}\Omega^{1}_{(B,\mathbf{M}_{B})}.$$ \qedhere
\end{proof}

We shall write 
$$\mathbf{Aff}(\textbf{R})\defeq\textbf{R}^{op}$$
As before for $(A,\mathbf{M}_{A},l_{A})\in\textbf{R}$ we write $\mathrm{Spec}(A,\mathbf{M}_{A},l_{A})$ for the corresponding object of $\mathbf{Aff}(\textbf{R})$. For $M$ a $1$-connective object of $\mathbf{C}$ and $d:k\rightarrow M$ a derivation, we define a functor
$$T^{d}:\mathbf{PreStk}(\mathbf{Aff}(\textbf{R}))\rightarrow\mathbf{PreStk}(\mathbf{Aff}(\textbf{R}))$$
by
$$T^{d}\mathcal{X}(\mathrm{Spec}(B,\mathbf{M}_{B}))\defeq\mathcal{X}(B\oplus_{d} \Omega(B\otimes M),\mathbf{Mod}_{B\oplus_{d} \Omega(B\otimes M)}(\mathbf{M}_{B})).$$

We easily get the following.

\begin{proposition}
Let $M\in\mathbf{C}$ be nuclear and $1$-connective. Then $T^{d}(\mathrm{Spec}(A,\mathbf{M}_{A},l_{A}))$ is representable by $\mathrm{Spec}$ of the normalisation of
$$(C_{d},\mathbf{Mod}_{C_{d}}(M^{\vee}\otimes\mathsf{L}\Omega^{1}_{A})(\mathbf{M}_{A})$$
\end{proposition}

Let us now compute the decorated Hochschild object of $(A,\mathbf{M}_{A})\in\textbf{R}$. Write
$$(A\otimes_{A\otimes A}A,\mathbf{M}_{A\otimes_{A\otimes A}A})\defeq (A,\mathbf{M}_{A})\otimes_{(A,\mathbf{M}_{A})\otimes(A,\mathbf{M}_{A})}(A,\mathbf{M}_{A})$$
where the coproduct is computed in $\mathsf{DAlg}^{dec}(\mathbf{C})$. An $A\otimes_{A\otimes A}A$-module is in $\mathbf{M}_{A\otimes_{A\otimes A}A}$ precisely if the restriction along the map $A\rightarrow A\otimes_{A\otimes A}A$ is in $\mathbf{M}_{A}$. Thus 
$$\mathbf{M}_{A\otimes_{A\otimes A}A}\cong\mathbf{Mod}_{A\otimes_{A\otimes A}A}(\mathbf{M}_{A})$$

\begin{proposition}
Let $(A,\mathbf{M}_{A})\in\textbf{R}$. Then the coproduct
$$(A,\mathbf{M}_{A})\otimes_{(A,\mathbf{M}_{A})\otimes(A,\mathbf{M}_{A})}(A,\mathbf{M}_{A})$$
in $\textbf{R}$ is the normalisation of $(A\otimes_{A\otimes A}A,\mathbf{Mod}_{A\otimes_{A\otimes A}A}(\mathbf{M}_{A}))$. 
\end{proposition}

\begin{theorem}[Decorated HKR]
Let $(\mathbf{C},\mathbf{C}_{\ge0},\mathbf{C}_{\le0},\mathbf{C}^{0})$ be a rational derived algebraic context and $\textbf{R}\subset\mathsf{DAlg}^{dec}(\mathbf{C}_{\ge0})$ be a geometric category of decorated rings. Then for any $(A,\mathbf{M}_{A})\in\textbf{R}$ there is an isomorphism
$$(\mathrm{Sym}_{\mathbf{M}_{A}}(\mathsf{L}\Omega^{1}_{A}),\mathbf{Mod}_{\mathrm{Sym}_{\mathbf{M}_{A}}(\mathsf{L}\Omega^{1}_{A})}(\mathbf{M}_{A}))\cong(A,\mathbf{M}_{A})\otimes_{(A,\mathbf{M}_{A})\otimes(A,\mathbf{M}_{A})}(A,\mathbf{M}_{A}).$$
\end{theorem}

There is almost certainly a version of this over $\mathbb{Z}$ as well, in which one considers decorated rings in $\mathbf{Filt}(\mathbf{C})$ and $\mathbf{Gr}(\mathbf{C})$.

\section{Further Directions}

\subsection{Geometric HKR Over $\mathbb{Z}$}

Let $(\mathbf{C},\mathbf{C}_{\ge0},\mathbf{C}_{\le0},\mathbf{C}^{0},\tau,\mathbf{P},\mathbf{A})$ be a weak relative DAG context. Let us recall the construction in \cite{toen-robalo} 6.4.1 of the geometric filtered circle.

Define $\mathbb{A}_{alg}^{1}\defeq\mathrm{Spec}(\mathrm{Sym}(k))$ where $k$ is the monoidal unit. Let $\mathbb{G}_{m}$ be the multiplicative affine group scheme. It can be constructed as $\mathrm{Spec}$ of the group algebra in $\mathbf{C}_{\ge0}$ on the cyclic group $\mathbb{Z}$. Consider the quotient stacks
$$B\mathbb{G}_{m}\defeq[\mathrm{Spec}(k)\big\slash\mathbb{G}_{m}] \quad \text{ and } \quad  [\mathbb{A}^{1}\big\slash\mathbb{G}_{m}].$$

The augmentation $\mathrm{Sym}(k)\rightarrow k$ determines a point $0:\mathrm{Spec}(k)\rightarrow\mathbb{A}^{1}$ which is a map of $\mathbb{G}_{m}$-stacks when $\mathrm{Spec}(k)$ is equipped with the trivial action. We therefore get a map
$$B\mathbb{G}_{m}\rightarrow[\mathbb{A}^{1}\big\slash\mathbb{G}_{m}].$$

\begin{definition}[\cite{toen-robalo}*{Definition 2.2.11}]
A \textit{graded stack} is a stack over $B\mathbb{G}_{m}$. A \textit{filtered stack} is a stack over $[\mathbb{A}^{1}\big\slash\mathbb{G}_{m}]$. 
\end{definition}

Write
$$\textbf{Gr}(\mathbf{Stk})\defeq\mathbf{Stk}_{\big\slash B\mathbb{G}_{m}} \quad \text{ and } \quad \textbf{Filt}(\mathbf{Stk})\defeq\mathbf{Stk}_{\big\slash [\mathbb{A}^{1}\big\slash\mathbb{G}_{m}]}$$ for the categories of graded and filtered stacks respectively. These are Cartesian closed. As explained in \cite{toen-robalo}*{Definition 2.2.11}, there are obvious functors
$\textbf{Gr}(\mathbf{Stk})\rightarrow\mathbf{Stk}$ and $\textbf{Filt}(\mathbf{Stk})\defeq\mathbf{Stk}_{\big\slash [\mathbb{A}^{1}\big\slash\mathbb{G}_{m}]}\rightarrow\mathbf{Stk}$ which are called the underlying stack functors. There is also an \textit{associated graded stack} functor 
$$\textbf{Filt}(\mathbf{Stk})\rightarrow\textbf{Gr}(\mathbf{Stk})$$
given by sending $\mathcal{X}\rightarrow [\mathbb{A}^{1}\big\slash\mathbb{G}_{m}]$ to the fibre product
$$\mathcal{X}^{gr}\defeq \mathcal{X}\times_{[\mathbb{A}^{1}\big\slash\mathbb{G}_{m}]}B\mathbb{G}_{m}\rightarrow B\mathbb{G}_{m}.$$
If $A$ is a filtered commutative monoid, then $\mathrm{Spec}(A)$ is naturally a filtered commutative stack. Thus we get a functor
$$\mathrm{Spec}_{\mathrm{Filt}}:\mathsf{DAlg}(\textbf{Filt}(\mathbf{C}))^{op}\rightarrow\textbf{Filt}_{\mathbf{Stk}}.$$
This functor has a left adjoint (\cite{toen-robalo}*{3.2})
$$\mathbf{Aff}_{\mathrm{Filt}}:\textbf{Filt}_{\mathbf{Stk}}\rightarrow\mathsf{DAlg}(\textbf{Filt}(\mathbf{C}))^{op}.$$
Define functors
$$\mathrm{triv}_{gr}:\mathbf{Stk}\rightarrow\textbf{Gr}(\mathbf{Stk}),\;\;\mathcal{X}\mapsto(\mathcal{X}\times B\mathbb{G}_{m}\rightarrow B\mathbb{G}_{m})$$
$$\mathrm{triv}_{filt}:\mathbf{Stk}\rightarrow\textbf{Filt}(\mathbf{Stk}),\;\;\mathcal{X}\mapsto(\mathcal{X}\times [\mathbb{A}_{1}\big\slash\mathbb{G}_{m}]\rightarrow [\mathbb{A}_{1}\big\slash\mathbb{G}_{m}]).$$

In \cite{toen-robalo}, he authors construct the geometric filtered circle as follows. Let $\mathcal{R}\subset\mathbb{Q}[x]$ denote the subring consisting of integer valued polynomials, i.e. those $p(x)$ such that $p(z)\in\mathbb{Z}$ for all $z\in\mathbb{Z}$. This is a filtered commutative and cocommutative Hopf algebra where the filtration is by the degree of the polynomials. Consider the Rees construction $\mathrm{Rees}(\mathcal{R})$ of this filtered Hopf algebra, which is itself a filtered Hopf algebra. Define $H_{\mathbb{Z}}\defeq\mathrm{Spec}(\mathrm{Rees}(\mathcal{R}))$ and the filtered circle to be the filtered stack. 
$$S^{1}_{\mathrm{Fil},\mathbb{Z}}\defeq BH_{\mathbb{Z}}.$$
The underlying stack of $S^{1}_{\mathrm{Fil},\mathbb{Z}}$ is $S^{1}$.
Then $\mathbf{Aff}_{\mathrm{Filt}}(S^{1}_{\mathrm{Fil},\mathbb{Z}})$ is the $\mathbb{Z}$-linear filtered circle \(\Z[S^1]_{\mathrm{fil}}\). (This works in $\mathbf{C}_{ab}$ and by base-change works in any derived algebraic context $\mathbf{C}$). The general HKR theorem of \cite{toen-robalo} can be formulated as follows.

\begin{theorem}[\cite{toen-robalo}*{Theorem 5.1.3}]
Let $X\cong\mathrm{Spec}(A)$ be affine. The natural map
$$(\underline{\mathrm{Map}}_{\textbf{Filt}(\mathbf{Stk})}(S^{1}_{filt,\mathbb{Z}},\mathrm{triv}_{filt}(X)))^{gr}\rightarrow\underline{\mathrm{Map}}_{\textbf{Gr}(\mathbf{Stk})}(\mathrm{Spec}(k\oplus k[-1]),\mathrm{triv}_{gr}(X))$$
where $\mathrm{Spec}(k\oplus k[-1])$ is equipped with the brutal grading.
\end{theorem}

As we did for the rational case in this paper, in future work we will investigate how to extend this theorem to schemes. 

%\begin{conjecture}
%Let $(\mathbf{C},\mathbf{C}_{\ge0},\mathbf{C}_{\le0},\mathbf{C}^{0},\tau,\mathbf{P},\mathbf{A})$ be a weak relative DAG context, and let $\mathcal{X}$ be an $\mathbf{A}$-scheme for the relative $(\infty,1)$-geometry tuple $(\mathbf{Aff}_{\mathbf{C}_{\ge0}},\tau,\mathbf{P},\mathbf{A})$. Then there is an equivalence, natural in $\mathcal{X}$
%$$(\underline{\mathrm{Map}}_{\textbf{Filt}(\mathbf{Stk})}(S^{1}_{filt,\mathbb{Z}},\mathrm{triv}_{filt}(\mathcal{X})))^{gr}\rightarrow\underline{\mathrm{Map}}_{\textbf{Gr}(\mathbf{Stk})}(\mathrm{Spec}(k\oplus k[-1]),\mathrm{triv}_{gr}(\mathcal{X}))$$
%\end{conjecture}
%

\subsection{HKR for Smooth Stacks}

Our proof of the geometric HKR theorem uses in a crucial way that the spaces we are considering are schemes. Indeed if in an atlas $\{U_{i}\rightarrow\mathcal{X}\}_{i\in\mathcal{I}}$ the map $U_{i}\rightarrow\mathcal{X}$ is not a monomorphism, then there is no reason why a loop in $\mathcal{X}$ should lift to a loop in some $U_{i}$. Thus $\{\mathcal{L}U_{i}\rightarrow\mathcal{L}\mathcal{X}\}_{i\in\mathcal{I}}$ will not in general be an atlas for $\mathcal{L}\mathcal{X}$. In the algebraic setting the remedy of Ben-Zvi and Nadler \cite{ben-zvinadler} is to instead consider the \textit{formal loop space}, which is the formal completion of $\mathcal{L}\mathcal{X}$ at the substack of constant loops (which is isomorphic to $\mathcal{X}$),
$$\hat{\mathcal{L}}\mathcal{X}\defeq\reallywidehat{\mathcal{L}\mathcal{X}_{\mathcal{X}}}.$$

On the shifted tangent side, one considers the formal completion of the shifted tangent complex along the zero section
$$\hat{\mathbb{T}}\mathcal{X}[-1]\defeq\reallywidehat{\mathbb{T}\mathcal{X}[-1]}_{\mathcal{X}}.$$ Theorem 1.23 in \cite{ben-zvinadler} states that the natural map
$$\hat{\mathcal{L}}\mathcal{X}\rightarrow\hat{\mathbb{T}}\mathcal{X}[-1]$$
is an equivalence. We expect that a similar result holds for general relative derived algebraic geometry contexts. However there are some obstacles. Firstly an appropriate definition of smooth map would need to be formulated, such that one has obstruction theories for smooth stacks (possibly via some version of Artin's conditions \cite{toen2008homotopical}*{1.4.3}). One also needs to clarify in this greater generality what it means for an object $R$ of $\mathsf{CAlg}(\mathbf{C}^{\heart})$ to be reduced, but this may be as simple as saying that the underlying ring
$$\mathrm{Hom}_{\mathbf{C}^{\heart}}(k,R)$$
is reduced in the usual sense, where $k$ is the monoidal unit of $\mathbf{C}^{\heart}$.

%\subsection{Shifted Symplectic Structures}

%The rational HKR theorem is crucial in the 

%https://mathoverflow.net/questions/58988/an-example-of-a-complex-manifold-without-a-finite-open-cover

\end{document}